\documentclass[11pt,reqno]{amsart}
\usepackage{amsmath}
\usepackage{mathpazo}
\usepackage[margin=1in]{geometry}
\usepackage{color}
\usepackage{tikz-cd}
\usetikzlibrary{cd}
\usepackage{hyperref}
\usepackage{url}
\definecolor{darkred}{rgb}{1,0,0}
\definecolor{darkgreen}{rgb}{0,1,0}
\definecolor{darkblue}{rgb}{0,0,1}

\usepackage{graphicx,txfonts}

\usepackage{amsthm}

\hypersetup{colorlinks,
linkcolor=darkblue,
filecolor=darkblue,
urlcolor=black,
citecolor=darkblue}

\renewcommand{\H}{\mathbf{H}^{p,q}}

\newcommand{\bH}{\partial_\infty \mathbf{H}^{p,q}}
\newcommand{\Ads}{\mathbf{H}^{2,1}}
\newcommand{\bAds}{\partial_\infty\mathbf{H}^{2,1}}
\newcommand{\Htwo}{\mathbf{H}^{2}}

\newcommand{\GH}{\mathbf G(\H)}
\newcommand{\GHH}{\mathbf G(\H_+)}
\newcommand{\Gr}{\mathbf{G}(E)}
\renewcommand{\P}{\mathbf P}
\newcommand{\g}{\mathbf g}

\newcommand{\B}{\mathbf B}
\newcommand{\Ball}{\mathbf{B}}
\newcommand{\seqn}[1]{\{ #1 \}_{n\in \mathbf N}}
\newcommand{\seqm}[1]{\{ #1 \}_{m\in \mathbf N}}
\newcommand{\Hom}{\text{Hom}}
\renewcommand{\S}{\mathbf S}
\newcommand{\T}{\mathsf{T}}
\newcommand{\N}{\mathsf{N}}
\newcommand{\G}{\mathsf{G}}
\newcommand{\K}{\mathsf{K}}
\newcommand{\I}{\mathrm I}
\newcommand{\II}{\mathrm{II}}
\newcommand{\q}{\mathbf{b}}
\newcommand{\SO}{\mathsf{O}(\q)}
\newcommand{\PO}{\mathsf{PO}(\q)}
\newcommand{\Stab}{\text{Stab}}
\renewcommand{\span}{\text{span}}

\newcommand{\Bd}{\mathcal B}
\newcommand{\M}{\mathcal M}
\newcommand{\Quad}{\mathbf Q^{p,q}}
\newcommand{\tr}{\text{tr}}
\newcommand{\Conv}{{\text{\rm Conv}}}
\newcommand{\Id}{\mathrm{Id}}
\newcommand{\Ker}{\text{Ker}}
\newcommand{\Lip}{\text{Lip}}
\newcommand{\V}{\mathbf{V}}
\newcommand{\Isot}{\mathbf{Isot}}

\newcommand{\Upp}{\mathbf{S}_+}
\newcommand{\oprad}{{\text{\rm rad}}}
\newcommand{\opExp}{{\text{\rm Exp}}}
\newcommand{\opArccosh}{{\text{\rm arccosh}}}
\newcommand{\opSin}{{\text{\rm sin}}}
\newcommand{\opArcSin}{{\text{\rm arcsin}}}

\catcode`\@=11
\def\eqalign#1{\null\,\vcenter{\openup1\jot \m@th %
\ialign{\strut\hfil$\displaystyle{{}##}$&$\displaystyle{{}##}$\hfil\crcr#1\crcr}}\,}
\def\triplealign#1{\null\,\vcenter{\openup1\jot \m@th %
\ialign{\strut\hfil$\displaystyle{##}$&$\displaystyle{{}##}$\hfil&$\displaystyle{{}##}$\hfil\crcr#1\crcr}}\,}
\def\multiline#1{\null\,\vcenter{\openup1\jot \m@th %
\ialign{\strut$\displaystyle{##}$\hfil&$\displaystyle{{}##}$\hfil\crcr#1\crcr}}\,}
\catcode`\@=12
\newcommand{\opConv}{{\text{\rm Conv}}}
\newcommand{\opHom}{{\text{\rm Hom}}}

\newcommand{\opIm}{{\text{\rm Im}}}
\newcommand{\opId}{{\text{\rm Id}}}

\newcommand{\opdVol}{{\text{\rm dVol}}}
\newcommand{\opVol}{{\text{\rm Vol}}}

\newcommand{\oploc}{{\text{\rm loc}}}

\newcommand{\POpqp}{\mathsf{PO}(p,q+1)}

\newcommand{\Rpqp}{\mathbf{R}^{p,q+1}}

\newcommand{\Rpq}{\textbf{R}^{p,q}}
\newcommand{\Rp}{\textbf{R}^p}
\newcommand{\Hp}{\mathbf{H}^p}

\newcommand{\Hqppn}{\mathbf{H}^{q+1,p-1}}
\newcommand{\Hthree}{\textbf{H}^3}
\newcommand{\Hpone}{\textbf{H}^{p,1}}
\newcommand{\Htwoq}{\textbf{H}^{2,q}}

\newcommand{\R}{\mathbf{R}}
\newcommand{\Exp}{\text{\rm Exp}}

\newcommand{\opL}{{\text{L}}}

\newcommand{\opCos}{{\text{\rm cos}}}
\newcommand{\opDet}{{\text{\rm Det}}}
\newcommand{\AdS}{{\text{\rm AdS}}}

\newcommand{\ophatdVol}{\widehat{\opdVol}{}}
\newcommand{\opequiv}{{\text{\rm equiv}}}
\def\myitem#1{\leavevmode\newline\noindent\hbox to .5cm{\hfill#1\hss}}

\newcommand{\opren}{{\text{\rm ren}}}
\newcommand{\oppr}{{\text{\rm pr}}}
\newcommand{\opDiam}{{\text{\rm Diam}}}

\def\newsubhead#1[#2]{\subsection{#1}\label{subhead:#2}}

\def\myeqnum#1{#1}
\def\nexteqnno[#1]{\label{eqn:#1}}
\def\eqnref#1{(\ref{eqn:#1})}
\def\proclabel#1{\label{proc:#1}}
\def\procref#1{\ref{proc:#1}}
\def\Cal#1{{\mathcal #1}}
\def\myproof{\begin{proof}}
\def\myqed{\end{proof}}
\def\opT{{\text{\rm T}}}
\def\mlim{\lim}

\def\geqslant{\geq}

\newcommand{\opSupp}{{\text{\rm Supp}}}
\newcommand{\opHess}{{\mathrm{Hess}}}

\newcommand{\opSinh}{{\text{\rm sinh}}}
\newcommand{\opCoth}{{\text{\rm coth}}}
\newcommand{\opCosh}{{\text{\rm cosh}}}

\newcommand{\opII}{{\text{\rm II}}}

\newcommand{\Spmq}{\mathbf{S}^{p-1,q}}
\def\seq[#1]#2{\{#2\}_{#1\in \mathbf N}}
\def\munion{\cup}
\def\minter{\cap}
\def\msup{{\text{\rm sup}}}
\def\harr#1#2{\smash{\mathop{\hbox to .5in{\rightarrowfill}}\limits^{\scriptstyle #1}_{\scriptstyle #2}}}%
\def\varr#1#2{\llap{$\scriptstyle #1$}\left\downarrow\vcenter to .5in{}\right.\rlap{$\scriptstyle #2$}}%
\def\diagram#1{{\normallineskip=8pt \normalbaselineskip=0pt \begin{matrix}#1\end{matrix}}}%
\def\sharr#1#2{\smash{\mathop{\hbox to .3in{\rightarrowfill}}\limits^{\scriptstyle #1}_{\scriptstyle #2}}}%
\makeatletter
\newtheorem*{rep@theorem}{\rep@title}
\newcommand{\newreptheorem}[2]{%
\newenvironment{rep#1}[1]{%
 \def\rep@title{#2 \ref{##1}}%
 \begin{rep@theorem}}%
 {\end{rep@theorem}}}
\makeatother

\makeatletter
\newtheorem*{rep@cor}{\rep@title}
\newcommand{\newrepcor}[2]{%
\newenvironment{rep#1}[1]{%
 \def\rep@title{#2 \ref{##1}}%
 \begin{rep@cor}}%
 {\end{rep@cor}}}
\makeatother

\newtheorem{proposition}{Proposition}[section]

\newtheorem{corollary}[proposition]{Corollary}
\newtheorem{corx}{Corollary}

\newtheorem{theorem}[proposition]{Theorem}
\newtheorem{thmx}[corx]{Theorem}

\newtheorem{definition}[proposition]{Definition}
\newtheorem{lemma}[proposition]{Lemma}

\newtheorem{step}{Step}

\theoremstyle{remark}
\newtheorem{remark}[proposition]{Remark}
\newtheorem*{remark*}{Remark}

\counterwithin{equation}{section}

\title{On complete maximal submanifolds in pseudo-hyperbolic space}

\author[A. Seppi]{Andrea Seppi}
\address{A. Seppi: Univ. Grenoble Alpes, CNRS, IF, 38000 Grenoble, France.} \email{andrea.seppi@univ-grenoble-alpes.fr}

\author[G. Smith]{Graham Smith}
\address{G. Smith: Departamento de Matem\'atica, Pontif\'\i cia Universidade Cat\'olica, Rio de Janeiro, Brasil} \email{grahamandrewsmith@gmail.com}

\author[J. Toulisse]{J\'er\'emy Toulisse}
\address{J. Toulisse: Universit\'e C\^ote d'Azur, CNRS,  LJAD,  France}
\email{jeremy.toulisse@univ-cotedazur.fr}

\begin{document}

\maketitle
\begin{abstract}
We provide a full classification of complete maximal $p$-dimensional spacelike submanifolds in the pseudo-hyperbolic space $\H$, and we study its applications to Teichm\"uller theory and to the theory of Anosov representations of hyperbolic groups in $\POpqp$.
\end{abstract}
\tableofcontents

\section{Introduction}

Minimal submanifolds have become an invaluable part of the modern geometer's toolbox, having played fundamental roles in the proofs of a number of notable results. In recent years 
there has been a growing interest in the application of maximal submanifolds to the study of pseudo-riemannian symmetric spaces. In this paper, we provide a full classification of complete maximal $p$-dimensional spacelike submanifolds in the pseudo-hyperbolic space $\H$, and we discuss new results that this classification yields both in Teichm\"uller theory and in the theory of Anosov representations of hyperbolic groups in $\POpqp$.\footnote{For the study of maximal surfaces in pseudo-riemannian symmetric spaces, please see \cite{ishi,open,collierSp,CTT,LTW,nie,LT,colliertoulisse}. For existing applications of maximal surfaces to Teichm\"uller theory please see \cite{bonschl,MR3035326,scarincikrasnov,MR3714718,andreaJEMS}, and for discussions of the theory of Anosov representations, please see \cite{LabAnFlow,GW,fannyICM,DGK,zbMATH07467770,canarynotes}.}

\subsection{Main results}

In this paper, we take \emph{pseudo-hyperbolic space} to be the projective space $\H$ of negative-definite lines in $\R^{p,q+1}$. It is a homogeneous pseudo-riemannian space of constant sectional curvature equal to $-1$ and of signature $(p,q)$, that is, having $p$ positive and $q$ negative linearly-independent directions. When $q$ is equal to $0$ or $1$, $\H$ reduces to the better-known examples of hyperbolic space $\Hp$ and anti-de Sitter space $\AdS^{p,1}$ respectively.

As in the hyperbolic and anti-de Sitter cases, $\H$ possesses an \emph{asymptotic boundary} $\partial_\infty\H$, which we identify with the space of isotropic lines in $\R^{p,q+1}$. The union $\H\cup\partial_\infty\H$ is compact with respect to the topology that it inherits as a subset of projective space.

We identify a special class of topologically embedded spheres in $\partial_\infty\H$ as follows. We say that a triple of pairwise distinct points in $\partial_\infty\H$ - which, we recall, represent lines in $\R^{p,q+1}$ - is \emph{positive} whenever their span is $3$-dimensional with signature $(2,1)$, and \emph{non-negative} whenever their span contains no negative-definite $2$-plane. When $p\geq 3$, we say that a topologically embedded $(p-1)$-sphere $\Lambda$ in $\partial_\infty\H$ is \emph{positive} (respectively \emph{non-negative}) whenever every triple of pairwise distinct points that it contains is positive (respectively non-negative). When $p=2$, for topological reasons, we require in addition that $\Lambda$ contain at least $1$ positive triple. This latter case is studied in detail by Labourie--Toulisse--Wolf in \cite{LTW}, where they call non-negative $1$-spheres \emph{semi-positive loops}. We denote by $\Bd$ the space of non-negative $(p-1)$-spheres in $\partial_\infty\H$, furnished with the Hausdorff topology.

We define a \emph{maximal $p$-submanifold} of $\H$ to be a connected, $p$-dimensional, smooth spacelike submanifold which is a critical point of the area functional with respect to compactly supported variations. We denote by $\M$ the space of complete maximal $p$-submanifolds of $\H$, furnished with the topology of smooth convergence over compact sets. Given any complete maximal $p$-submanifold $M$ of $\H$, we denote by $\partial_\infty M$ the intersection of its closure with $\partial_\infty\H$, and we call this set the \emph{asymptotic boundary} of $M$. It is relatively straightforward to show (see Corollary \ref{cor:BoundaryOperatorIsContinuous}) that $\partial_\infty M$ is always a non-negative $(p-1)$-sphere, and even that $\partial_\infty$ maps $\M$ continuously into $\Bd$. We prove the converse of this fact.

\begin{thmx}\label{thm:introhomeo}
The asymptotic boundary map $\partial_\infty:\M\to\Bd$ is a homeomorphism. In particular, for every non-negative $(p-1)$-sphere $\Lambda$ in $\bH$, there exists a unique complete maximal $p$-submanifold of $\H$ with asymptotic boundary $\Lambda$.
\end{thmx}

\noindent The proof of Theorem \ref{thm:introhomeo} requires an in-depth study of the asymptotic structure of complete maximal $p$-submanifolds in $\H$. The techniques that we develop yield, in addition, the following new result concerning the total curvatures of complete maximal $p$-submanifolds with sufficiently regular asymptotic boundaries.

\begin{thmx}\label{thm:DecayOfSecondFundamentalForm}
If $M$ is a complete maximal $p$-submanifold in $\H$ with $C^{3,\alpha}$ asymptotic boundary, and if $\opII$ denotes its second fundamental form, then, for all $s>p-1$,
\begin{equation}\label{eqn:DecayOfSecondFundamentalForm}
\|\opII\|\in L^s(M,\mathrm{dVol}_M)\ .
\end{equation}
\end{thmx}

\subsection{Historical background}\label{sec:historical}

From a historical perspective, Theorem \ref{thm:introhomeo} is best understood in the context of asymptotic Plateau problems. The asymptotic Plateau problem in hyperbolic $p$-space was first addressed by Anderson in \cite{anderson}, where he proved the existence of volume-minimizing $k$-dimensional currents in $\Hp$ bounded by any given closed, embedded $(k-1)$-dimensional submanifold of $\partial_\infty\Hp$. In general, the questions of regularity, topological type, and uniqueness of such minimizing currents present hard and subtle problems (see, for example, \cite{surveyAPP}). Nonetheless, in dimension $3$, Anderson showed that every Jordan curve $\Lambda$ in $\partial_\infty\Hthree$ bounds a smoothly embedded minimal disk. However, these disks are not necessarily unique, even under the additional hypothesis that $\Lambda$ be the limit set of some quasi-fuchsian representation (see \cite[Theorem 5.2]{anderson}). Indeed, Huang--Wang constructed in \cite{huangwang} quasi-fuchsian representations whose limit sets may bound arbitrarily many smoothly embedded, invariant minimal disks.

This contrasts sharply with the situation in anti-de Sitter space. Indeed, in \cite{ABBZ}, Andersson--Barbot--B{\'e}guin--Zeghib showed that, given any representation $\rho:\Gamma=\pi_1(N)\to\mathsf{PO}(p,2)$, which is the holonomy of a maximal, globally hyperbolic, anti-de Sitter manifold homeomorphic to $N\times\R$, for some closed manifold $N$, there exists a unique $\rho$-invariant maximal hypersurface in $\Hpone$ with asymptotic boundary equal to the limit set of $\rho$. In particular, $\rho$ acts on this hypersurface freely and properly discontinuously, with quotient diffeomorphic to $N$. Likewise, in \cite{bonschl}, Bonsante--Schlenker showed that every positive $(p-1)$-sphere in $\partial_\infty\Hpone$ is the asymptotic boundary of some complete maximal hypersurface and, furthermore, when $p=2$ and the boundary curve is the graph of a quasi-symmetric homeomorphism, this hypersurface is also unique.

Analogous results have recently been obtained for maximal surfaces in the pseudo-hyperbolic space $\textbf{H}^{2,q}$. Indeed, in \cite{CTT}, Collier--Tholozan--Toulisse proved that, for any closed, orientable surface $S$ of genus at least $2$, and any maximal representation $\rho:\pi_1(S)\rightarrow\mathsf{PO}_0(2,q+1)$, the limit set of $\rho$ is the asymptotic boundary of a unique complete $\rho$-invariant maximal surface in $\Htwoq$. As before, $\rho$ then acts freely and properly-discontinuously on this surface, with quotient diffeomorphic to $S$. Finally, in \cite{LTW}, Labourie--Toulisse--Wolf generalised this result to prove that every non-negative 1-sphere in $\partial_\infty\textbf{H}^{2,q}$ is the asymptotic boundary of a unique complete maximal surface.

Theorem \ref{thm:introhomeo} thus unifies the known results for $\Hpone$ and $\Htwoq$, and extends them to $\H$ for all $(p,q)$, whilst addressing the most general hypotheses on the asymptotic boundary.

\subsection{Techniques and novelties}\label{subsec:techniques}

We note first that our approach is quite different from those used in the earlier works mentioned above. Indeed, Anderson used geometric measure theory to address the Plateau problem in $\Hp$, a technique which to date has no pseudo-riemannian analogue. Likewise, Andersson--Barbot--B{\'e}guin--Zeghib and Bonsante--Schlenker used the works \cite{zbMATH03822642,zbMATH05173443} of Gerhardt, \cite{zbMATH04092296} of Bartnik, and \cite{zbMATH02198988} of Ecker, which are peculiar to the Lorentzian setting. Finally Collier--Tholozan--Toulisse used Higgs bundles, whilst Labourie--Toulisse--Wolf used pseudo-holomorphic curves, both of which are peculiar to the case of surfaces.

We prove Theorem \ref{thm:introhomeo} by applying the continuity method in a global manner. This has the advantage over previous approaches of yielding detailed information concerning the asymptotic structures of complete maximal $p$-submanifolds with smooth asymptotic boundaries, yielding, as a byproduct, Theorem \ref{thm:DecayOfSecondFundamentalForm}.

The continuity method decomposes into three main steps, namely compactness, uniqueness and perturbation (or, stability).

Our compactness result is a manifestation of the dichotomy first observed by Labourie in \cite{LabMA} in the context of $k$-surfaces in hyperbolic $3$-space (see also \cite{Schl_degen,SmiSLC,SmiCSC,SmiQMAK}).

\begin{step}[Compactness -- Theorem \ref{thm:degenerate}]\label{thm:introdegenerate}
If $\seq[m]{M_m}$ is a sequence of complete maximal $p$-submanifolds of $\H$ then, either
\begin{enumerate}
\item $\seq[m]{M_m}$ subconverges in the $C^\infty_\oploc$ topology to a smooth, complete maximal $p$-submanifold of $\H$, or
\item $\seq[m]{M_m}$ subconverges in the Hausdorff topology to a Lipschitz $p$-submanifold foliated by complete, lightlike geodesics, all having the same endpoint at infinity.
\end{enumerate}
\end{step}

\noindent We call submanifolds of the second type \emph{degenerate}. We will see below (see Lemma \ref{lemma:AdmissableBoundaryCriteria}) that, up to isometries of the ambient space, the space of degenerate submanifolds is itself homeomorphic to the space of $1$-Lipschitz maps from a hemisphere $\Upp^{p-1}\subseteq\S^{p-1}$ into $\S^{q-1}$. Indeed, degenerate submanifolds are simply the graphs of suspensions of such maps.

Recall now that the space $\M$ of complete maximal $p$-submanifolds carries the topology of smooth convergence over compact sets, whilst the space $\Bd$ of non-negative $(p-1)$-spheres carries the Hausdorff topology. Since no degenerate submanifold can have a non-negative sphere as its asymptotic boundary, one of the main consequences of Theorem \ref{thm:degenerate} for the proof of Theorem \ref{thm:introhomeo} is that the asymptotic boundary map $\partial_\infty:\M\rightarrow\Bd$ is proper.

Uniqueness is proven using the maximum principle. Our proof is similar to that of \cite{LTW}, although some care is required in the present, higher-dimensional setting.

\begin{step}[Uniqueness -- Theorem \ref{thm:injective}]\label{thm:introuniqueness}
A non-negative $(p-1)$-sphere in $\bH$ is the asymptotic boundary of at most one complete maximal $p$-submanifold of $\H$.
\end{step}

The lengthiest and most technical part of the proof is the following stability result.

\begin{step}[Stability -- Theorem \ref{thm:perturb H+}]\label{thm:introMaximalGraphsPerturb}
Let $(\Lambda_t)_{t\in(-\epsilon,\epsilon)}$ be a smoothly varying family of smooth, spacelike spheres in $\partial_\infty\H$. If there exists a complete maximal $p$-submanifold $M$ of $\H$ with asymptotic boundary $\Lambda_0$ then, upon reducing $\epsilon$ if necessary, there exists a family $(M_t)_{t\in(-\epsilon,\epsilon)}$ of complete maximal $p$-submanifolds such that, $M_0=M$ and, for all $t$, $M_t$ has asymptotic boundary $\Lambda_t$.
\end{step}

\noindent The usual approach to proving stability results of this kind is to first represent maximal $p$-submanifolds near $M$ as zeroes of some functional over some Banach space, and then to apply the implicit function theorem. In the present case, the non-compactness of the submanifolds in question presents an extra layer of difficulty, requiring us to study such things as asymptotic models and pre-elliptic estimates. This will be carried out in Section \ref{sec:stability general}, where, in addition, we will provide a far more detailed discussion of the main ideas used in this step of the proof.

Having established Steps \ref{thm:introdegenerate},  \ref{thm:introuniqueness} and \ref{thm:introMaximalGraphsPerturb}, Theorem \ref{thm:introhomeo} readily follows by the continuity method. Indeed, let $\opIm(\partial_\infty)$ denote the image of $\partial_\infty$, let $\Bd^\infty$ denote the space of smooth, spacelike $(p-1)$-spheres in $\partial_\infty\H$, and note that this is a dense subset of $\Bd$. Since totally geodesic, spacelike $p$-subspaces of $\H$ are trivially maximal, $\opIm(\partial_\infty)$ has non-trivial intersection with $\Bd^\infty$. Since $\Bd^\infty$ is path-connected, the continuity method then shows that $\opIm(\partial_\infty)$ contains $\Bd^\infty$. Finally, by density of $\Bd^\infty$, and properness of $\partial_\infty$, it follows that $\opIm(\partial_\infty)$ contains $\Bd$, thus proving Theorem \ref{thm:introhomeo}.

Finally, the asymptotic analysis developed to prove Step \ref{thm:introMaximalGraphsPerturb} involves the use of weighted function spaces over complete maximal $p$-submanifolds. Although weighted spaces are not actually required for the proof of Theorem \ref{thm:introhomeo}, they provide detailed asymptotic information at little extra cost. In particular, they permit us to show that the norms of the second fundamental forms of suitably regular cones, on the one hand, and suitably regular, complete maximal $p$-submanifolds, on the other, both decay exponentially at an equal rate, whilst their volume forms both grow exponentially at an equal, slower rate. Theorem \ref{thm:DecayOfSecondFundamentalForm} is then a straightforward consequence of these two properties.

\subsection{Applications I - Anosov representations}

Amongst our main motivations for the study of complete maximal $p$-submanifolds in $\H$ are their applications to the study of representations of word-hyperbolic groups in $\mathsf{PO}(p,q+1)$. In what follows, $\Gamma$ will denote a word-hyperbolic group with connected Gromov boundary.

In \cite{DGK}, Danciger--Gu\'eritaud--Kassel introduced a notion of convex-cocompactness for representations in $\mathsf{PO}(p,q+1)$. We say that a representation $\rho:\Gamma\rightarrow\mathsf{PO}(p,q+1)$ is $\H$-\emph{convex-cocompact} whenever it has finite kernel and acts properly-discontinuously and cocompactly on some closed, convex subset $K$ of $\H$ whose interior $\text{Int}(K)$ is non-trivial and whose ideal boundary $\partial_\infty K$ contains no non-trivial segment. We will be concerned here with the more specific case where the Gromov boundary of $\Gamma$ is homeomorphic to a $(p-1)$-sphere. In this case, Danciger--Gu\'eritaud--Kassel showed that $\Lambda=\partial_\infty K$ is a positive $(p-1)$-sphere\footnote{\label{footnote2}Note that, in the terminology of \cite{DGK}, objects that are here described as ``positive'' are there referred to as ``negative'' (and vice-versa). The terminology of \cite{DGK} reflects the fact that, in the above situation, $\Lambda$ lifts to a cone in $\R^{p,q+1}$ of which any pair of non-colinear points has negative scalar product (c.f. Lemma \ref{lemma:PositiveSpheres}). The terminology of the present paper has been chosen to be consistent with \cite{fockgoncharov} and \cite{GLW}.} in $\partial_\infty\H$. Since every positive sphere is, in particular, non-negative, Theorem \ref{thm:introhomeo} thus yields the following result.

\begin{corollary}\label{cor:intro2}
Let $\Gamma$ be a word-hyperbolic group with Gromov boundary homeomorphic to a $(p-1)$-sphere. Every $\H$-convex-cocompact representation $\rho:\Gamma\rightarrow\mathsf{PO}(p,q+1)$ preserves a unique complete maximal $p$-submanifold $M$ in $\H$. Furthermore
\begin{enumerate}
\item $\rho$ acts properly-discontinuously and cocompactly on $M$, and
\item $M$ depends real analytically on $\rho$.
\end{enumerate}
\end{corollary}

Beyrer--Kassel have told us of an intriguing application of Corollary \ref{cor:intro2} to the study of higher-dimensional extensions of higher-rank Teichm\"uller theory. Recall (c.f. \cite{wienhardICM}) that when $\Gamma$ is a compact surface group, and $\G$ is a reductive Lie group of higher rank, a connected component of $\opHom(\Gamma,\G)$ is said to be a \emph{higher-rank Teichm\"uller space} whenever it consists entirely of discrete and faithful representations. The study of such spaces, which share a number of properties of classical Teichm\"uller space, has yielded over the last two decades a rich and fascinating theory. Although it is natural to generalize this concept to higher-dimensional word-hyperbolic groups, to date only a few higher-dimensional cases have been shown to exist. One such is given by Barbot in \cite{barbot}, where he showed that, when $\Gamma$ is the fundamental group of a compact, $p$-dimensional hyperbolic manifold, the quasi-fuchsian component of $\opHom(\Gamma,\mathsf{PO}(p,2))$ consists entirely of $\Hpone$-convex-cocompact representations, which, in particular, are discrete and faithful. In their work \cite{BK}, contemporaneous with the present paper, Beyrer--Kassel prove a converse of Corollary \ref{cor:intro2}, namely that any representation acting properly-discontinuously and cocompactly on a weakly spacelike $p$-dimensional submanifold of $\H$ is $\H$-convex-cocompact. Together with Corollary \ref{cor:intro2}, this allows them to show in \cite[Theorem $1.3$]{BK} that, for all $p\geq 2$ and $q\geq 1$, and for any word-hyperbolic group $\Gamma$ with Gromov boundary homeo\-morphic to a $(p-1)$-sphere, the set of $\H$-convex-cocompact representations $\rho:\Gamma\rightarrow\mathsf{PO}(p,q+1)$ is a union of connected components of $\opHom(\Gamma,\mathsf{PO}(p,q+1))$. In this manner, they extend Barbot's result to yield a large family of new higher-dimensional higher-rank Teichm\"uller spaces, including many examples which are not quasi-fuchsian (c.f. \cite{leemarquis,MST}).

We now discuss applications of Corollary \ref{cor:intro2} to the study of $\mathsf{P}_1$-Anosov representations of word-hyper\-bolic groups in $\mathsf{PO}(p,q+1)$. Anosov representations were introduced by Labourie in \cite{LabAnFlow}, and have since become a cornerstone of higher-rank Teichm\"uller theory. We refer the reader to Danciger--Gu\'eritaud--Kassel's paper \cite{DGK} for a formal definition of the concept of $\mathsf{P}_1$-Anosov representations. In addition, in the same paper, they described a direct relationship between such representations and the $\H$-convex-cocompact representations discussed above. To understand this, note first that the set of non-null lines in $\R^{p,q+1}$ consists of two connected components, namely $\H$ and $\Hqppn$, where the sign of the metric of the latter is inverted. Danciger--Gu\'eritaud--Kassel showed that every $\H$-convex-cocompact representation is $\mathsf{P}_1$-Anosov and, conversely, that every $\mathsf{P}_1$-Anosov representation is either $\H$-convex-cocompact or $\Hqppn$-convex-cocompact. We say that the representation is \emph{positive} in the former case, and \emph{negative} in the latter, so that every $\H$-convex-cocompact representation is a positive $\mathsf{P}_1$-Anosov representation, and vice versa. In particular, in this case, $\Lambda:=\partial_\infty K$ coincides with the proximal limit set of the representation $\rho$\footnote{See Footnote \ref{footnote2}.}.

Our first application of Corollary \ref{cor:intro2} provides a new constraint on the hyperbolic groups which admit positive $\mathsf{P}_1$-Anosov representations.

\begin{corollary}\label{cor:intro3}
Let $\Gamma$ be a torsion-free word-hyperbolic group with Gromov boundary homeomorphic to a $(p-1)$-sphere. If $\Gamma$ admits a positive $\mathsf{P}_1$-Anosov representation in $\mathsf{PO}(p,q+1)$, then it is isomorphic to the fundamental group of a smooth, closed $p$-dimensional manifold with universal cover diffeomorphic to $\R^p$.
\end{corollary}

\noindent It is worth comparing Corollary \ref{cor:intro3} to the result \cite{barluckwein} of Bartels--L\"uck--Weinberger, which states that, for $p\geq 6$, any torsion-free hyperbolic group with Gromov boundary homeomorphic to a $(p-1)$-sphere is the fundamental group of a closed $p$-dimensional \emph{topological} manifold with universal cover \emph{homeomorphic} to $\R^p$. Note, in particular, that in Example $5.2$ and Lemma $5.3$ of that paper, the authors construct, for all $k\geq 2$, a torsion-free hyperbolic group $\Gamma$, with Gromov boundary homeomorphic to $\S^{4k-1}$, which is \emph{not} isomorphic to the fundamental group of any closed \emph{smooth} aspherical manifold. We consequently have the following corollary.

\begin{corollary}\label{cor:nonexistence}
For any $k\geq 2$ and $q\geq 1$, there exists a torsion-free word-hyperbolic group $\Gamma$ with Gromov boundary homeomorphic to a $(4k-1)$-sphere which does not admit any positive $\mathsf{P}_1$-Anosov representation in $\mathsf{PO}(4k,q+1)$.
\end{corollary}

\begin{remark*}
More precisely, the top Pontrjagin number of the piecewise linear manifold associated to $\Gamma$ appears to provide an obstruction to the existence of positive $\mathsf{P}_1$-Anosov representations in $\mathsf{PO}(p,q+1)$.
\end{remark*}

Our second application concerns the structure of a class of manifolds introduced by Guichard--Weinhard in \cite{GW}. Indeed, let $\Isot(E)$ denote the space of maximally isotropic subspaces of $E$. Given a positive $\mathsf P_1$-Anosov representation $\rho$ of $\Gamma$ in $\mathsf{PO}(p,q+1)$, Guichard--Wienhard constructed a domain $\Omega_\rho$ in $\Isot(E)$ upon which $\rho(\Gamma)$ acts properly-discontinuously and cocompactly. The quotient $\Omega_\rho/\rho(\Gamma)$is a closed manifold locally modelled on $\Isot(E)$. However, to date, little more has been shown concerning its structure.

Given two natural numbers $k\leq n$, the \emph{Stiefel manifold} $\V_{k,n}$ is the space of $k$-tuples of unit vectors in $\R^n$ that are pairwise orthogonal. For all such $k$ and $n$, $\V_{k,n}$ is diffeomorphic to the homogeneous space $\mathsf{O}(n)/\mathsf{O}(n-k)$. In particular, it is connected unless $k=n$, in which case it has $2$ connected components.

\begin{corollary}\label{theorem:GeometricStructures}
Let  $\Gamma$ be a torsion-free word-hyperbolic group with Gromov boundary homeomorphic to a $(p-1)$-sphere, let $\rho:\Gamma\to\mathsf{PO}(p,q+1)$ be a positive $\mathsf{P}_1$-Anosov representation, and let $\Omega_\rho$ denote its Guichard--Wienhard domain.
\begin{enumerate}
	\item If $q\geq p$, and if $M$ denotes the complete maximal $p$-submanifold preserved by $\rho$, then $\Omega_\rho/\rho(\Gamma)$ is homeomorphic to a $\V_{p,q}$-bundle over $M/\rho(\Gamma)$.
	\item If $q<p$, then $\Omega_\rho$ is empty.
\end{enumerate}
In particular, in the former case, $\Omega_\rho$ is connected, unless $p=q$ and the first Stiefel-Whitney class of $\N M/\rho(\Gamma)$ vanishes, in which case it has 2 connected components.
\end{corollary}

\subsection{Applications II - renormalized area}\label{subsec:introren}

We conclude this introduction with a discussion of the renormalized area of complete maximal surfaces. Motivated by applications to the AdS-CFT correspondence, the \emph{renormalized area} of a maximal surface $M$ in $\mathbf{H}^{2,q}$ with second fundamental form $\opII$ is defined by
\begin{equation}\label{eqn:RenormalizedArea}
A_\opren(M) := \int_M\|\opII\|^2\opdVol_M\ .
\end{equation}
The study of the renormalized area of maximal surfaces in $\mathbf{H}^{2,q}$ presents a number of interesting, as yet unstudied, problems. For example, following \cite{AlexMazz} it is of interest to determine its first and second variations. Likewise, in the spirit of \cite{Bishop}, it is also of interest to determine how finiteness of the renormalized area may be expressed in terms of properties of the asymptotic boundary. Substituting $p=s=2$ in Theorem \ref{thm:DecayOfSecondFundamentalForm} yields a partial response to this latter problem for surfaces in pseudo-hyperbolic space.

\begin{corollary}\label{cor:renarea}
Every complete maximal surface $M$ in $\mathbf{H}^{2,q}$ with $C^{3,\alpha}$ asymptotic boundary has finite renormalized area.
\end{corollary}

Finally, in the $(2+1)$-dimensional anti-de Sitter case, that is, when $(p,q)=(2,1)$, using the ideas developed by Mess in \cite{mess} (see also \cite{mess,surveyandreafra}), Corollary \ref{cor:renarea} yields the following new Teichm\"uller-theoretic result. A self-diffeomorphism of $\mathbf{H}^2$ is called \emph{minimal Lagrangian} whenever its graph is a minimal Lagrangian surface in $\mathbf{H}^2\times \mathbf{H}^2$. In \cite{bonschl}, Bonsante-Schlenker used maximal surfaces in anti-de Sitter geometry to prove that every quasisymmetric circle homeomorphism admits a unique quasiconformal minimal Lagrangian extension to $\mathbf{H}^2$.

\begin{corollary}\label{cor:minlag}
If $f$ is a $C^{3,\alpha}$ circle diffeomorphism, then the Beltrami coefficient of its unique quasiconformal minimal Lagrangian extension is an element of $L^2(\mathbf{H}^2,\mathrm{dVol}_{\mathbf{H}^2})$.
\end{corollary}

\noindent It is known that quasiconformal maps of $\mathbf{H}^2$ with square integrable Beltrami coefficient extend to \emph{Weil-Petersson} circle homeomorphisms. This latter family is precisely the closure of the space of circle diffeomorphisms with respect to the topology induced on the universal Teichm\"uller space by the (infinite-dimensional) K\"ahler structure constructed in \cite{takteo}. Motivated by the recent work of Bishop \cite{Bishop}, which characterized Weil-Petersson quasicircles in ${\mathbf{H}^2}$ as those Jordan curves that bound a complete minimal surface in $\mathbf{H}^3$ of finite renormalized area, it is natural to conjecture that the quasiconformal minimal Lagrangian extension of $f$ is in $L^2(\mathbf{H}^2,\mathrm{dVol}_{\mathbf{H}^2})$ if and only if $f$ is in the Weil-Petersson class. We leave this question for future investigation.

\subsection{Structure of paper}
In Section \ref{sec:prel pseudo hyp}, we provide some background on pseudo-hyperbolic space and on the associated Riemannian symmetric space. In Section \ref{sec:prel graphs}, we introduce spacelike and maximal submanifolds and, making use of so-called Fermi charts, we define the notion of entire graph. In Section \ref{sec:boundaries}  we study positive and non-negative spheres, and their relationship to the asymptotic boundaries of entire graphs.  In Section \ref{sec:proper} we prove Theorem \ref{thm:degenerate} and deduce the properness of the asymptotic boundary map (Theorem \ref{theorem:Properness}). In Section \ref{sec:uniqueness} we prove the uniqueness of a complete maximal $p$-submanifold with given asymptotic boundary (Theorem \ref{thm:injective}).
In Section \ref{sec:stability general} we prove  the stability result, namely Theorem \ref{thm:perturb H+}. In Section \ref{sec:conclude} we conclude the proofs of Theorem \ref{thm:introhomeo} and Theorem \ref{thm:DecayOfSecondFundamentalForm}.
In Section \ref{sec:anosov} we prove Corollaries \ref{cor:intro2}, \ref{cor:intro3}, \ref{theorem:GeometricStructures}, and \ref{cor:renarea}. Finally, in Section \ref{sec:RenormalizedArea}, we prove Corollary \ref{cor:minlag}.

\subsection{Acknowledgements} We are grateful to Jonas Beyrer and Fanny Kassel for the interest they have shown in this work, and for helpful comments made to earlier drafts of this paper. A good portion of this research was carried out whilst the second author was visiting the Institut des Hautes \'Etudes Scientifiques, and we are grateful for the excellent research conditions provided during that time. A further part of this research was also carried out in the Institut de Mathématiques de Marseille, and we are grateful for their hospitality.

\section{Pseudo-hyperbolic geometry}\label{sec:prel pseudo hyp}

We begin by introducing pseudo-hyperbolic space and reviewing its basic properties. In this section, we describe the $3$ main models for pseudo-hyperbolic space that will be used throughout the sequel, namely the projective model $\H$, its double cover $\H_+$, and the quadric $\Quad$, which is a subset of $\Rpqp$ isometric to $\H_+$. In addition, we classify its geodesics and totally geodesic subspaces, and we describe the geometry of its grassmannian bundle of spacelike tangent $p$-planes.

\subsection{Pseudo-hyperbolic space}\label{subsec:pseudohyperbolicspace}

Let $p,q$ be non-negative integers, let $E$ be a real vector space of dimension $(p+q+1)$, and let $\P(E)$ denote its projective space. We equip $E$ with a bilinear form $\q$ of signature $(p,q+1)$, that is, with $p$ positive and $(q+1)$ negative linearly independent directions. We define \emph{pseudo-hyperbolic space} by
\begin{equation}
\H := \left\{ x\in \P(E)\ |\ \q(x,x)<0\right\}~.
\end{equation}
It will often be useful to work with a double cover of this space. Thus, let $\P_+(E)$ denote the projective space of {\sl oriented} lines in $E$, that is the quotient of $E\setminus\{0\}$ by $\mathbf{R}_+$, and define
\begin{equation}
\H_+ := \left\{ x\in \P_+(E)\ |\ \q(x,x)<0\right\}~.
\end{equation}
The natural map $\P_+(E)\to \P(E)$ yields the desired double cover $\H_+ \to \H$. Note that the preimage of any connected subset $A$ of $\P(E)$ consists of either one or two connected components. We call any such connected component a \emph{lift} of $A$.

The double cover $\H_+$ naturally identifies with a quadric as follows. Denote
\begin{equation}
\Quad := \left\{ x\in E\ |\ \q(x,x)=-1 \right\}\ .
\end{equation}
Trivially, $\Quad$ projects onto $\H_+$, whilst every point of $\H_+$ contains a unique representative in $\Quad$. We thus obtain a natural diffeomorphism $r:\H_+\to\Quad$. In particular, this diffeomorphism induces pseudo-riemannian metrics on both $\H$ and $\H_+$. Indeed, for any $x\in\H_+$, the differential of $r$ maps the tangent space $\T_x\H_+$ to $r(x)^\bot$, the $\q$-orthogonal subspace to $r(x)$ in $E$. The pull-back through $r$ of the restriction of $\q$ to $\Quad$ then defines a metric $\g$ on $\H_+$ of signature $(p,q)$ and of constant sectional curvature equal to $-1$. Since this metric invariant under the involution $x\mapsto-x$, it likewise induces a metric over $\H$ which we also denote by $\g$, making the double cover $\H_+ \to \H$ into a local isometry.

The group $\mathsf{O}(p,q+1)$ of orthogonal transformations of $(E,\q)$ acts by isometries on $\H_+$, whilst its projectivization $\mathsf{PO}(p,q+1)$ acts by isometries on $\H$. Together with $\mathbf{Z}_2$, these groups form the short exact sequence
\begin{equation}
1 \longrightarrow \mathbf Z_2 \longrightarrow \SO \longrightarrow \PO \longrightarrow 1 ~.
\end{equation}
In both cases, the stabilizer of any point is isomorphic to $\mathsf{O}(p,q)$, yielding identifications of $\H$ and $\H_+$ with the respective pseudo-riemannian symmetric spaces $\mathsf{PO}(p,q+1)/\mathsf{O}(p,q)$ and $\mathsf{O}(p,q+1)/\mathsf{O}(p,q)$, equipped with suitable scalar multiples of their Killing metrics.

Finally, the ideal boundary $\bH$ of $\H$ (respectively $\bH_+$ of $\H_+$) is defined to be the space of isotropic lines (respectively oriented isotropic lines) in $E$. Each of these spaces carries a unique $\mathsf{O}(p,q+1)$-invariant conformal class of pseudo-riemannian metrics of signature $(p-1,q)$.

\subsection{Totally-geodesic subspaces}

A submanifold $S$ of $\H$ or $\H_+$ is called \emph{spacelike}, \emph{timelike} or \emph{lightlike} whenever its induced metric is respectively positive-definite, negative-definite or degenerate.

Complete totally-geodesic subspaces of $\H$ are given by the intersections with $\H$ of projective subspaces $\P(V)$ of $\P(E)$. In particular, any \emph{complete geodesic} $\gamma$ in $\H$ is the intersection of $\H$ with a projective line $\P(V)$, and is spacelike, timelike or lightlike whenever $V$ has signature $(1,1)$, $(0,2)$ or $(0,1)$ respectively. Similarly, complete totally-geodesic subspaces of $\H_+$ are defined to be connected components of the intersections of projective subspaces $\P_+(V)$ with $\H_+$, that is, they are lifts of complete totally-geodesic subspaces of $\H$, and a trichotomy analogous to that given above likewise holds for geodesics of this space.

We define a \emph{hyperbolic $p$-space} $H$ in $\H$ (resp. $\H_+$) to be the intersection of $\H$ with a projective subspace $\P(V)$ (resp. a connected component of the intersection of $\H_+$ with $\P_+(V)$), for some $V$ of dimension $(p+1)$ and signature $(p,1)$. In particular, any such $H$ is spacelike and isometric to the $p$-dimensional hyperbolic space $\mathbf H^p$.

We define a \emph{marked hyperbolic $p$-space} to be a pair $(x,H)$ where $H$ is a hyperbolic $p$-space and $x$ is a point of $H$. Note that this is equivalent to the data of an orthogonal decomposition $E= \ell \oplus U \oplus V$, where $\ell$ is a negative-definite line, $U$ is a positive-definite $p$-dimensional subspace, and $V$ is the orthogonal complement of their direct sum. Indeed, the point $x$ and the hyperbolic $p$-space $H$ are simply the respective projectivizations of $\ell$ and $\ell\oplus U$. A \emph{marked hyperbolic $p$-space} in $\H_+$ is likewise equivalent to the data of an orthogonal decomposition $E= \ell \oplus U \oplus V$, where now $\ell$ is an oriented negative-definite line, and $U$ and $V$ are as before.

\subsection{Riemannian symmetric space}

Let $\Gr$ denote the grassmannian of $p$-planes in $E$ of signature $(p,0)$. Note that the operation of orthogonal complement identifies $\Gr$ with the set of $(q+1)$-planes of signature $(0,q+1)$, and thus in turn with the space of totally-geodesic timelike $q$-spheres in $\H$.

Since $\mathsf{O}(p,q+1)$ acts transitively on $\Gr$, and since every point stabilizer is a maximal compact subgroup isomorphic to $\mathsf{O}(p)\times\mathsf{O}(q+1)$, this space naturally identifies with the riemannian symmetric space of $\mathsf{O}(p,q+1)$.

Every $p$-plane $P$ in $\Gr$ yields an orthogonal decomposition $E=P\oplus P^\perp$. Furthermore, since $P$ and $P^\perp$ are respectively positive- and negative-definite, this decomposition also yields the positive-definite inner product
\begin{equation}
\overline{\q}_P := \q_{\vert P}\oplus (-\q_{\vert P^\perp})~.
\end{equation}
The tangent space $\T_P\Gr$ naturally identifies with $\Hom(P,P^\perp)$, so that $\overline{\q}_P$ induces a scalar product $\q_{\Hom,P}$ on this space, and thus yields a complete riemannian metric over $\Gr$. This metric is  preserved by the action of $\mathsf{O}(p,q+1)$ and so identifies with a scalar multiple of the Killing metric.

Note finally that $\Gr$ is an open set in the grassmanian of $p$-planes in $E$, with topological boundary given by the set of degenerate $p$-planes not containing any negative direction.

\subsection{Grassmannian bundles} \label{ss:GrassmanianBundle}

We define the \emph{grassmanian bundle} over $\H$ to be the space $\GH$ of pairs $(x,P)$ where $x$ is a point in $\H$ and $P$ is a positive-definite $p$-plane in $\T_x\H$. Note that this space naturally identifies under the exponential map with the space of marked hyperbolic $p$-spaces. We denote by $\oppr:\GH\rightarrow\H$ the \emph{canonical projection}.

At each point $(x,P)$ in $\GH$, the Levi-Civita connection on $\H$ yields the identification
\begin{equation}
\T_{(x,P)} \GH = \T_x \H \oplus \Hom(P,P^\bot)~.
\end{equation}
Furthermore, since $P$ is non-degenerate, $\T_x\H = P \oplus P^\bot$, so that
\begin{equation}\label{eq:dec TGH}
\T_{(x,P)} \GH = P \oplus P^\bot \oplus \Hom(P,P^\bot)~.
\end{equation}
With respect to this decomposition, we define the complete riemannian metric $h$ over $\GH$ by
\begin{equation}\label{eq:metric TGH}
h_{(x,P)} = \q_{\vert P} \oplus (-\q_{\vert P^\bot}) \oplus \q_{\Hom,P}~.
\end{equation}

The group $\mathsf{PO}(p,q+1)$ acts transitively by isometries on $(\GH,h)$ such that every point stabilizer is isomorphic to $\mathsf{P}(\mathsf{O}(p)\times\mathsf{O}(q))$. Moreover, the \emph{forgetful map}
\begin{equation*}
\begin{array}{llll}
F : & \GH & \longrightarrow & \Gr \\
& (x,P) & \longrightarrow & P
\end{array}
\end{equation*}
is an $\mathsf{PO}(p,q+1)$-equivariant proper riemannian submersion.

The grassmannian bundle $\GHH$ over $\H_+$ is defined in a similar manner and has similar properties. In particular, its forgetful map $F:(x,P)\mapsto P$ likewise defines an $\mathsf{O}(p,q+1)$-equivariant proper riemannian submersion from $\GHH$ to $\Gr$.

\begin{lemma}\label{lemma:DivergingInGrassmanian}
If $\seqn{(x_n,P_n)}$ is an unbounded sequence in $\GH$ (resp. $\GHH$), then, up to extraction of a subsequence, either $\seqn{x_n}$ converges in $\P(E)$ (resp. $\P_+(E)$) to a point in $\bH$ (resp. $\bH_+$) or $\seqn{P_n}$ converges in the grassmanian of $p$-planes to a degenerate $p$-plane.
\end{lemma}

\begin{proof}
Indeed, the space $\GH$ naturally embeds homeomorphically into $\H\times \Gr$, and the result now follows by definition of the induced topology.
\end{proof}

\section{Entire graphs}\label{sec:prel graphs}

We now study the basic geometry of $p$-dimensional spacelike immersions in $\H$ and $\H_+$. Note that, since we are working in mixed signature, completeness is not necessarily the most convenient property to work with, and we thus introduce the complementary concept of entire graphs. To this end, it will be useful to introduce a special class of parametrizations of $\H_+$, which we call Fermi charts. These will allow us to describe smooth entire graphs in $\H_+$, and to then adapt this concept to submanifolds of $\H$. We establish sufficient conditions for a smooth entire graph to be complete, we study the topology of the space of smooth entire graphs, as well as its closure, and we conclude by describing a sense in which families of smooth entire graphs can be considered locally uniformly spacelike.

\subsection{Spacelike immersions}

Let $M$ be a connected $p$-dimensional manifold. We say that a smooth immersion $f: M \to \H$ is \emph{spacelike} whenever its induced metric is riemannian. When this holds, $f$ identifies $\T M$ with a subbundle of $f^*\T \H$, and we define the \emph{normal bundle} $\N M$ of $f$ to be the orthogonal complement of $\T M$ in $f^*\T \H$. Note that every fibre of $\N M$ is timelike.

We denote the induced metrics on $\T M$ and $\N M$ respectively by $g_\T$ and $g_\N$. With respect to the splitting $f^* \T \H = \T M \oplus \N M$, the pull-back $\nabla$ of the Levi-Civita connection on $\H$ decomposes as
\begin{equation}\label{DecompositionOfCovDer}
\nabla = \left(\begin{array}{ll} \nabla^\T & -B \\ \II & \nabla^\N \end{array}\right)~,
\end{equation}
where $\nabla^\T$ is the Levi-Civita connection of the induced metric $g_\T$ on $M$, $\nabla^\N$ is a unitary connection on the normal bundle $\N M$, $\II$ denotes the \emph{second fundamental form}, which is an element of $\Omega^1(M,\Hom(\T M,\N M))$, and $B$ denotes the \emph{shape operator}, which is an element of $\Omega^1(M,\Hom(\N M,\T M))$.

Since $\nabla$ is torsion free and preserves $f^*\g$, for any two vector fields $X$ and $Y$ over $M$, and any section $\eta$ of $\N M$,
\begin{equation}
\eqalign{
\II(X)(Y)&=\II(Y)(X) ~ \ \text{ and }\cr
g_\N(\II(X)(Y),\eta)&=g_\T(B(X)(\eta),Y)~.\cr}
\end{equation}

We now define maximal spacelike submanifolds of $\H$, which are the main objects of interest to us.

\begin{definition}
A spacelike submanifold $M$ of $\H$ (or $\H_+$) is called \emph{maximal} whenever $\tr_{g_\T}(\II)=0$.
\end{definition}

In addition, in the sequel, we will use the following norm of the second fundamental form.\footnote{Throughout the paper, we use the symbol $\|\cdot\|$ to denote \emph{positive} norms. Similarly, we will always use the symbol $\langle\cdot,\cdot\rangle$ to denote \emph{positive definite} scalar products.}
\begin{equation}
\Vert \II\Vert^2  =  \sum_{i,j=1}^p \big\vert g_\N\left(\II(e_i)(e_j),\II(e_i)(e_j) \right)\big\vert ~,
\end{equation}
where $(e_1,\cdots,e_n)$ is any local orthonormal frame of $\T M$.

We define the \emph{Gauss lift} of a spacelike immersion $f$ by
\begin{equation}
\begin{array}{llll}
\mathcal G_f : & M & \longrightarrow & \GH \\
& x & \longmapsto & (f(x),d_xf(\T_x M))~. \end{array}
\end{equation}
With respect to the decomposition \eqref{eq:dec TGH} of Subsection \ref{ss:GrassmanianBundle}, for any $(x,P)\in \GH$, the pull-back through $\mathcal{G}_f$ of the tangent bundle of $\GH$ decomposes as
\begin{equation}
\mathcal G_f^* \T\GH = \T M \oplus \N M \oplus \Hom(\T M, \N M)~.
\end{equation}
Furthermore, with respect to the same decomposition,
\begin{equation}\label{eq:diff gauss lift}
d \mathcal G_f = (df,0,\II)~.
\end{equation}
Let $g_{\mathcal G_f}$ denote the metric induced by $\mathcal G_f$ over $M$.

\begin{lemma}\label{lemma:GaussLiftBilipschitz}
Let $f: M \to \H$ be a spacelike immersion. If there exists $C>0$ such that $\Vert \II\Vert<C$, then
\begin{equation}
g_\T \leq g_{{\mathcal G_f}} \leq (1+C^2) g_\T~.
\end{equation}
\end{lemma}
\begin{proof}
Indeed, by \eqref{eq:metric TGH} and \eqref{eq:diff gauss lift}, for all $(u,v)$,
\begin{equation*}
g_{\mathcal{G}_f}(u,v) = \q(df\cdot u,df\cdot v) + \q_{\Hom}(\II(u),\II(v)) = g_\T(u,v) + \q_{\Hom}(\II(u),\II(v))\ ,
\end{equation*}
and the first inequality follows. Observe now that, for all $u$,
\begin{equation*}
\q_\Hom(\II(u),\II(u))\leq \|\II\|^2\q(df\cdot u,df\cdot u)\ ,
\end{equation*}
and the second inequality follows.
\end{proof}

\subsection{Fermi charts}\label{subsection:FermiCoordinates}
A key advantage of working with the double cover $\H_+$ is the existence of nice global coordinate systems that we now describe. Let $(x,H)$ be a marked hyperbolic $p$-space in $\H_+$ corresponding  to the orthogonal decomposition $E=\ell\oplus U\oplus V$. Let $\Vert.\Vert$ denote the induced norm over $U$ and let $\Ball^p$ denote its unit ball. Note that the restriction of $\q$ to $W:=\ell\oplus V$ is negative-definite. Let $\S^q$ denote its unit sphere, that is, the set of all $y\in W$ such that $\q(y,y)=-1$. We define the parametrisation $\hat{\Phi}:\Ball^p\times\S^q\rightarrow\Quad$ by
\begin{equation}
\hat{\Phi}(u,w) := f(\|u\|){\phi}(u,w)\ ,\label{eqn:FermiParametrisationOfQuadricI}
\end{equation}
where,
\begin{equation}\label{eqn:FermiParametrisationOfQuadricIA}
{\phi}(u,w) := \bigg(\frac{2u}{1+\|u\|^2},w\bigg)\ ,
\end{equation}
and, for all $t$,
\begin{equation}
f(t) := \frac{1+t^2}{1-t^2}\ .\label{eqn:FermiParametrisationOfQuadricII}
\end{equation}
Let $\q'$ denote the quadratic form obtained by reversing the sign of $\q$ over $\ell$. The restriction of $\q'$ to $\ell\oplus U$ is positive-definite. Let $\S^p$ denote its unit sphere. Note now that, since we are working with $\H_+$, $\ell$ has a natural orientation, so that $\S^p$ has a well-defined upper hemisphere, which we denote by $\Upp^p$. Let $\sigma:\Upp^p\rightarrow\Ball^p$ denote stereographic projection, and define the parametrisation $\hat{\Psi}:\Upp^p\times\S^q\rightarrow\Quad$ by
\begin{equation}
\hat{\Psi}(u,w) := \hat{\Phi}(\sigma(u),w)\ .\label{eqn:FermiParametrisationOfQuadricIII}
\end{equation}
Finally, upon projectivising, we define the parametrisation $\Psi:\Upp^p\times\S^q\rightarrow\H_+$ by
\begin{equation}
\Psi(u,w) := [\hat{\Psi}(u,w)]\ .\label{eqn:FermiParametrisationOfQuadricIV}
\end{equation}
We call $\Psi$ the \emph{Fermi parametrisation} of $\H_+$ associated to $(x,H)$. In this coordinate system, the point $x$ corresponds to a point $(0,w_0)\in\{0\}\times \S^q$, the hyperbolic $p$-space $H$ identifies with $\Ball^p\times\left\{w_0\right\}$, each submanifold of the form $\Ball^p\times \{w\}$ is a hyperbolic $p$-space, and $\{0\}\times \S^q$ is a totally-geodesic timelike sphere.

The following result extends Proposition $3.5$ of \cite{CTT}.
\begin{lemma}\label{lemma:warped parameterization}
For any marked hyperbolic $p$-space $(x,H)$, the map $\Psi$ is a diffeomorphism from $\Upp^p\times\S^q$ onto $\H_+$. The pull-back through this map of the pseudo-hyperbolic metric is
\begin{equation}\label{eq:WarpedMetric1}
\Psi^*\g = f(\|u\|)^2\big(g_{\Upp^p} - g_{\S^q}\big)~.
\end{equation}
Moreover, $\Psi$ induces a diffeomorphism from $\S^{p-1}\times\S^q$ into $\partial_\infty \H_+$, and the pull-back through this diffeomorphism of the conformal structure of $\partial_\infty \H_+$ is compatible with the metric $\left(g_{\S^{p-1}} - g_{\S^q}\right)$.
\end{lemma}

\begin{proof}
Indeed, \eqref{eq:WarpedMetric1} follows by explicit calculation, as in \cite{CTT}. We likewise readily verify that $\Psi$ yields a diffeomorphism from $\Upp^{p-1}\times\S^q$ into $\partial_\infty\H_+$. Finally, the above relations also imply that the conformal structure of $\partial_\infty\H_+$ is compatible with that of $g_{\S^{p-1}} - g_{\S^q}$ on $\S^{p-1}\times\S^q$, and this completes the proof.
\end{proof}

\noindent We define the \emph{Fermi projection} $\pi:\H_+\to H$ in the above Fermi chart by
\begin{equation}
(\pi\circ\Psi^{-1})(y,z) := \Psi^{-1}(y,w_0)\ .
\end{equation}
The following lemma summarises the properties of Fermi projections that will be required in the sequel.
\begin{lemma}\label{lemma:WarpedProjectionIncreaseLength}
Let $(x,H)$ be a marked hyperbolic $p$-space in $\H_+$ with associated Fermi projection $\pi$. Then
\begin{enumerate}
\item If $x_1,x_2\in \H_+$ are such that $\pi(x_1)=\pi(x_2)$, and if, for each $i$, $\hat x_i$ denotes the representative of $x_i$ in $\Quad$, then
\begin{equation*}
\q( \hat x_1,\hat x_2) \geq -1~,
\end{equation*}
with equality holding if and only if $x_1=x_2$.
\item For any vector $v\in\T\H_+$,
\begin{equation*}
\q(d\pi(v),d\pi(v))\geq \q(v,v)~.
\end{equation*}
\end{enumerate}
\end{lemma}

\begin{proof}
To prove the first assertion, in the notation of Lemma \ref{lemma:warped parameterization}, we need to show that the product of $\hat x_1:=\hat\Phi(u,w_1)$ and $\hat x_2:=\hat\Phi(u,w_2)$ is bounded from below by $-1$. Let $\langle\cdot,\cdot\rangle_W$ denote the  restriction of $-\q$ to $W$, and note that this form is positive-definite. Since $\langle w_1,w_2\rangle_W\leq 1$ for any two elements of the unit sphere in $W$, a direct computation yields
\begin{equation*}
\eqalign{
\q(\hat x_1,\hat x_2)&=
\q\left(\left(\frac{2u}{1-\Vert u\Vert^2},\frac{1+\Vert u\Vert^2}{1-\Vert u\Vert^2}w_1\right), \left(\frac{2u}{1-\Vert u\Vert^2},\frac{1+\Vert u\Vert^2}{1-\Vert u\Vert^2}w_2\right)\right)\cr
&= \frac{4\Vert u\Vert^2}{(1-\Vert u\Vert^2)^2}-\frac{(1+\Vert u\Vert^2)^2}{(1-\Vert u\Vert^2)^2}\langle w_1,w_2\rangle_W\geq -1~.\cr}
\end{equation*}
Equality holds if and only if $\langle w_1,w_2\rangle_W=1$, that is, if and only if $w_1=w_2$. The second assertion follows immediately from (\ref{eq:WarpedMetric1}), and this completes the proof.
\end{proof}

\subsection{Entire graphs}\label{subsection:EntireGraphs}
We now establish a suitable notion of \emph{entire graph} in $\H$ and in $\H_+$. We first consider the case of smooth submanifolds. Our definition will be based on the following lemma, where we recall that $\Upp^p$ denotes the hemisphere in $\S^p$, endowed with the restriction of the spherical metric.
\begin{lemma}\label{lemma:EntireGraphsAreGraphs}
For $M$ a subset of $\H_+$, the following assertions are equivalent.
\begin{enumerate}
\item $M$ is a closed, connected, smooth, spacelike $p$-dimensional submanifold of $\H_+$.
\item $M$ is a closed, connected, smooth, spacelike submanifold of $\H_+$ diffeomorphic to $\R^p$.
\item There exists a Fermi chart in which $M$ is the graph of a smooth function $\varphi:\Upp^p\rightarrow\S^q$ with $\Vert d\varphi\Vert <1$.
\item In every Fermi chart, $M$ is the graph of a smooth function $\varphi:\Upp^p\rightarrow\S^q$ with $\Vert d\varphi\Vert<1$.
\end{enumerate}
\end{lemma}

\begin{proof}
We first prove that $(1)$ implies $(4)$. Let $(x,H)$ be a marked hyperbolic $p$-space and let $\pi:\H_+\rightarrow H$ denote its Fermi projection. Since $M$ is spacelike, by Item $(2)$ of Lemma \ref{lemma:WarpedProjectionIncreaseLength}, $\pi|_M:M\to H$ is a local diffeomorphism. Since $M$ is closed, and since $\pi$ is proper, $\pi|_M$ is also proper, so that $\pi|_M$ is a covering map. Since $H$ is simply connected and $M$ is connected, it follows that $\pi|_M$ is a homeomorphism, so that $M$ is the graph of some smooth function $\varphi:\Upp^p\rightarrow\S^q$. Since the condition of being spacelike only depends on the conformal class of the pseudo-Riemannian metric, it follows by \eqref{eq:WarpedMetric1} that $\varphi$ satisfies $\|d_x\varphi\|<1$ for every $x\in H$. This shows that $(1)$ implies $(4)$, as asserted.

In a similar manner, we show that the graph of any smooth function $\varphi$ satisying $\|d\varphi\|<1$ is spacelike, and $(3)$ therefore implies $(2)$. Since $(4)$ trivially implies $(3)$ and $(2)$ trivially implies $(1)$, this completes the proof.
\end{proof}
Bearing in mind Lemma \ref{lemma:EntireGraphsAreGraphs}, we define \emph{smooth entire graphs} in $\H_+$ as follows.

\begin{definition} \label{defi smooth entire graph} A \emph{smooth entire graph in} $\H_+$ is a subset satisfying any (and hence all) of the conditions of Lemma \ref{lemma:EntireGraphsAreGraphs}.
\end{definition}

\begin{lemma}\label{lemma:Acausal0}
Let $M$ be a smooth entire graph in $\H_+$. If $x_1,x_2\in M$, and if, for each $i$, $\hat x_i$ denotes the representative of $x_i$ in $\Quad$, then
\begin{equation*}
\q( \hat x_1,\hat x_2)\leq -1~,
\end{equation*}
with equality holding if and only if $x_1=x_2$.
\end{lemma}
\begin{proof}
Let $\hat{\Phi}:\Ball^p\times\S^q\rightarrow\Quad$ be a Fermi parametrisation such that $\hat x_1=\hat{\Phi}(0,N)$, where $N$ denotes the north pole of $\S^q$. Let $(u,w)\in\Ball^p\times \S^q$ be such that $\hat x_2=\hat{\Phi}(u,w)$. By \eqref{eqn:FermiParametrisationOfQuadricI},
\begin{equation*}
\q(\hat{x}_1,\hat{x}_2)
=-\frac{1+\|u\|^2}{1-\|u\|^2}\langle N,w\rangle_W
=-\frac{1+\|u\|^2}{1-\|u\|^2}\opCos\big(d_{\S}(N,w)\big)~,
\end{equation*}
where $\langle\cdot,\cdot\rangle_W=-\q|_W$, which, we recall, is positive-definite, and $d_{\S}$ here denotes the spherical distance in $\S^q$. Let $d_{\S}$ also denote the spherical distance in $\Ball^p$ induced by the stereographic projection. Since $M$ is the graph of a smooth, strictly $1$-Lipschitz function, and since $(-\opCos)$ is increasing,
\begin{equation*}
\q(\hat{x}_1,\hat{x}_2)\leq-\frac{1+\|u\|^2}{1-\|u\|^2}\opCos\big(d_{\S}(0,u)\big)\ ,
\end{equation*}
with equality holding if and only if $u=0$. However, we readily verify that
\begin{equation*}
\opCos\big(d_{\S}(0,u)\big) = \frac{1-\|u\|^2}{1+\|u\|^2}\ ,
\end{equation*}
so that $\q(\hat{x}_1,\hat{x}_2)\leq-1$, as desired. Finally, since $M$ is a strictly $1$-Lipschitz graph, equality holds if and only if $x_1=x_2$, and this completes the proof.
\end{proof}

\begin{remark}
Lemma \ref{lemma:Acausal0} implies that two points $x_1,x_2$ on a smooth entire graph $M$ are connected by a spacelike geodesic segment in $\H_+$. See also  \cite[Lemma 3.7]{CTT}.
\end{remark}

\begin{corollary}\label{cor:two lifts}
Let $M$ be a smooth entire graph in $\H_+$. If $\iota$ denotes the covering involution $\H_+ \to \H$, then $\iota(M)$ is disjoint from $M$.
\end{corollary}

\begin{proof}
By Lemma \ref{lemma:Acausal0}, given any two points $x_1,x_2\in M$, $\q( \hat x_1,\hat x_2) \leq  -1$. Hence $\q( \hat x_1,\iota(\hat x_2))=-\q( \hat x_1,\hat x_2)\geq 1$. By Lemma \ref{lemma:Acausal0} again, $\iota(\hat x_2)$ is not in $M$, and the result follows.
\end{proof}

\noindent This allows us to define smooth entire graphs in $\H$.

\begin{definition}
A \emph{smooth entire graph in} $\H$ is a (connected) subset whose every lift to $\H_+$ is a smooth entire graph.
\end{definition}

\noindent By Corollary \ref{cor:two lifts}, a smooth entire graph in $\H$ has two lifts to $\H$, which are smooth entire graphs differing from each other only by the covering involution. We now define
\begin{equation*}
\eqalign{
\mathcal E &:= \{\text{smooth entire graphs in }\H \}~,\ \text{and}\cr
\mathcal E_+ &:= \{\text{smooth entire graphs in }\H_+ \}~.\cr}
\end{equation*}
We furnish these spaces with their Hausdorff topologies. It follows from the above discussion that the natural map from $\H_+$ to $\H$ induces a double cover $\mathcal E_+\to\mathcal E$.

The following lemma relates the condition of being a smooth entire graph to the completeness of the induced metric.

\begin{lemma}\label{lemma:EntireGraph}
Let $M$ be a connected, $p$-dimensional spacelike submanifold of $\H$ (or $\H_+$).
\begin{enumerate}
\item If the induced metric on $M$ is complete, then $M$ is a smooth entire graph.
\item If $M$ is a smooth entire graph with bounded second fundamental form, then its induced metric is complete.
\end{enumerate}
\end{lemma}

\begin{proof}
It suffices to work with submanifolds of $\H_+$. To prove the first assertion, let $\pi$ be the Fermi projection associated to some marked hyperbolic $p$-space $(x,H)$. By by Item $(2)$ of Lemma \ref{lemma:WarpedProjectionIncreaseLength}, since $M$ is spacelike, $\pi$ is a distance-increasing local diffeomorphism. Since $M$ is complete, $\pi$ therefore has the path lifting property, and is thus a covering map. Since $H$ is simply connected, $\pi$ is a diffeomorphism, so that, by Lemma \ref{lemma:EntireGraphsAreGraphs}, $M$ is an entire graph. The second item is proven in the case of surfaces in \cite[Corollary 3.30]{LTW}, and the general case is addressed in an analogous manner. This completes the proof.
\end{proof}

The following theorem due to Ishihara shows that, among \emph{maximal} entire graphs, completeness of the induced metric and boundedness of the second fundamental form are in fact equivalent. This will form the basis of the compactness result that we will prove in Section \ref{sec:proper}.

\begin{theorem}[\cite{ishi}]\label{theorem:Ishihara}
If $M$ is a complete maximal $p$-dimensional submanifold of $\H$, then the norm of its second fundamental form is everywhere bounded above by $pq$.
\end{theorem}

It will be convenient to extend the notion of smooth entire graphs to subsets which are not necessarily smooth. For this purpose, observe that the closure of $\mathcal{E}_+$ in the Hausdorff topology consists of those subsets of $\H_+$ which are represented in some Fermi chart - and thus, by Lemma \ref{lemma:EntireGraphsAreGraphs}, in any Fermi chart - as the graphs of $1$-Lipschitz functions.

\begin{definition} We call a subset $M$ of $\H_+$ an \emph{entire graph} whenever it is the graph of a $1$-Lipschitz function in some, and thus in every, Fermi chart. We call a subset $M$ of $\H$ an \emph{entire graph} whenever it lifts to an entire graph $M_+$ in $\H_+$.
\end{definition}

\begin{remark}
As before, every entire graph in $\H$ has precisely two disjoint lifts in $\H_+$.
\end{remark}

Let $\overline{\mathcal E}$ (respectively $\overline{\mathcal E}_+$) denote the space of entire graphs in $\H$ (respectively $\H_+$). Let $\Lip_1(\Upp^p,\S^q)$ denote the space of $1$-Lipschitz functions from $\Upp^p$ into $\S^q$, furnished with the topology of uniform convergence. Given a marked hyperbolic $p$-space $(x,H)$, let
\begin{equation}\label{eq:defi Gamma}
\Gamma_{(x,H)}:\Lip_1(\Upp^p,\S^q)\rightarrow \overline{\mathcal E}_+~,
\end{equation}
denote the map which sends a function $\varphi\in\Lip_1(\Upp^p,\S^q)$ to its graph in the Fermi chart of $(x,H)$. With the above definitions and topologies, $\Gamma_{(x,H)}$ trivially defines a homeomorphism from $\Lip_1(\Upp^p,\S^q)$ into $\overline{\mathcal E}_+$.

\subsection{Locally uniformly spacelike families}
We now recall the grassmannian bundle $\GHH$ introduced in Section \ref{ss:GrassmanianBundle}. Given an element $M\in\mathcal{E}_+$, we define its Gauss lift in $\GHH$ by
\begin{equation}
\mathcal G(M) = \left\{(x,\T_xM) \in \GHH\right\}~.
\end{equation}
We say that a subset $\mathcal{F}\subseteq\mathcal{E}_+$ is \emph{locally uniformly spacelike} whenever the set
\begin{equation}
\bigcup_{M\in \mathcal F} \mathcal G(M)\cap\oppr^{-1}(K)
\end{equation}
is bounded in $\GHH$ for every compact subset $K$ of $\H_+$.

\begin{lemma}
For subsets $\mathcal{F}$ of $\mathcal{E}_+$, the following are equivalent.
\begin{enumerate}
\item $\mathcal F$ is locally uniformly spacelike.
\item For any marked hyperbolic $p$-space $(x,H)$, and for all $R>0$, there exists $\epsilon>0$ such that, for all $\varphi\in\Gamma_{(x,H)}^{-1}(\mathcal{F})$,
\begin{equation*}
\Vert d\varphi_{\vert B(x,R)} \Vert \leq 1-\epsilon~,
\end{equation*}
    where $B(x,R)$ denotes the ball of radius $R$ about $x$ in $H$.
\item There exists a marked hyperbolic $p$-space $(x,H)$ with the property that, for all $R>0$ there exists $\epsilon>0$ such that, for all $\varphi\in\Gamma_{(x,H)}^{-1}(\mathcal{F})$,
\begin{equation*}
\Vert d\varphi_{\vert B(x,R)} \Vert \leq 1-\epsilon~,
\end{equation*}
    where $B(x,R)$ denotes the ball of radius $R$ about $x$ in $H$.
\end{enumerate}
\end{lemma}

\begin{proof}
First let $(x,H)$ be a marked hyperbolic $p$-space in $\H_+$. Let $J_1$ denote the space of $1$-jets of smooth functions over $H$, and let $J_1^1$ denote the subset consisting of those jets $(y,\varphi(y),d\varphi(y))$ such that $\Vert d_y\varphi\Vert<1$. Define $G:J_1^1(X)\rightarrow\GHH$ such that $G(y,\varphi(y),d\varphi(y))$ is the tangent space of the graph of $\varphi$ at $x$. Note that this function is a homeomorphism.

We now show that $(1)$ implies $(2)$. Indeed, suppose that $\mathcal{F}$ is uniformly spacelike, let $(x,H)$ be a marked hyperbolic $p$-space with Fermi projection $\pi$, and choose $R>0$. Since the closed ball $\overline B(x,R)$ in $H$ is compact, and since $\pi$ is proper, the set $K:=\pi^{-1}(\overline B(x,R))$ is also compact in $\H_+$. Since $\mathcal{F}$ is locally uniformly spacelike, the set
\begin{equation*}
A:=\left\{(x,\T_xM)\ |\ M\in \mathcal F~,~x\in K\right\}
\end{equation*}
then has compact closure in $\GHH$. Since $G$ is a homeomorphism, the set $\hat A:=G^{-1}(A)$ likewise has compact closure in $J_1^1$. However, for all $\varphi\in\Gamma_{(x,H)}^{-1}(\mathcal{F})$ and, for all $y\in B(x,R)$, the $1$-jet of $\varphi$ at $y$ is an element of $\hat{A}$. In particular,
\begin{equation}\label{equation:SupOfNormIsBounded}
\sup\left\{ \Vert d_y\varphi\Vert\ |\ \varphi\in\Gamma_{(x,H)}^{-1}(\mathcal{F})~,~y\in \overline B(x,R)\right\}
\leq\sup\left\{ \|d\varphi(y)\|\ |\ (y,\varphi(y),d\varphi(y))\in\hat{A}\right\}<1~,
\end{equation}
as desired.

Note that $(2)$ trivially implies $(3)$. We now show that $(3)$ implies $(1)$. Indeed, using the same notation as before, suppose that \eqref{equation:SupOfNormIsBounded} holds. In particular, the set $\hat A$ has compact closure in $J_1^1$. Since $G$ is a homeomorphism, it follows that $A:=G(\hat A)$ has compact closure in $\GHH$. Since $A$ is the set of all tangent planes over $B(x,R)$ of elements of $\mathcal{F}$, it follows that $\mathcal{F}$ is locally uniformly spacelike, as desired.
\end{proof}

\begin{lemma}\label{cor:loc unif spacelike}
If $\mathcal F$ is a family of spacelike entire graphs in $\mathcal E_+$ such that
\begin{enumerate}
	\item the norms of the second fundamental forms of elements of $\mathcal{F}$ are uniformly bounded, and
    \item there exists a compact subset $K\subseteq\GHH$ which meets the Gauss lift of every element of $\mathcal{F}$,
\end{enumerate}
then $\mathcal F$ is locally uniformly spacelike.
\end{lemma}

\begin{proof}
Let $(x,H)$ be a marked hyperbolic $p$-space and let $\pi$ denote its Fermi projection. Given $M\in\mathcal{F}$, let $\pi_M$ denote the restriction of $\pi$ to $M$, and let $\mathcal{G}(M)$ denote its Gauss lift. By Lemma \ref{lemma:GaussLiftBilipschitz} and Lemma \ref{lemma:WarpedProjectionIncreaseLength} $(2)$, there exists $C>0$ such that
\begin{equation*}
g_{\mathcal G(M)} \leq (1+C^2) g_M \leq (1+C^2)\pi^*_Mg_H~,
\end{equation*}
where $g_M$ and $g_{\mathcal G(M)}$ denote the respective metrics of $M$ and $\mathcal{G}(M)$, and $g_H$ denotes the hyperbolic metric of $H$. Now let $L$ be a compact subset of $\H_+$ and choose $R>0$ such that
\begin{equation*}
\opDiam(\pi(L)\cap(\pi\circ\oppr)(K))<R.
\end{equation*}
By hypothesis, there exists $y'\in M$ such that $T_{y'}M\in K$. Let $y$ be an element of $L$. By definition of $R$, $d_H(\pi(y),\pi(y'))<R$ so that, by the above inequality,
\begin{equation*}
d_{\mathcal{G}(M)}(T_yM,K) \leq d_{\mathcal{G}(M)}(T_yM,T_{y'}M) \leq (1+C^2)R.
\end{equation*}
Since $M$ and $y$ are arbitrary, it follows that
\begin{equation*}
\bigcup_{M\in\mathcal{F}}\big\{\mathcal G(M)\cap\pi^{-1}(L)\big\} \subset B_{\GHH}\big(K,(1+C^2)R\big)~,
\end{equation*}
where $B_{\GHH}\big(K,(1+C^2)R\big)$ here denotes the ball in $\GHH$ of radius $(1+C^2)R$ about $K$. It follows that $\mathcal{F}$ is locally uniformly spacelike, as desired.
\end{proof}

\section{Non-negative spheres}\label{sec:boundaries}

We now study topologically embedded spheres in $\bH$ and $\bH_+$. We introduce the concept of admissibility for non-negative spheres in $\bH_+$, which will be the main geometric property that we will use to prove Theorem \ref{thm:introhomeo}. Using Fermi charts, we study the geometries of non-negative spheres, positive spheres and admissible non-negative spheres in $\bH_+$. We introduce the asymptotic boundary operator, which associates a non-negative sphere to every entire graph, we study its topological properties, and we describe the exceptional geometry of those entire graphs whose asymptotic boundaries are not admissible. Finally, we conclude by studying the convex hulls of non-negative spheres, and their relationship to complete maximal $p$-submanifolds.

\subsection{Non-negative spheres}\label{subsection:NonNegativeSpheres}

We begin by introducing the notions of positive and non-negative triples.

\begin{definition}\label{defi:positive triple}
A triple of distinct points $(x,y,z)$ in $\bH$ (or $\bH_+$) is called
\begin{enumerate}
	\item \emph{positive} whenever $x\oplus y\oplus z$ is $3$-dimensional with signature $(2,1)$;
	\item \emph{non-negative} whenever $x\oplus y\oplus z$ does not contain any negative-definite 2-plane.
\end{enumerate}
\end{definition}

\noindent This allows us to define positive and non-negative spheres.

\begin{definition}\label{defi:non negative sphere}
A subset $\Lambda$ of $\bH$ (or $\bH_+$) is called
\begin{enumerate}
\item a \emph{positive $(p-1)$-sphere} whenever it is homeomorphic to $\S^{p-1}$ and any triple of distinct points of $\Lambda$ is positive;
\item a \emph{non-negative $(p-1)$-sphere} whenever it is homeomorphic to $\S^{p-1}$, any triple of distinct points of $\Lambda$ is non-negative and, if $p=2$, $\Lambda$ contains at least one positive triple.
\end{enumerate}
\end{definition}

\begin{remark}
When $p=2$, non-negative spheres as in Definition \ref{defi:non negative sphere} are precisely the \emph{semi-positive loops} of \cite{LTW}.
\end{remark}

We now describe how non-negative spheres in $\bH$ lift to $\bH_+$.

\begin{definition}
We say that two points $x$ and $y$ of $\partial_\infty\H_+$ are \emph{antipodal} whenever they represent the same projective line with opposite orientations. We say that a non-negative $(p-1)$-sphere in $\bH_+$ is \emph{admissible} whenever it contains no pair of antipodal points.
\end{definition}

\begin{lemma}\label{lemma:LiftNonNegativeSpheres}
A subset $\Lambda$ of $\bH$ is a non-negative $(p-1)$-sphere if and only if any lift $\Lambda_+$ to $\bH_+$ is an admissible non-negative $(p-1)$-sphere.\end{lemma}
\begin{proof}
Suppose first $p>2$. Since $\Lambda$ is a $(p-1)$-sphere, it is simply connected. Since the projection $\bH_+\to\bH$ is a double cover, it follows that $\Lambda$ admits two disjoint lifts $\Lambda_\pm$ in $\partial_\infty\H_+$ which are non-negative. Since the covering involution of $\bH_+\to\bH$ is $x\mapsto -x$, it follows that neither lift contains antipodal points, so that $\Lambda_+$ and $-\Lambda_+$ are admissible. Conversely, an admissible $(p-1)$-sphere in $\bH_+$ projects injectively to a non-negative $p$-sphere in $\bH$, and this proves the result when $p>2$.

When $p=2$, it is shown in \cite[Lemma 2.8]{LTW} that the lift of any non-negative sphere in $\partial_\infty \mathbf{H}^{2,q}$ has two connected components in $\bH_+$, and the proof proceeds as before.
\end{proof}

We will mostly work with admissible non-negative spheres in the double cover $\bH_+$. We now describe how these objects are characterized in Fermi charts. We first characterize non-negative and positive spheres. The latter will be used in Section \ref{sec:anosov} to study the applications of Theorem \ref{thm:introhomeo} to the study of Anosov representations.

\begin{lemma}\label{lemma:NonNegativeSpheres}
Let $\Lambda_+$ be a subset of $\bH_+$. The following assertions are equivalent.
\begin{enumerate}
\item $\Lambda_+$ is a non-negative $(p-1)$-sphere.
\item $\Lambda_+$ is homeomorphic to $\S^{p-1}$ and, for every pair of points $x,y$ in $\Lambda_+$ and every choice of representatives $\hat x,\hat y$, we have $\q( \hat x,\hat y) \leq 0$.
\item There exists a Fermi chart in which $\Lambda_+$ is the graph of a $1$-Lipschitz map $\varphi:\S^{p-1}\rightarrow\S^q$.
\item In every Fermi chart, $\Lambda_+$ is the graph of a $1$-Lipschitz map $\varphi:\S^{p-1}\rightarrow \S^q$.
\end{enumerate}
\end{lemma}

\begin{proof}
The fact that $(1)$ implies $(2)$ is a direct generalization of \cite[Corollary 2.10]{LTW}.

We now show that $(2)$ implies $(1)$. Let $(x,y,z)$ be a triple of distinct points of $\Lambda_+$, and let $V:=x\oplus y\oplus z$ denote the subspace that they generate. Let $\hat{x}$, $\hat{y}$ and $\hat{z}$ denote respectively non-zero elements of the oriented lines $x$, $y$ and $z$, and recall that the Gram matrix of $V$ is defined by
\begin{equation*}
G:=\begin{pmatrix}
\q(\hat{x},\hat{x})\hfill&\q(\hat{x},\hat{y})\hfill&\q(\hat{x},\hat{z})\hfill\\
\q(\hat{y},\hat{x})\hfill&\q(\hat{y},\hat{y})\hfill&\q(\hat{y},\hat{z})\hfill\\
\q(\hat{z},\hat{x})\hfill&\q(\hat{z},\hat{y})\hfill&\q(\hat{z},\hat{z})\hfill
\end{pmatrix}\ .
\end{equation*}
Since $\hat{x}$, $\hat{y}$ and $\hat{z}$ are lightlike,
\begin{equation}\label{eqn:DetGram}
\opDet(G) = 2\q(\hat{x},\hat{y})\q(\hat{y},\hat{z})\q(\hat{x},\hat{z})\leq 0\ .
\end{equation}
We claim that this implies that $(x,y,z)$ is non-negative. Indeed, suppose the contrary, and that $V$ contains a negative-definite $2$-plane. By \eqref{eqn:DetGram}, $V$ is either negative-definite or null. However, in the former case, $V$ would contain no lightlike subspace, which is absurd, and in the latter, it would contain exactly one such subspace, which is also absurd. It follows that $(x,y,z)$ is non-negative, as asserted, and this proves that $(2)$ implies $(1)$.

We now prove that $(2)$ implies $(4)$. Let $\Lambda_+$ be a non-negative sphere and let $(x,H)$ be a marked hyperbolic $p$-space, with associated  Fermi projection $\pi$. In the notation of Lemma \ref{lemma:warped parameterization}, the inner product of any two points $\hat x:=\phi(u_1,w_1)$ and $\hat y:=\phi(u_2,w_2)$ in $\bH_+\cong\S^{p-1}\times\S^q$ is given by
\begin{equation}\label{eq:scalar}
\q( \hat x,\hat y) = \langle u_1,u_2\rangle_U - \langle w_1,w_2\rangle_W\ ,
\end{equation}
where here $\langle\cdot,\cdot\rangle_U$ and $\langle\cdot,\cdot\rangle_W$ denote the respective restrictions of $\q$ to $U$ and of $-\q$ to $W$, defined as in Section \ref{subsection:FermiCoordinates}.

Now, assume $\hat x$ and $\hat y$ are distinct elements of $\Lambda_+$. By hypothesis, $\langle w_1,w_2\rangle_{W}\geq \langle u_1,u_2\rangle_U$. However, denoting the spherical distances in $\S^{p-1}$ and $\S^q$ both by $d_S$,
\begin{equation}\label{eq:spherical distance and products}
\eqalign{
\langle u_1,u_2\rangle_U &= \cos(d_S(u_1,u_2))\ ,\ \text{and}\cr
\langle w_1,w_2\rangle_{W} &= \cos(d_S(w_1,w_2))\ .\cr}
\end{equation}
Thus, since $\cos$ is decreasing on $[0,\pi]$, $d_S(u_1,u_2)\geq d_S(w_1,w_2)$. In particular, $\pi|_{\Lambda_+}$ is injective, and, by the conservation of domain, its image is open. Since $\S^p$ is compact, its image is also compact, and thus closed so that, by connectedness, $\pi(\Lambda_+)$ coincides with the whole of $\S^p$. This shows that $\Lambda_+$ is the graph of some $1$-Lipschitz function $\varphi:\S^{p-1}\rightarrow\S^q$ and $(2)$ implies $(4)$, as desired.

Using again \eqref{eq:scalar} and \eqref{eq:spherical distance and products}, we readily see that $(3)$ implies $(2)$. Since $(4)$ trivially implies $(3)$, this completes the proof.
\end{proof}

Recall that a map $f$ between metric spaces is contractive whenever it satisfies $d(f(x_1),f(x_2))<d(x_1,x_2)$ for any pair $(x_1,x_2)$ of distinct points. Using \cite[Lemma 3.2]{DGK} and repeating the proof of Lemma \ref{lemma:NonNegativeSpheres} replacing all inequalities with strict inequalities, we obtain the following result.

\begin{lemma}\label{lemma:PositiveSpheres}
Let $\Lambda_+$ be a subset of $\bH_+$. The following assertions are equivalent.
\begin{enumerate}
\item $\Lambda_+$ is a positive $(p-1)$-sphere.
\item $\Lambda_+$ is homeomorphic to $\S^{p-1}$ and, for every pair of distinct points $x,y$ in $\Lambda_+$, and for every choice of representatives $\hat x,\hat y$, we have $\q( \hat x,\hat y) < 0$.
\item There exists a Fermi chart in which $\Lambda_+$ is the graph of a contractive map $\varphi:\S^{p-1}\rightarrow\S^q$.
\item In every Fermi chart, $\Lambda_+$ is the graph of a contractive map $\varphi:\S^{p-1}\rightarrow \S^q$.
\end{enumerate}
\end{lemma}

\begin{remark}\label{remark:PositiveImpliesAdmissable}
It follows from Item (2) of the above lemma that a positive sphere is an admissible non-negative sphere.
\end{remark}
We now have the following characterisation of admissible non-negative spheres.

\begin{lemma}\label{lemma:ProperNonNegativeSpheres}\label{lemma:CharacterisationOfAdmissable}
Let $\Lambda_+$ be a subset of $\bH_+$. The following assertions are equivalent.
\begin{enumerate}
\item $\Lambda_+$ is an admissible non-negative $(p-1)$-sphere.
\item $\Lambda_+$ is a non-negative $(p-1)$-sphere not contained in the orthogonal complement of any point of $\bH$.
\item There exists a Fermi chart in which $\Lambda_+$ is the graph of a $1$-Lipschitz map $\varphi:\S^{p-1}\rightarrow\S^q$ whose image does not contain antipodal points.
\item In every Fermi chart, $\Lambda_+$ is the graph of a $1$-Lipschitz map $\varphi:\S^{p-1}\rightarrow \S^q$ whose image does not contain antipodal points.
\end{enumerate}
\end{lemma}

\begin{proof}
To show that (2) implies (1), we prove the contrapositive. Let $\Lambda_+$ be a non-negative sphere which is not admissible. Let $\hat x$ and $-\hat x$ denote representatives of a pair of antipodal points that it contains. Let $\hat y$ denote the representative of any other point of $\Lambda_+$. By  Lemma \ref{lemma:NonNegativeSpheres}, $\q( \hat x,\hat y) \leq 0$ and $\q( (-\hat x),\hat y) \leq 0$, so that $\q( \hat x,\hat y)=0$. It follows that $\Lambda_+$ is contained in $x^\bot$, as desired.

We now show that (1) implies (4). By Lemma \ref{lemma:NonNegativeSpheres}, in every Fermi chart, $\Lambda_+$ is the graph of a $1$-Lipschitz map $\varphi:\S^{p-1}\to\S^q$. Suppose now that $\varphi(\S^{p-1})$ contains a pair $(w,-w)=(\varphi(u),\varphi(u'))$ of antipodal points. Since $\varphi$ is 1-Lipschitz, $d_{S}(w,-w)=\pi$ and $d_{S}(u,u')\leq \pi$, it follows that $d_{S}(u,u')= \pi$, so that $u'=-u$ and these points are therefore antipodal. It follows that $\Lambda_+$ itself contains the antipodal pair $(u,w),(-u,-w)$, and is therefore not admissible.

Since (4) trivially implies (3), it remains only to show that (3) implies (2). By Lemma \ref{lemma:NonNegativeSpheres}, $\Lambda_+$ is non-negative. Suppose now that $\Lambda$ is contained in $l^\perp$ for some $l:=\phi(u_0,w_0)$, where $\phi$ is as in Lemma \ref{lemma:warped parameterization}. By \eqref{eq:scalar},
\begin{equation*}
0=\langle u_0,u_0\rangle_U-\langle w_0,\varphi(u_0)\rangle_W=1-\langle w_0,\varphi(u_0)\rangle_W\ ,
\end{equation*}
so that $\langle w_0,\varphi(u_0)\rangle_W=1$, and thus $\varphi(u_0)=w_0$. Likewise
\begin{equation*}
0=\langle u_0,(-u_0)\rangle_U-\langle w_0,\varphi(-u_0)\rangle_W=-1-\langle w_0,\varphi(-u_0)\rangle_W\ ,
\end{equation*}
so that $\langle w_0,\varphi(-u_0)\rangle_W=-1$, and thus $\varphi(-u_0)=-w_0$. The image of $\varphi$ therefore contains the antipodal points $\pm w_0$, which is absurd, and it follows that $(3)$ implies $(2)$, as desired.
\end{proof}

\subsection{Asymptotic boundaries}\label{subsection:AsymptoticBoundary}

Define
\begin{equation}
\eqalign{
\mathcal{B}  &:=  \left\{ \text{non-negative $(p-1)$-spheres in } \bH \right\}~,\ \text{and}\cr
\mathcal{B}_+  &:=  \left\{ \text{admissible non-negative $(p-1)$-spheres in } \bH_+ \right\}~.\cr}
\end{equation}
We furnish these spaces with the Hausdorff topology. By Lemma \ref{lemma:LiftNonNegativeSpheres} the projection $\bH_+ \to \bH$ induces a two-to-one map $\mathcal B_+ \to \mathcal B$.

The Hausdorff closure of  $\mathcal B_+$ in the space of closed subsets of $\bH_+$ is the space $\overline{\mathcal B}_+$ of non-negative spheres.
We define the \emph{asymptotic boundary operator}
\begin{equation}
\partial_\infty^+: ~ \overline{\mathcal E}_+ \rightarrow \overline{\mathcal B}_+
\end{equation}
as follows. Let $M\subset\H_+$ be an entire graph, let $(x,H)$ be a marked hyperbolic $p$-space, let $\Gamma_{(x,H)}$ be as in equation (\ref{eq:defi Gamma}), denote $\varphi:=\Gamma_{(x,H)}^{-1}(M)$, and note that $\varphi$ extends uniquely to a $1$-Lipschitz function $\overline{\varphi}:\overline{\S}_+^p\rightarrow\S^q$. We define
\begin{equation}
\partial^+_\infty M := \overline{\varphi}(\S^{p-1})\ .
\end{equation}
This trivially does not depend on the marked hyperbolic $p$-space chosen and, since $\varphi$ depends continuously on $M$, so too does $\partial^+_\infty M$.

\begin{lemma}\label{lemma:AdmissableBoundaryCriteria}
Let $M$ be an entire graph in $\H_+$. The following assertions are equivalent.
\begin{enumerate}
\item $\partial^+_\infty M$ is not admissible.
\item $M$ contains a complete, lightlike geodesic.
\item $M$ is foliated by complete, lightlike geodesics.
\item In every Fermi chart, $M$ is the graph of a $1$-Lipschitz map $\varphi:\Upp^p\to\S^q$ of the form
\begin{equation}\label{eq:entire graphs lightlike foliated}
\varphi(\sin(t) u_0+\cos(t) v)=\sin(t) w_0+\cos(t) \overline\varphi(v)
\end{equation}
for some $u_0\in\S^{p-1}$, $w_0\in \S^q$ and some $1$-Lipschitz function $\overline\varphi$ from the hemisphere $u_0^\bot\cap \Upp^p$ to the $(q-1)$-sphere $w_0^\bot\cap \S^q$.
\end{enumerate}
\end{lemma}

\begin{proof}We first prove that $(1)$ implies $(4)$. Suppose that $\partial_\infty M$ is not admissible. Let $(x,H)$ be a marked hyperbolic $p$-space, and denote $\varphi:=\Gamma_{(x,H)}^{-1}(M)$. Since $\partial_\infty M$ is not admissible, it contains a pair of antipodal points. In the Fermi chart of $(x,H)$, this pair has the form $\pm(u_0,w_0)\in\S^{p-1}\times\S^q$. Let $\gamma:(-\pi/2,\pi/2)\rightarrow\overline{\S}{}_+^p$ be a unit speed geodesic such that $\gamma(\pm\pi/2)=\pm u_0$, and note that
\begin{equation*}
\gamma(t) = \sin(t)u_0 + \cos(t)v\ ,
\end{equation*}
for some $v\in\Upp^{p-1}$. Since $\varphi$ is $1$-Lipschitz
\begin{equation*}
\pi = d(-w_0,w_0) \leq \opL[\varphi\circ\gamma] \leq \opL[\gamma] = d(-u_0,u_0) = \pi\ ,
\end{equation*}
so that $(\varphi\circ\gamma)$ is also a unit speed geodesic, and thus has the form
\begin{equation*}
(\varphi\circ\gamma)(t) = \sin(t)w_0 + \cos(t)\overline{\varphi}(v)\ ,
\end{equation*}
for some $\overline{\varphi}(v)\in\S^{q-1}$. Since $\gamma$ is arbitrary, $\varphi$ satisfies \eqref{eq:entire graphs lightlike foliated}. Furthermore, upon restricting to the orthogonal complement of $u_0$, we see that $\overline{\varphi}$ is $1$-Lipschitz, and this proves $(4)$.

We now show that $(4)$ implies $(3)$. Indeed, with $\gamma$ as above, $\overline{\gamma}(t):=(\gamma(t),(\varphi\circ\gamma)(t))$ is a clearly a lightlike geodesic for the conformal metric $g_{\Upp^p}-g_{\S^q}$, with endpoints $\pm(u_0,v_0)$. Since $\H$ is geodesically complete, and since unparameterized lightlike geodesics only depend on the conformal class of the pseudo-riemannian metric (see \cite[Proposition 2.131]{GHL}), $\overline\gamma$ is a complete lightlike geodesic for the metric $\g$, and $(3)$ follows. Finally, $(3)$ trivially implies $(2)$, and, since the endpoints of a complete lightlike geodesic are antipodal, $(2)$ implies $(1)$. This completes the proof.
\end{proof}

\begin{corollary}\label{cor:BoundaryOperatorIsContinuous}
The asymptotic boundary operator $\partial_\infty^+$ restricts to a continuous map from $\mathcal E_+$ to  $\mathcal B_+$, and induces a continuous map $\partial_\infty:\mathcal E\to\mathcal B$ such that the following diagram commutes.
$$\begin{tikzcd}
\mathcal E_+\arrow{r}{\partial_\infty^+}\arrow{d}{} &
\mathcal B_+\arrow{d}{} \\
\mathcal E\arrow{r}{\partial_\infty} &
\mathcal B \\
\end{tikzcd}~.$$
\end{corollary}
\begin{proof}
Indeed, by Definition \ref{defi smooth entire graph}  and Lemma \ref{lemma:AdmissableBoundaryCriteria}, if $M\in\mathcal E_+$, then $\partial_\infty^+(M)$ is admissible. The result now follows by Lemma \ref{lemma:LiftNonNegativeSpheres}.
\end{proof}

\subsection{Convex hulls}

Given a subset $\Lambda$ in $\P_+(E)$, its \emph{convex hull} $\Conv(\Lambda)$ is the intersection of all closed half-spaces containing $\Lambda$. Observe that, if $\Lambda$ is a non-negative set, then it is contained in at least one closed half-space: indeed, for any $a\in\Lambda$, one has $\q( a,b)\leq 0$ for all $b\in\Lambda$, and $\Lambda$ is thus contained in the half-space defined by $\q( a,\cdot)\leq 0$. Furthermore, a suitable adaptation of the proof of \cite[Proposition 2.15, Item (vi)]{LTW} shows that $\Conv(\Lambda)$ is a subset of $\H_+\cup\partial_\infty\H_+$.

\begin{lemma}\label{lemma:ConvIsContinuous}
$\opConv$ is continuous with respect to the Hausdorff topology of $\overline{\mathcal{B}}_+$ and the Hausdorff topology of the space of closed, convex subsets of $\P_+(E)$.
\end{lemma}

\begin{proof}
Indeed, upon taking a projective chart, we identify $\H_+$ with the interior of the quadric $\Quad$ in $\Rpqp$. With this identification, the convex hull of any subset of $\H_+\cup\partial_\infty \H_+$ identifies with its convex hull in $\Rpqp$. The result now follows by the continuity of convex hulls in $\R^n$, see  \cite[Lemma 2.1]{zbMATH07444885} or \cite{zbMATH04113268}.
\end{proof}

\noindent The following result is proven in \cite[Proposition 2.15, Item (vi)]{LTW} in the case $p=2$, but the proof extends to our case.

\begin{lemma}\label{lemma:DeltaConvIsDelta}
For every non-negative sphere in $\bH_+$, the intersection between $\opConv(\Lambda)$ and $\bH_+$ coincides with $\Lambda$.
\end{lemma}

\noindent We also have

\begin{lemma}\label{lemma:ContainedInConvexHull}
If $M$ is a complete maximal submanifold of $\H_+$, then $M$ is contained in $\Conv(\partial_\infty M)$.
\end{lemma}

\begin{proof}
Let $M$ be a complete maximal submanifold of $\H_+$, and let $\varphi \in E^*$ be a linear form on $E$. Identifying $\H_+$ with the quadric $\Quad$, and decomposing the trivial covariant derivative $D$ on $E$ along $M \subset \Quad \subset E$, we have for any vector fields $X,Y$ over $M$ and for any point $x\in M$,
\begin{equation}\label{eq:decomposition flat connection}
(D_XY)_x = (\nabla^\T_XY)_x + \II(X,Y)_x + \q( X,Y)x ~,
\end{equation}
where $\nabla^\T$ denotes the Levi-Civita covariant derivative over $M$, and $\II$ denotes its second fundamental form. In particular, for any  $u\in \T_x M$ and for any geodesic $\gamma(t)$ in $M$ with $\overset{\bullet}{\gamma}(0)=u$,
\begin{equation*}
\text{Hess}^M_x \varphi(u,u) = \left.\frac{d^2}{dt^2}\right\vert_{t=0} \varphi(\gamma(t)) = \varphi((D_u\overset{\bullet}{\gamma})_x) = \varphi(\II(u,u)_x) + \q(u,u)\varphi(x)~.
\end{equation*}
Taking the trace with respect to $g_\I$ yields
\begin{equation*}
\Delta \varphi = p \varphi~,
\end{equation*}
where $\Delta$ is the Laplace-Beltrami operator on $M$. It now follows by the maximum principle that if $\varphi$ is positive (respectively negative) on $\partial_\infty M$, then it must also be positive (respectively negative) on $M$, and the result follows.
\end{proof}

\section{Compactness}\label{sec:proper}

We now begin our study of complete {maximal} spacelike submanifolds of $\H$. Recall that, by Lemma \ref{lemma:EntireGraph}, every complete maximal $p$-dimensional submanifold is an entire graph. We thus define
\begin{equation}
\eqalign{
\mathcal{M} &:= \{ \text{complete maximal $p$-dimensional submanifolds of }\H\} \subset \mathcal E~,\ \text{and}\cr
\mathcal{M}_+ &:= \{ \text{complete maximal $p$-dimensional submanifolds of }\H_+\} \subset \mathcal E_+~,\cr}
\end{equation}
and we equip these subspaces with the topologies inherited from $\mathcal E$ and $\mathcal E_+$ respectively. The aim of this section is to prove the following result.

\begin{theorem}\label{thm:degenerate}\label{theorem:DegenerationOfMaximalSubmanifolds}
Let $\seq[n]{M_n}$ be a sequence in $\M_+$. Either
\begin{enumerate}
\item $\seq[n]{M_n}$ subconverges in the $C^\infty_\oploc$ topology to a complete maximal $p$-submanifold of $\H_+$, or
\item $\seq[n]{M_n}$ subconverges in the Hausdorff topology to a Lipschitz $p$-submanifold foliated by complete, lightlike geodesics, all having the same endpoints at infinity.
\end{enumerate}
\end{theorem}

\noindent As a consequence, we will derive the following properness result.

\begin{corollary}\label{theorem:Properness}
The boundary maps $\partial_\infty:\M\rightarrow\Bd$ and $\partial_\infty^+:\M_+\rightarrow\Bd_+$ are proper.
\end{corollary}

\subsection{Compactness I - smooth limits}

Let $\hat{\mathcal{M}}$ denote the space of all pairs of the form $(x,M)$, where $M$ is in $\M$ and $x$ is a point of $M$. We equip this space with the topology that it inherits as a subset of $\H\times\M$. We define the continuous maps $\tau:\hat{\mathcal{M}}\rightarrow\GH$ by
\begin{equation}
\tau(x,M) := \T_xM\ .
\end{equation}
We define the space $\hat{\mathcal{M}}_+$ and the map $\tau_+:\hat{\mathcal{M}}_+\rightarrow\GHH$ in a similar manner.

\begin{theorem}\label{theorem:PropernessOfTau}
The maps $\tau:\hat{\mathcal{M}}\rightarrow\GH$ and $\tau_+:\hat{\mathcal{M}}_+\rightarrow\GHH$ are proper.
\end{theorem}

In the context of this section, elliptic regularity is expressed as follows.

\begin{lemma}\label{lemma:EllipticRegularity}
Let $\seqn{M_n}$ be a sequence of maximal entire graphs in $\H_+$. If this sequence converges in the $C^{1,\alpha}_\oploc$ sense to a spacelike entire graph $M_\infty$, say, then $M_\infty$ is also maximal, and $\seqn{M_n}$ converges to $M_\infty$ in the $C^\infty_\oploc$ sense.
\end{lemma}
\begin{proof} Let $x_0$ be a point of $M_\infty$. Choose Fermi coordinates associated to the marked hyperbolic $p$-space $(x_0,H)$ where $\T_{x_0}M_\infty=\T_{x_0}H$. By Lemma \ref{lemma:EntireGraphsAreGraphs}, $M_\infty$ is the graph of a smooth $1$-Lipschitz function $f_\infty:\Upp^p\to\S^q$. Using stereographic projection, we identify $\Upp^p$ with $\Ball^p$ and the complement of a point in $\S^q$ with $\mathbf{R}^q$, in such a manner that $x_0$ identifies with $(0,0)$. Since $f_\infty$ is $1$-Lipschitz, its image is contained in a hemisphere (see, for example, \cite[Lemma 2.8]{heinonen}), and we thus view $f_\infty$ as a function from $\Ball^p$ to $\Ball^q$. For large $n$, $M_n$ thus also identifies with the graph of a smooth function $f_n:\Ball^p\to\Ball^q$.

Recall now that maximality is a quasi-linear property. That is, for all $n$, and for all $x\in B^p(0,\epsilon)$,
\begin{equation}\label{equation:QLEquation}
\sum_{i,j} a_n^{ij}(x)(\partial_i\partial_j f_n)(x) = b_n(x)\ ,
\end{equation}
where
\begin{equation}\label{equation:FormulaForAAndB}
\eqalign{
a_n^{ij}(x) &:= A(J^1f_n(x),J^1g(J^0f_n))\ ,\ \text{and}\cr
b_n(x) &:= B(J^1f_n(x))\ ,\cr}
\end{equation}
for some smooth functions $A$ and $B$, where $g$ here denotes the pseudo-hyperbolic metric over $\Ball^p\times\Ball^q$ and, for all $k$, $J^k$ denotes the $k$-jet operator. Differentiating \eqref{equation:QLEquation} yields, for all $k$, for all $n$, and for all $x$,
\begin{equation}\label{equation:QLEquationII}
\sum_{i,j} a_n^{ij}(x)(\partial_i\partial_j D^kf_n)(x) = b_{n,k}(x)\ ,
\end{equation}
where
\begin{equation}\label{equation:FormulaForBII}
b_{n,k}(x) :=  B_k(J^{k+1}f_n(x),J^{k+1}g(J^0f_n))\ ,
\end{equation}
for some smooth function $B_k$.

We now prove by induction that, for all $k$, there exists $C_k$ such that, for all $n$,
\begin{equation*}
\|f_n\|_{C^{k,\alpha}(B(0,{\epsilon/2^{k-1}}))} \leq C_k\ .
\end{equation*}
Indeed, by hypothesis, the result holds for $k=1$. Suppose now that the result holds for $k=l$. Since $M_\infty$ is spacelike, and since $\seqm{f_n}$ converges in the $C^1_\oploc$ sense to $f_\infty$, there exists $B_1>0$ such that, for all $n$, over $B(0,\epsilon)$,
\begin{equation*}
\sum_{i,j} a_n^{ij} \geq \frac{1}{B_1}\delta^{ij}\ .
\end{equation*}
By \eqref{equation:FormulaForAAndB}, \eqref{equation:FormulaForBII} and the inductive hypothesis, there exists $B_2>0$ such that, for all $n,i,j$ we have
\begin{equation*}
\|a^{ij}_n\|_{C^{0,\alpha}(B(0,{\epsilon/2^{l-1}}))},\|b_{n,l-1}\|_{C^{0,\alpha}(B(0,{\epsilon/2^{l-1}}))}\leq B_2\ .
\end{equation*}
Upon applying the Schauder estimates to \eqref{equation:QLEquationII} (see Chapter $6$ of \cite{GilbTrud}), we see that there exists $B_3>0$ such that, for all $n$,
\begin{equation*}
\|f_n\|_{C^{l+1,\alpha}(B(0,{\epsilon/2^l}))} \leq B_3\ .
\end{equation*}
Thus, for all $k$, and for all $\beta<\alpha$, $\seqn{f_n}$ converges to $f_\infty$ in the $C^{k,\beta}$ sense over $B_{\epsilon/2^{k-1}}(0)$. Since $x_0$ is arbitrary, it follows that $\seqn{M_n}$ converges to $M_\infty$ in the $C^\infty_\oploc$ sense, and this completes the proof.
\end{proof}

\begin{proof}[Proof of Theorem \ref{theorem:PropernessOfTau}]
Up to taking lifts, it suffices to prove the statement for $\tau_+$. Let $\seqn{(x_n,M_n)}$ be a sequence of elements of $\hat{\mathcal{M}}_+$ such that $\seqn{\T_{x_n}M_n}$ is bounded. We work in the Fermi chart corresponding to some marked hyperbolic $p$-space $(x_0,H)$. Upon applying a (bounded) sequence of isometries of $\H_+$, we may suppose that, for all $n$, $x_n=x_0$ and $\T_{x_n}M_n=\T_{x_0}H$. For all $n$, let $f_n:\Upp^p\rightarrow\S^q$ denote the function whose graph is $M_n$. Since, for all $n$, $f_n$ is $1$-Lipschitz, upon extracting a subsequence if necessary, we may suppose that $\seqn{f_n}$ converges in the $C^{0,\alpha}$ sense for all $\alpha$ to the $1$-Lipschitz function $f_\infty$.

By Theorem \ref{theorem:Ishihara}, the second fundamental forms of the graphs of these functions are bounded above by $pq$. By Lemma \ref{cor:loc unif spacelike}, the sequence $\seqn{f_n}$ is locally uniformly spacelike. We claim that uniform bounds also hold for the second derivatives of these functions. Indeed, choose $R>0$, and let $x$ be a point of $B(0,R)$ in $\textbf{H}^p$. For all $n$, let $\alpha_{n,x}:\H_+\rightarrow\H_+$ be an isometry that sends $x_n:=(x,f_n(x))$ to $x_0$ and $\T_{x_n}M_n$ to $\T_{x_0}H$, and let $f_{n,x}'$ denote the function whose graph is $\alpha_{n,x}M_n$. For all $n$, since $D^2f_{n,x}'(0)$ is the second fundamental form of $\alpha_{n,x}M_n$ at $(0,0)$, by Theorem \ref{theorem:Ishihara}, its norm is uniformly bounded above by $pq$. Thus, since, for all $x$ and for all $n$, $D^2f_n(x)$ depends only on $D^2f_{n,x}'(0)$ and $\alpha_{n,x}$, and since the family $(\alpha_{n,x})_{n\in\mathbf{N},x\in B_R(0)}$ is precompact, $\|D^2f_n(x)\|$ is uniformly bounded over $B(0,R)$, as asserted.

By the Arzelà-Ascoli Theorem, $\seqn{f_n}$ converges to $f_\infty$ in the $C^{1,\alpha}_\oploc$ sense, for all $\alpha$, and $f_\infty$ is a $C^{1,\alpha}$ function with spacelike graph. It then follows by Lemma \ref{lemma:EllipticRegularity} that $f_\infty$ is smooth and maximal, and that $\seqn{f_n}$ converges towards this function in the $C^\infty_\oploc$ sense. Finally, since the second fundamental form of $M_n$ has norm bounded above by $pq$, the same holds for the limit $M_\infty$. It follows by Lemma \ref{lemma:EntireGraph} that $M_\infty$ is complete, and this completes the proof.
\end{proof}

\subsection{A dynamical interlude}

In order to complete the proof of Theorem \ref{thm:degenerate}, it is necessary to understand the structure of diverging sequences in $\mathsf{O}(p,q)$. This theory is well-known, and we recall the main ideas for the reader's convenience. Let $(F,B)$ be a $(p+q)$-dimensional vector space equipped with a signature $(p,q)$ bilinear form $B$, and let $\textsf{O}(p,q)$ denote the group of linear endomorphisms of $F$ preserving $B$. We prove the following lemma, that we will apply in Section \ref{sec:diverging sequences degenerate} in the situation where $F=\T_{x_0}\H_+$ for some $x_0\in\H_+$, and $B$ is the pseudo-riemannian metric of $\H_+$ on $\T_{x_0}\H_+$.

\begin{lemma}\label{lemma:DynamicalLemma}
Let $\seqn{g_n}$ be an unbounded sequence in $\mathsf{O}(p,q)$. After extracting a subsequence if necessary, there exists a sequence $\seqn{\mu_n}$ of positive numbers converging to $+\infty$ such that $\seqn{\exp(-\mu_n) g_n}$ converges to some linear map $\varphi\in\mathrm{End}(F)$ with degenerate kernel and totally isotropic image.
\end{lemma}

We first establish some notation. Let $F=U\oplus V$ be a $B$-orthogonal splitting such that $U$ has signature $(p,0)$, and denote
\begin{equation}
\theta:=\begin{pmatrix} \Id_U & 0 \\ 0 & -\Id_V \end{pmatrix}~,
\end{equation}
so that $B_\theta(x,y):=B(\theta(x),y)$ defines a positive-definite scalar product on $F$. The set of fixed points of the action of $\theta$ by conjugation defines a maximal compact subgroup $\K$ of $\mathsf{O}(p,q)$ which is isomorphic to $\mathsf{O}(p)\times \mathsf{O}(q)$.

Let $(u_1,\cdots ,u_p)$ and $(v_1,\cdots ,v_q)$ be respectively orthonormal bases of $U$ and $V$ with respect to $B_\theta$. Let $m_0:=\min\{p,q\}$. For each $i \in \{1,\cdots,m_0\}$ denote
\begin{equation}
E_i = \span\{u_i+v_i\} ~\ ,\ ~E_i^\vee = \theta(E_i)~\ , \ ~ W= \bigg(\bigoplus_{i=1}^{m_0}  E_i\oplus E_i^\vee \bigg)^\bot~.
\end{equation}
This yields the decomposition
\begin{equation}\label{eq:decomposition F}
F=E_1\oplus \cdots  \oplus E_{m_0} \oplus E_{m_0}^\vee \oplus\cdots  \oplus E_1^\vee \oplus W~.
\end{equation}
Note that
\begin{equation}
W=\begin{cases}
\span\{v_{p+1},\ldots,v_q\}\ \text{if}\ p<q\ ,\\
\span\{u_{q+1},\ldots,u_p\}\ \text{if}\ p>q\ ,\ \text{and}\\
\{0\}\ \text{otherwise}.
\end{cases}
\end{equation}
In particular, the restriction of $B$ to $W$ is definite: negative-definite if $p<q$, and positive-definite if $p>q$.

For $\lambda=(\lambda_1,\ldots,\lambda_{m_0}) \in \mathbf R^{{m_0}}$, let $a(\lambda)$ denote the diagonal matrix which, with respect to the decomposition \eqref{eq:decomposition F}, has coefficients $(\lambda_1,\cdots ,\lambda_{m_0},-\lambda_{m_0},\cdots ,-\lambda_1,0_W)$ . The \emph{Cartan subspace} associated to the above bases is the space of matrices of the form $a(\lambda)$, and its \emph{closed Weyl chamber} is
\begin{equation}
\overline{\frak{a}}^+ := \big\{a(\lambda)~\vert~ \lambda_1\geq \lambda_2\geq \cdots \geq \lambda_{m_0} \geq 0\big\}~.
\end{equation}
We are now ready to state the Cartan decomposition theorem (see \cite{Knapp}).

\begin{theorem}
For any $g\in\mathsf{O}(p,q)$, there exist $k,k'\in K$ and a unique $a(\lambda)\in \overline{\frak{a}}^+$ such that
\begin{equation}
g=k\exp(a(\lambda))k'.
\end{equation}
\end{theorem}

\noindent This now allows us to prove Lemma \ref{lemma:DynamicalLemma}.

\begin{proof}[Proof of Lemma \ref{lemma:DynamicalLemma}]
Let $g_n=:k_n \exp(a(\lambda_n)) k'_n$ be a Cartan decomposition of $g_n$, where here $\lambda_n:=(\lambda_{1,n},\cdots ,\lambda_{{m_0},n})$. Define $\mu_n:=\lambda_{1,n}$. Since $\seqn{g_n}$ is unbounded, upon extracting a subsequence if necessary, we may suppose that that $\seqn{\mu_n}$ tends to infinity. We may likewise suppose that, for any $i$, $\underset{n\to \infty}{ \lim} \exp(\lambda_{i,n}-\mu_n) = \alpha_i \in [0,1]$ and that the sequences $\seqn{k_n}$ and $\seqn{k'_n}$ converge to $k_\infty$ and $k'_\infty$ respectively. Since $1=\alpha_1\geq \cdots \geq \alpha_{m_0}\geq 0$,
\begin{equation*}
\underset{n\to \infty}{\lim}\left(\exp(-\mu_n) g_n\right) = k_\infty h_\infty k'_\infty~,
\end{equation*}
where $h_\infty$ is the diagonal matrix with coefficients $(1,\alpha_2,\cdots ,\alpha_{{m_0}},0,\cdots ,0)$. The result follows from the fact that the kernel of $h_\infty$ has the form $E_j\oplus\cdots \oplus E_{m_0}\oplus E_{m_0}^\vee\oplus\cdots \oplus E_1^\vee \oplus W$ and that its image is $E_1\oplus\cdots \oplus E_{j-1}$ where $j>1$ is the least integer for which $\alpha_j=0$.
\end{proof}

\subsection{Compactness II - degenerate limits}\label{sec:diverging sequences degenerate}

Theorem \ref{thm:degenerate} will now follow upon describing the behaviour of diverging sequences in $\mathcal{M}$.

\begin{proof}[Proof of Theorem \ref{thm:degenerate}]
We work in the double cover $\H_+$. The result for $\H$ then follows upon taking lifts. Given a sequence $\seqn{M_n}$ in $\M_+$, we will show that, up to extracting subsequences, either it converges in the $C^\infty_\oploc$ sense to an element of $\M_+$, or it converges to an entire graph of the form described in Lemma \ref{lemma:AdmissableBoundaryCriteria}.

We work in the Fermi chart of some marked hyperbolic $p$-space $(x,H)$. Let $\Gamma_{(x,H)}$ be as in equation (\ref{eq:defi Gamma}). For all $n$, denote $\varphi_n:=\Gamma_{(x,H)}^{-1}(M_n)$, and note that $\varphi_n$ is smooth for all $n$. Suppose that $\seqn{\varphi_n}$ does not subconverge in the $C^\infty_\oploc$ sense to a smooth function $f_\infty:\B^p\rightarrow\S^q$ whose graph is maximal. By Theorem \ref{theorem:PropernessOfTau}, for every $y\in H$, the sequence $\seqn{\T_{(y,\varphi_n(y))}M_n}$ is unbounded in $\GHH$. Since $\overline{\mathcal{E}}$ is compact, we may suppose that $\seqn{M_n}$ converges in the Hausdorff sense to some entire graph $M_\infty$, say.

We now show that $M_\infty$ contains a complete lightlike geodesic through every point. We continue to work in the above Fermi chart. Choose $y\in H$ and, for all $n$, denote $\hat{y}_n:=(y,\varphi_n(y))$. Note first that, since the fibre over $y$ in $\H_+$ is compact, upon applying a bounded sequence of isometries of $\H_+$, we may suppose that, for all $n$, $(y,\varphi_n(y))=x_0$. For all $n$, let $g_n\in\Stab_{\mathsf{O}(p,q+1)}(x_0)$ be such that $g_n\T_{x_0}H=\T_{x_0}M_n$. Since $\seqn{\T_{x_0}M_n}$ diverges, so too does $\seqn{g_n}$. By Lemma \ref{lemma:DynamicalLemma}, we may suppose that there exists a sequence $\seqn{\mu_n}$ of positive real numbers converging to $+\infty$ such that $\underset{n\to\infty}{\lim}(e^{-\mu_n}g_n)=g_\infty$ for some linear map $g_\infty$ with degenerate kernel and isotropic image.

For all $n$, denote $M_n':=g_n^{-1}M_n$. By Theorem \ref{theorem:PropernessOfTau}, $\seqm{M'_n}$ subconverges to a complete maximal $p$-dimensional submanifold $M'_\infty$, say. For all $n\in\mathbf{N}\cup\left\{\infty\right\}$, denote $S_n:=\exp_{x}^{-1}(M'_n)$. Then $\seqn{S_n}$ is a sequence of submanifolds of $\T_{x_0}\H_+$, passing through the origin, and tangent to $\T_{x_0} H$ at this point. Furthermore, this sequence converges to $S_\infty$ in the $C^\infty_\oploc$ sense.

For all $n\in\mathbf{N}\cup\left\{\infty\right\}$, let $\delta_n:(-\epsilon,\epsilon)\rightarrow S_n$ be a smoothly immersed curve such that $\delta_n(0)=0$. Suppose furthermore that the sequence $\seqn{\delta_n}$ converges in the $C^\infty_\oploc$ sense to $\delta_\infty$, and that the derivative at $t=0$ of  ${\delta}_\infty$ does not vanish. Since $\Ker(g_\infty)$ is degenerate, it has trivial intersection with $\T_x H$, and so, upon reducing $\epsilon$ if necessary, we may suppose that $(g_\infty\circ\delta_\infty)$ is a smoothly immersed curve in some isotropic subspace of $\T_{x_0}\H_+$.

For all $n$, we now denote
\begin{equation*}
\gamma_n(t):=(g_n\circ\exp_{x_0}\circ\delta_n)(e^{-\mu_n}t)=(\exp_{x_0}\circ g_n\circ \delta_n)(e^{-\mu_n}t)\ ,
\end{equation*}
where the last equality follows from the fact that $\exp_{x_0}\circ g_n = g_n\circ \exp_{x_0}$, since $g_n$ is an isometry of $\H_+$. The sequence $\seqn{\gamma_n}$ thus converges in the $C^\infty_\oploc$ sense to the complete lightlike geodesic
\begin{equation*}
\gamma_\infty(t):=\exp_{x_0}\big(t g_\infty(\overset{\bullet}{\delta}_\infty(0))\big)\ .
\end{equation*}
It follows that $M_\infty$ contains a complete lightlike geodesic passing through the point $y$, and this concludes the proof.
\end{proof}

We now prove Corollary \ref{theorem:Properness}.

\begin{proof}[Proof of Corollary \ref{theorem:Properness}]
Up to taking lifts, it suffices to prove the statement for $\partial_\infty^+:\M_+\rightarrow\Bd_+$. Suppose that $\partial_\infty^+:\M_+\rightarrow\Bd_+$ is not proper. Let $\seqn{\Lambda_n}$ be a sequence in $\Bd_+$ converging to $\Lambda_\infty\in\Bd_+$, say. For all $n$, let $M_n\in\M_+$ be such that $\partial_\infty M_n=\Lambda_n$ and suppose that the sequence $\seqn{M_n}$ does not subconverge to any element of $\M_+$. By Theorem \ref{theorem:DegenerationOfMaximalSubmanifolds}, up to extraction of a subsequence, this sequence converges to an entire graph $M_\infty$, say, containing a complete lightlike geodesic. By  Lemma \ref{lemma:AdmissableBoundaryCriteria}, $\Lambda_\infty=\partial_\infty M_\infty$ is not admissible. This is absurd, and the result follows.
\end{proof}

\section{Uniqueness}\label{sec:uniqueness}

We now prove that, given an admissible sphere $\Lambda$ in $\bH$ (or $\bH_+$), there is at most one complete maximal $p$-submanifold $M$ whose asymptotic boundary is $\Lambda$.

\begin{theorem}\label{thm:injective}
The boundary maps $\partial_\infty:\M\rightarrow\Bd$ and $\partial_\infty^+:\M_+\rightarrow\Bd_+$ are injective. In particular, a non-negative $(p-1)$-sphere in $\bH$ or $\bH_+$ is the asymptotic boundary of at most one complete maximal $p$-submanifold.
\end{theorem}

\noindent In particular, this yields the following useful corollary.

\begin{corollary}
Let $\Gamma$ be a subgroup of $\mathsf{PO}(p,q+1)$ and denote by $\M^\Gamma$ and $\Bd^\Gamma$ the space of $\Gamma$-invariant elements in $\M$  and $\Bd$ respectively. If $M$ is an element of $\M$ such that $\partial_\infty M\in \Bd^\Gamma$, then $M\in \M^\Gamma$.
\end{corollary}

\noindent The proof of this result is similar in spirit to that given in \cite[Section 4]{LTW} for the case $\mathbf H^{2,n}$, although a litte more care is required in the present, higher-dimensional setting.

\subsection{The test function}

The result will follow upon applying the maximum principle to a certain well chosen function that we describe in this section. Recall that, given a point $x\in \H_+$,  $\hat x\in\Quad$ denotes its unique representative such that $\q(\hat x,\hat x)=-1$.

\begin{lemma}\label{lemma:Acausal}
Let $\Lambda$ be a non-negative sphere in $\bH_+$. If $x$ is an element in $\Conv(\Lambda)\cap\H_+$ and if $\{y_n\}_{n\in\mathbf N}$ is a sequence in $\H_+$ converging to a point $y$ in $\Lambda$, then $\q( \hat x,\hat y_n)$ tends to $-\infty$.
\end{lemma}
\begin{proof}
The arguments of \cite[Proposition 2.15]{LTW} show that, since $\Lambda$ is a non-negative sphere, $\q( \hat x,\hat y)<0$ for any representative $\hat y\in E$ of $y$. By hypothesis, there exists a sequence $\{\epsilon_n\}_{n\in\mathbf N}$ of positive numbers, converging to zero, such that $\epsilon_n\hat y_n\to\hat y$. Since $\q( \hat x,\epsilon_n\hat y_n)\to\q(\hat x,\hat y) <0$, it follows that $\q( \hat x,\hat y_n)\to-\infty$, as desired.
\end{proof}

We now describe the function of interest to us. Upon taking lifts, it will suffice to show that the boundary operator $\partial_\infty^+:\M^+\to\Bd^+$ is injective. Let $M_1, M_2 \in \M_+$ be such that $\partial^+_\infty M_1 =\partial^+_\infty M_2 =\Lambda \in \Bd_+$.
Using the identification of $\H_+$ with the quadric $\Quad=\{\hat x\in E~,~\q(\hat x,\hat x)=-1\}$ we define
\begin{equation}\label{eq:defi beta}
\begin{array}{llll}
\beta : & M_1\times M_2 & \longrightarrow & \mathbf R \\
& (x,y) & \longmapsto & \q( \hat x,\hat y) ~.\end{array}
\end{equation}
Note that, by Lemma \ref{lemma:Acausal0}, if $M_1=M_2$ then the supremum of $\beta$ is $-1$, and is achieved at all pairs of the form $(x,x)$. We now show that this condition in fact characterises equality of $M_1$ and $M_2$.

\begin{lemma}\label{lem:supbeta}
If $M_1$ is different from $M_2$ then $\sup \beta\in(-1,0]$.
\end{lemma}

\begin{proof}
By Lemma \ref{lemma:ContainedInConvexHull} and Lemma \ref{lemma:Acausal}, for all $x_1\in M_1$, $\Lambda$ is disjoint from $x_1^\bot$. However, by Lemma \ref{lemma:ContainedInConvexHull}, $M_2$ is contained in $\Conv(\Lambda)$, and is therefore also disjoint from $x_1^\bot$, so that $\beta(x_1,\cdot)$ does not vanish on $M_2$. Since, by Lemma \ref{lemma:Acausal}, $\beta(x_1,y)$ becomes negative as $y$ approaches $\Lambda$, it follows that $\beta(x_1,x_2)<0$ for any $(x_1,x_2)\in M_1\times M_2$.

We now work in a Fermi chart with Fermi projection $\pi$. Since $M_1$ and $M_2$ are complete, they are entire graphs. In particular, if $M_1\neq M_2$, then there exists $(x_1,x_2)\in M_1\times M_2$ such that $x_1\neq x_2$ and $\pi(x_1)=\pi(x_2)$. It now follows by Lemma \ref{lemma:WarpedProjectionIncreaseLength} $(1)$ that $\beta(x_1,x_2)>-1$, and the result follows.
\end{proof}

We now compute the Hessian of $\beta$, using the quadric model $\Quad$.

\begin{lemma}\label{lem:HessianB}
The Hessian of $\beta$ at a point $p:=(x_1,x_2)$ in $M_1\times M_2$ in the direction $v:=(u_1,u_2)$ in $\T_{x_1} M_1 \times \T_{x_2} M_2$ is given by
\begin{equation*}
\opHess_p\beta(v,v)= (\|u_1\|^2+\|u_2\|^2))\beta(p) + 2 \q (u_1,u_2) + \q (\II_1(u_1,u_1),x_2) + \q (x_1,\II_2(u_2,u_2))\ ,
\end{equation*}
where, for each $i$, $\II_i$ denotes the second fundamental form of $M_i$ and $\|u_i\|=\sqrt{\q(u_i,u_i)}$.
\end{lemma}

\begin{proof}
Let $\gamma_i$ be a local geodesic in $M_i$ with $\gamma_i(0)=x_i$ and $\overset{\bullet}{\gamma_i}(0)=u_i$. Using the decomposition \eqref{eq:decomposition flat connection} of the trivial connection $D$ on $E$, we obtain
\begin{eqnarray*}
\opHess_p\beta (v,v) & = & \left.\frac{d^2}{dt^2}\right\vert_{t=0} \q (\gamma_1(t),\gamma_2(t))  \\
& = & \q (\|u_1\|^2x_1 + \II_1(u_1,u_1) , x_2) + 2\q (u_1,u_2) + \q (x_1, \|u_2\|^2x_2 + \II_2(u_2,u_2))\ ,
\end{eqnarray*}
as desired.
\end{proof}

\subsection{A linear-algebraic interlude}

Theorem \ref{thm:injective} will now be a consequence of the following technical result of linear algebra.

\begin{lemma}\label{lemma:LinAlgLemma}
Let $U$ and $V$ be finite dimensional Euclidean vector spaces of equal dimension and let $\varphi$ be an endomorphism of $U\oplus V$ of the form
\begin{equation}\label{eq:lemma linear algebra varphi}
\varphi = \begin{pmatrix} \mu \Id + A & M^* \\ M & \mu \Id + B\end{pmatrix}~,
\end{equation}
where
\begin{enumerate}
	\item $\mu$ is a real number in $(-1,+\infty)$,
	\item $A$ and $B$ are traceless and symmetric,
	\item $M^*$ is the adjoint of $M$ with respect to the underlying scalar products of $U$ and $V$, and
	\item for every $u$ in $U$, $\Vert M (u) \Vert \geq \Vert u \Vert$.
\end{enumerate}
Then $\varphi$ has at least one positive eigenvalue.
\end{lemma}

\begin{proof}
Since the endomorphism $M^*M$ of $U$ is symmetric and non-negative semi-definite, there is an orthonormal basis $(u_1,\cdots,u_n)$ of $U$ such that, for each $i$, $M^*M(u_i)=\lambda^2_i u_i$ for some $\lambda_i\in\mathbf{R}$. Observe now that for every $i,j$,
\begin{equation}\label{eq:productsM}
\langle M(u_i),M(u_j)\rangle = \langle M^*M (u_i),u_j\rangle = \lambda^2_i \langle u_i,u_j\rangle=\lambda^2_i\delta_{ij}~,
\end{equation}
so that, by $(4)$, we may take $\lambda_i\geqslant 1$ for all $i$. Furthermore, setting $v_i := \frac{1}{\lambda_i} M(u_i)$ yields an orthonormal basis $\mathcal B = (u_1,\cdots,u_n,v_1,\cdots,v_n)$ of $U\oplus V$ with respect to which both blocks $M$ and $M^*$ in \eqref{eq:lemma linear algebra varphi} are represented by the diagonal matrix $\Lambda=\mathrm{diag}(\lambda_1,\cdots,\lambda_n)$.

Consider now the orthogonal matrix
\begin{equation*}
P:= \frac{1}{\sqrt 2}\begin{pmatrix} \Id & \Id \\ -\Id & \Id \end{pmatrix}~.
\end{equation*}
A direct computation yields
\begin{equation*}
\varphi_{P^{-1}\mathcal B}= P^{-1} \varphi_{\mathcal B} P = 2 \begin{pmatrix} \mu \Id - \Lambda +\frac{1}{2}(A+B) & \frac{1}{2}(A-B) \\ \frac{1}{2}(A-B) & \mu\Id +\Lambda +\frac{1}{2}(A+B) \end{pmatrix}~.
\end{equation*}
Since $\mu > -1$ and $\lambda_i\geq 1$, and since $A$ and $B$ are traceless, the trace of $\mu\Id +\Lambda +\frac{1}{2}(A+B)$ is positive, so that the diagonal of $\varphi_{P^{-1}\mathcal B}$ has at least one positive element. However, by the Schur-Horn theorem \cite{schur} (see also \cite[Example 5.50]{mcduff} for a proof using moment maps), the diagonal vector $(d_1,\cdots,d_{2n})$ of $\varphi_{P^{-1}\mathcal B}$ is contained in the permutation polytope generated by the eigenvalues of $\varphi_{P^{-1}\mathcal B}$, and it follows that $\varphi_{P^{-1}\mathcal B}$ has at least one positive eigenvalue, as desired.
\end{proof}

\begin{corollary}\label{cor:Hessian}
For any point $p$ in $M_1\times M_2$ with $\beta(p)>-1$, there exists a vector $v$ in $\T_p(M_1\times M_2)$ such that $\opHess_p\beta(v,v) >0$.
\end{corollary}

\begin{proof}
By Lemma \ref{lem:HessianB}, with respect to the splitting $\T_{x_1}M_1\oplus \T_{x_2}M_2$, the matrix $\text{Hess}_p \beta(v,v)$ can be written in the form $\langle \varphi(v),v\rangle$, where $\langle\cdot,\cdot\rangle$ is the scalar product given by the sum of the restrictions of $\q$ to $\T_{x_1}M_1$ and $\T_{x_2}M_2$,
\begin{equation*}
\varphi= \begin{pmatrix} \beta(p)\Id + A_1(\cdot,x_2) & \pi_1 \\ \pi_2 & \beta(p)\Id + A_2(\cdot,x_1) \end{pmatrix} ~,
\end{equation*}
$p=(x_1,x_2),~\pi_i$ is the orthogonal projection from $E$ to $\T_{x_i}M_i$, and $A_i\in\mathrm{End}(\T_{x_i}M_i)$ satisfies the identity $\q (A_i(u_i,x),\cdot)=\q (\II(u_i,\cdot),x)$ for $u_i\in T_{x_i}M_i$ and $x\in E$. Since the orthogonal complement of $\T_{x_i}M_i$ in $E$ is negative-definite, for all $x\in E$, $\q(\pi_i(x),\pi_i(x))\geq \q(x,x)$, and the result now follows by Lemma \ref{lemma:LinAlgLemma}.
\end{proof}

\subsection{Uniqueness}

\begin{proof}[Proof of Theorem \ref{thm:injective}]
Upon taking lifts, it suffices to prove the statement for $\H_+$. Suppose the contrary, so that $M_1\neq M_2$. By Lemma \ref{lem:supbeta}, $\sup\beta\in (-1,0]$. Consider a maximizing sequence $\{(x_n,y_n)\}_{n\in\mathbf N}$ for $\beta$. For each $n$, let $g_n$ be an isometry such that $g_n(x_n)=x$ and $g_n(\T_{x_n}M_1)$ coincides with some fixed $p$-dimensional real subspace of $\T_x\H$. Applying Theorem \ref{theorem:PropernessOfTau}, the sequence $\{g_n(M_1)\}_{n\in\mathbf N}$ subconverges to a complete maximal submanifold $M'_1$. By continuity of the boundary map, the sequence $\{g_n(\Lambda)\}_{n\in \mathbf N}$ likewise subconverges to an admissible sphere $\Lambda'$. By Theorem \ref{theorem:Properness}, upon extracting another subsequence, we may suppose that $\{g_n(M_2)\}_{n\in \mathbf N}$ also converges to $M'_2$ with $\partial_\infty M'_2=\Lambda'$.

We now claim that the sequence $\{g_n(y_n)\}_{n\in \mathbf N}$ is bounded. Indeed, suppose the contrary. Upon extracting a subsequence, we may suppose that this sequences converges to some point of $\Lambda'$ so that, by Lemma \ref{lemma:Acausal}, $\q (x,g_n(y_n))\to -\infty$. However by Lemma \ref{lem:supbeta}, for all large $n$,
\begin{equation*}
\q (x,g_n(y_n))=\q (g_n(x_n),g_n(y_n))=\q (x_n,y_n)>-1\ ,
\end{equation*}
which is absurd, so that this sequence is indeed bounded, as asserted. Upon extracting a further subsequence, we may suppose that it converges to some point $y$, say, of $M'_2$ such that $\q (x,y)>-1$. In particular, by Lemma \ref{lemma:Acausal0}, $y\notin M'_1$, so that $M_1'\neq M_2'$. Defining $\beta':M_1'\times M_2'\to\R$ as before, we see that this function attains its maximum at the point $(x,y)$. Furthermore, by Lemma \ref{lem:supbeta} again, $\beta'(x,y)\in(-1,0]$. However, by Corollary \ref{cor:Hessian}, the Hessian of $\beta'$ at $(x,y)$ has a positive eigenvalue, which contradicts the fact that $(x,y)$ is a maximum, and this completes the proof.
\end{proof}

\section{Stability}\label{sec:stability general}

We now study the problem of perturbing maximal graphs. It will suffice to work in the double cover $\H_+$. We prove the following result.

\begin{theorem}\label{thm:perturb H+}
Let $(\Lambda_t)_{t\in(-\epsilon,\epsilon)}$ be a family of smooth, spacelike positive $(p-1)$-spheres in $\partial_\infty\H_+$ varying continuously in the $C^\infty$ topology. If there exists a complete maximal $p$-submanifold $M_0$ of $\H_+$ such that $\partial_\infty M_0=\Lambda_0$ then, for all sufficiently small $t$, there exists a complete maximal $p$-submanifold $M_t$ of $\H_+$ such that $\partial_\infty M_t=\Lambda_t$.
\end{theorem}

The overall strategy of our proof of Theorem \ref{thm:perturb H+} follows the now standard approach towards perturbing minimal surfaces: smooth maximal $p$-submanifolds close to $M_0$ are first represented as zeroes of some smooth functional over some Banach space, and the desired perturbations are then obtained upon applying the implicit function theorem.

This process is usually relatively straightforward when the submanifolds of interest are compact, but becomes harder in the non-compact case of interest to us here. The main challenge lies in proving invertibility of the Jacobi operator of $M_0$. Although systematic treatments of this kind of problem exist (see, for example, \cite{Mazzeo,AlexMazz,Fine}), we find it preferrable, and more informative, to adopt an alternative, more ad-hoc approach, guided by two general principles.

The first general principle is to choose an appropriate asymptotic model, that is, a class of submanifolds with readily determined properties which are approached asymptotically by the maximal graphs of interest to us. The asymptotic model used here will be given by the class of radially parametrized cones, and we will study their geometric and analytic properties in Sections \ref{sec:cones} and \ref{section:RadiallyParametrizedCones}.

The second general principle concerns the invertibility of elliptic operators over non-compact manifolds. In order to correctly explain it, we recall the concept of elliptic estimates and introduce the concept of pre-elliptic estimates. Let $L:\mathcal E\rightarrow \mathcal F$ be a bounded linear map between Banach spaces. An \emph{elliptic estimate} for $L$ is an estimate of the form
\begin{equation}\label{eqn:EllipticEstimate}
\|x\|_{\mathcal E} \leq C\big(\|Rx\|_{\mathcal G}+\|Lx\|_{\mathcal F}\big)\ ,
\end{equation}
valid for all $x\in \mathcal E$, where $R: \mathcal E\rightarrow \mathcal G$ is a \emph{compact} linear map into a normed vector space. Elliptic estimates measure the defect of invertibility in terms of compact operators. When such an estimate holds for $L$, it is known that $L$ has finite-dimensional kernel and closed image (see, for example, \cite[Chapter 21, Theorem 4]{Lax}), and it is then often a formal matter to derive first the Fredholm property and then invertibility from the specific geometry of the problem at hand.

We define a \emph{pre-elliptic estimate} for $L$ to be an estimate of the form
\begin{equation}\label{eqn:PreEllipticEstimate}
\|x\|_{\mathcal E} \leq C\big(\|Rx\|_{\mathcal G}+\|Lx\|_{\mathcal F}\big)\ ,
\end{equation}
valid for all $x\in \mathcal E$, where now $R:\mathcal E\rightarrow \mathcal G$ is an arbitrary bounded linear map from $\mathcal E$ into some normed vector space. Although pre-elliptic estimates are in themselves of little interest, being readily obtained for any operator upon setting $R:=\opId$, they provide a useful means of organising our ideas, as we will now see.

Suppose that $L$ is an elliptic partial differential operator defined over some manifold $X$ of bounded geometry. A typical estimate provided by the general theory of elliptic partial differential operators is
\begin{equation}\label{eqn:GeneralEllipticEstimate}
\|f\|_{H^2(X)} \leq C\big(\|f\|_{L^2(X)} + \|Lf\|_{L^2(X)}\big)\ ,
\end{equation}
where $L^2(X)$ here denotes the Lebesgue space of square integrable functions over $X$, and $H^2(X)$ denotes the Sobolev space of $L^2$ functions with $L^2$ weak derivatives up to order $2$. When $X$ is compact, the Rellich-Kondrachov Theorem (see, for example, \cite[Section 5.7]{Evans}) implies that the canonical embedding $H^2(X)\hookrightarrow L^2(X)$ is compact, so that \eqref{eqn:GeneralEllipticEstimate} is elliptic, and all is well. However, when $X$ is non-compact, the Rellich-Kondrachov Theorem does not apply, and \eqref{eqn:GeneralEllipticEstimate} is unfortunately only pre-elliptic.

Our second general principle is thus to transform the pre-elliptic estimates provided by the general theory of elliptic partial differential operators into elliptic estimates. The main tool required for this process is a standard estimate for solutions of a certain class of ordinary differential equations which we will describe in Section \ref{subhead:DifferentialOperatorsOverTheLine}. This will be used to obtain elliptic estimates for Sobolev spaces over radially parametrized cones in Section \ref{subhead:WeightedSobolevSpaces}, and for H\"older spaces over such cones in Section \ref{subhead:WeightedHoelderSpaces}. These results will be adapted in Section \ref{subhead:MaximalGraphs} to address the case of Sobolev and H\"older spaces over maximal graphs. This will yield the desired invertibility of the Jacobi operator, allowing us to prove Theorem \ref{thm:perturb H+} in Section \ref{PerturbationsOfMinimalEnds}.

Finally, the reader may wonder why we chose to work over both Sobolev and H\"older spaces. We do so because, although H\"older spaces provide the natural framework for studying the perturbation theory of maximal surfaces, invertibility of operators is more readily proven over Sobolev spaces, on account of the natural dualities they present. It is thus standard practice to obtain invertibility results first over Sobolev spaces, which are then used as a basis for proving the corresponding results over H\"older spaces.

\newsubhead{The Jacobi operator}[TheJacobiOperator]

We first introduce the protagonist of this section, namely the Jacobi operator, which measures variations of the mean curvature vector arising from infinitesimal perturbations of spacelike graphs. Most of the rest of this section will be devoted to studying its invertibility properties over suitable Banach spaces.

The Jacobi operator is defined formally as follows. Let $M$ be a smooth $p$-dimensional spacelike submanifold of $\H$, and let $\N M$ denote its normal bundle. Given a section $\sigma$ of $\N M$, we define $\iota(\sigma):M\rightarrow\H$ by
\begin{equation*}
\iota(\sigma)(y) := \Exp(\sigma(y))\ ,\myeqnum{\nexteqnno[ImageOfNormalGraph]}
\end{equation*}
where $\Exp$ denotes the exponential map of $\H$. We define $\mathcal{H}(\sigma)\in\Gamma(\N M)$ such that, for all $y\in M$, the vector $\mathcal{H}(\sigma)(y)$ is the parallel transport along the unique geodesic from $\iota(\sigma)(y)$ to $y$ of the mean curvature vector of $\iota(\sigma)$ at this point. The \emph{Jacobi operator} of $Y$ is then defined by
\begin{equation*}
(J\sigma)(y) := \frac{\partial}{\partial t}\mathcal{H}(t\sigma)(y)\bigg|_{t=0}\ .\myeqnum{\nexteqnno[DefinitionOfJacobiOperator]}
\end{equation*}

\begin{remark}
Note that when $M$ is $C^3$, its normal bundle is $C^2$, and the preceding construction still makes sense. This will become relevant in Section \ref{sec:RenormalizedArea}, where we will be concerned with problems of relatively low regularity.
\end{remark}

\begin{lemma}
Let $M$ be a $p$-dimensional spacelike submanifold of $\H$. The Jacobi operator of $M$ is given by
\begin{equation}
J\sigma = \Delta^N\sigma - p\sigma + \sum_{m,n}\g\big(\opII(e_m,e_n),\sigma\big)\opII(e_m,e_n)~,\myeqnum{\nexteqnno[JacobiOperator]}
\end{equation}
where $\opII$ here denotes the second fundamental form of $M$, and $e_1,\cdots,e_p$ is any orthonormal frame of the tangent bundle of $M$.
\proclabel{JacobiOperator}
\end{lemma}

\begin{proof}
Since the result is local, and since we may always add to $\sigma$ the normal component of some timelike Killing vector field of $\H$, we may suppose that $\sigma$ does not vanish. We extend $\iota(\sigma)$ to a smooth immersion of $M\times(-\epsilon,\epsilon)$ into $\H$, defined by $\iota(\sigma)(y,t)=\Exp(\sigma(ty))$. The section $\sigma$ is thus identified with the tangent vector field of the second component. For all $t$, we denote $M_t:=M\times\{t\}$, and we denote respectively by $\T M_t$ and $\N M_t$ its tangent and normal bundles in $\H$. Note, in particular, that $\sigma$ is timelike and normal to $M_0$. Choose $y_0\in M=M_0$, let $e_1,\cdots,e_p$ be an orthonormal frame of $(\T M_t)_{t\in(-\epsilon,\epsilon)}$, let $f_1,\cdots,f_q$ be an orthonormal frame of $(\N M_t)_{t\in(-\epsilon,\epsilon)}$, and suppose that, for all $i$ and for all $j$,
\begin{equation}
(\nabla^\T e_i)(y_0,0)=0~\text{ and }\ (\nabla^\N f_j)(y_0,0) = 0\ .\label{ConditionsOnFrames}
\end{equation}
Now let $\xi$ and $\nu$ be vector fields tangent to $(M_t)_{t\in(-\epsilon,\epsilon)}$ such that
\begin{equation}\label{eqn:JOFirstCondOnXiNu}
[\xi,\sigma]=[\nu,\sigma]=0~,
\end{equation}
and such that, for every other tangent vector field $\mu$,
\begin{equation}\label{eqn:JOSecondCondOnXiNu}
(\nabla^\T_\mu\xi)(y_0,0) = (\nabla^\T_\mu\nu)(y_0,0) = 0~.
\end{equation}
For all $i$,
\begin{equation*}\eqalign{
D_\sigma(\g(\opII(\xi,\nu),f_j))
&=D_\sigma(\g(\pi^N(\nabla_\xi\nu),f_j))\cr
&=D_\sigma(\g(\nabla_\xi\nu,f_j))\cr
&=\g(\nabla_\sigma\nabla_\xi\nu,f_j) + \g(\nabla_\xi\nu,\nabla_\sigma f_j)\cr
&=\g( R_{\sigma\xi}\nu,f_j) + \g(\nabla_\xi\nabla_\sigma\nu,f_j) + \g(\nabla_\xi\nu,\nabla_\sigma f_j)\ ,\cr}
\end{equation*}
where, we recall, $\g$ denotes the pseudo-riemannian metric of $\H$, and $R$ denotes its Riemann curvature tensor. By \eqref{eqn:JOSecondCondOnXiNu} and  \eqref{ConditionsOnFrames}, $(\nabla_\xi\nu)(y_0,0)$ is normal to $M_0$ and $(\nabla_\sigma f_i)(y_0,0)$ is tangential, so that the final term vanishes. Next, since $\nabla$ is torsion free and $[\sigma,\nu]$ vanishes,
\begin{equation*}
\g(\nabla_\xi\nabla_\sigma\nu,f_j) = \g(\nabla_\xi\nabla_\nu\sigma,f_j)~.
\end{equation*}
Recalling the decomposition of \eqref{DecompositionOfCovDer}, since $\H$ has constant sectional curvature equal to $-1$, and since $\sigma$ is normal to to $M$,
\begin{equation*}\eqalign{
D_\sigma(\g(\opII(\xi,\nu),f_j))
&=-\g(\sigma,f_j)\g(\xi,\nu) + \g(\nabla_\xi\nabla_\nu\sigma,f_j)\cr
&=-\g(\sigma,f_j)(\xi,\nu) + \g(\nabla_\xi\nabla^N_\nu\sigma,f_j) - \g(\nabla_\xi B(\nu)(\sigma),f_j)\cr
&=-\g(\sigma,f_j)(\xi,\nu) + \g(\opHess^N(\sigma)(\xi,\nu),f_j) + \g( B(\nu)(\sigma),\nabla_\xi f_j)\ .\cr}
\end{equation*}
However, for any normal vector field $\tau$,
\begin{equation*}
B(\xi)(\tau)
=-\sum_{i=1}^p\g(\nabla_\xi\tau,e_i) e_i
=\sum_{i=1}^p\g(\tau,\nabla_\xi e_i) e_i
=\sum_{i=1}^p\g(\tau,\opII(\xi,e_i)) e_i\ .
\end{equation*}
Applying to this to $\nabla_\nu\sigma$ and $\nabla_\xi f_j$ yields
\begin{equation}\label{eqn:JOCalcI}
D_\sigma(\g(\opII(\xi,\nu),f_j))
=-\g(\sigma,f_i)\g(\xi,\nu) + \g(\opHess^N(\sigma)(\xi,\nu),f_j) - \sum_{i=1}^p\g(\opII(\nu,e_i),\sigma)\g(\opII(\xi,e_i),f_j)\ .
\end{equation}
On the other hand, recalling that $[\sigma,\xi]=0$, we have $\nabla_\sigma\xi = \nabla_\xi\sigma$, so that
\begin{equation*}
(\nabla_\sigma\xi)^T = (\nabla_\xi\sigma)^T = -B(\xi)(\sigma)= -\sum_{i=1}^p\g(\opII(\xi,e_i),\sigma) e_i\ ,
\end{equation*}
with a similar formula holding for $(\nabla_\sigma\nu)^T$. It follows that
\begin{equation}\label{eqn:JOCalcII}
\eqalign{
D_\sigma(\g(\opII(\xi,\nu),f_j))
&=\g( \nabla_\sigma (\opII (\xi,\nu)),f_j) + \g(\opII (\xi,\nu),\nabla_\sigma f_j )\vphantom{\frac{1}{2}}\cr
&=\g( \nabla_\sigma \opII (\xi,\nu),f_j) + \g( \opII (\nabla^T_\sigma\xi,\nu),f_j) + \g( \opII (\xi,\nabla^T_\sigma\nu),f_j)\vphantom{\frac{1}{2}}\cr
&=\g( \nabla^N_\sigma \opII (\xi,\nu),f_j) - \sum_{i=1}^p\g(\opII(\xi,e_i),\sigma )\g( \opII (\nu,e_i),f_j)\cr
&\qquad\qquad - \sum_{i=1}^p\g( \opII (\nu,e_i),\sigma)\g(\opII(\xi,e_i),f_j )\ ,\cr}
\end{equation}
where the last term of the first line vanishes since, as before, $(\nabla_\sigma f_i)(y_0,0)$ is tangential.

Combining \eqref{eqn:JOCalcI} and \eqref{eqn:JOCalcII} now yields
\begin{equation*}
\g(\nabla^N_\sigma\opII(\xi,\nu),f_j) = -\g(\sigma,f_j)\g(\xi,\nu) + \g(\opHess^N(\sigma)(\xi,\nu),f_j) + \sum_{i=1}^p\g(\opII(\xi,e_i),\sigma)\g(\opII(\nu,e_i),f_j)\ ,
\end{equation*}
so that
\begin{equation*}
\nabla_\sigma^N\opII(\xi,\nu) = -\g(\xi,\nu)\sigma + \opHess^N(\sigma)(\xi,\nu) + \sum_{i=1}^p\g(\opII(\xi,e_i),\sigma)\opII(\nu,e_i)\ ,
\end{equation*}
and the result follows upon taking the trace.\end{proof}

\subsection{The asymptotic model I - spacelike polar coordinates}\label{sec:cones}

We now study the asymptotic model that will be used in the sequel. We first describe the coordinate systems of $\H_+$ and $\partial_\infty\H_+$ in which our constructions will be carried out. Let $x_0$ be a point of $\H_+$. We denote the (positive) unit tangent bundle of $\H_+$ at $x_0$ by
\begin{equation*}
\T_{x_0}^1\H_+ := \{ v\in \T_{x_0}\H_+\ |\ \g (v,v)=1 \}\ .\label{PseudoSphericalSpacetime}
\end{equation*}
Define $\Phi:\T_{x_0}^1\H_+\times(0,\infty)\rightarrow\H_+$ and $\Phi_\infty:\T_{x_0}^1\H_+\rightarrow\partial_\infty\H_+$ respectively by
\begin{equation*}\eqalign{
\Phi(v,r)&:=\Exp_{x_0}(rv)\ ,\ \text{and}\cr
\Phi_\infty(v)&:=\mlim_{r\rightarrow+\infty}\Exp_{x_0}(rv)\ ,\cr}\label{PolarCoordinates}
\end{equation*}
where $\opExp_{x_0}$ here denotes the exponential map of $\H_+$. It follows that $\Phi$ parametrizes the set of all points in $\H_+$ separated from $x_0$ by some non-trivial spacelike geodesic, whilst $\Phi_\infty$ parametrizes the set of end-points of all such geodesics in $\H_+$. We call $\Phi$ (resp. $\Phi_\infty$) \emph{spacelike polar coordinates} of $\H_+$ (resp. $\partial_\infty\H_+$) about $x_0$.

We now provide algebraic descriptions of the images of these parametrizations, as well as their induced metrics. Recall first that we denote the unique representative of any point $x\in\H_+$ in $\Quad$ by $\hat{x}$, and we denote {\sl any} representative of any point $x\in\partial_\infty\H_+$ by $\hat{x}$. Let $h$ denote the pseudo-riemannian metric that $\T^1_{x_0}\H_+$ inherits as a submanifold of $\T_{x_0}\H_+$.

\begin{lemma}\label{eqn:GeometriesOfPhiAndPhiInfinity}
For any  $x_0\in\H_+$,
\begin{enumerate}
\item $\Phi$ maps $\T_{x_0}^1\H_+\times(0,\infty)$ diffeomorphically onto the open subset $\Omega_{x_0}\subseteq\H_+$ given by
\begin{equation}\label{eqn:ImageOfPhi}
\Omega_{x_0}:=\{x\in\H_+\ |\ \q(\hat{x},\hat{x}_0)<-1\}\ ,
\end{equation}
and the pull-back of the pseudo-hyperbolic metric is
\begin{equation}\label{eqn:MetricInPolarCoordinates}
\hat{h} := \opSinh^2(r)h\oplus dr^2\ .
\end{equation}
\item $\Phi_\infty$ maps $\T^1_{x_0}\H_+$ conformally diffeomorphically onto the open subset $\Omega_{x_0,\infty}\subseteq\partial_\infty\H_+$ given by
\begin{equation}\label{eqn:ImageOfPhiInfinity}
\Omega_{x_0,\infty}:=\{x\in\partial_\infty\H_+\ |\ \q(\hat{x},\hat{x}_0)<0\}\ .
\end{equation}
\end{enumerate}
\end{lemma}

\begin{proof} For any spacelike vector $u\in\opT_{\hat x_0}\Quad$, the geodesic leaving $\hat x_0$ in the direction of $u$ is
\begin{equation*}
\gamma_u(t) := \opCosh(t\|u\|)\hat x_0 + \frac{1}{\|u\|}\opSinh(t\|u\|)u\ .
\end{equation*}
Indeed, this is the unique constant-speed parametrised curve in the intersection of $\Quad$ with the plane spanned by $\hat x_0$ and $u$ having velocity $u$ at $\hat x_0$. It follows that, for all $(u,r)\in \T^1_{x_0}\Quad\times(0,\infty)$,
\begin{equation}\label{eqn:ExplicitFormulaForPhi}
\hat{\Phi}(u,r) = \opCosh(r)\hat x_0 + \opSinh(r)u\ .
\end{equation}
We now show that $\opIm(\Phi)=\Omega_{x_0}$. Indeed, for all $(u,r)$, since $\q (u,\hat x_0)=0$,
\begin{equation*}
\q(\hat{\Phi}(u,r),\hat{x}_0)=-\opCosh(r)<-1\ ,
\end{equation*}
so that $\Phi(u,r)\in\Omega_{x_0}$. Conversely, we verify that every $y\in\Omega_{x_0}$ satisfies $\hat y=\hat\Phi(u,r)$, where
\begin{equation*}
\eqalign{
r &:=\opArccosh(-\q(\hat x_0,\hat y))\ ,\ \text{and}\vphantom{\frac{1}{2}}\cr
u &:=\frac{1}{\opSinh(r)}\hat y - \opCoth(r)\hat x_0\ ,\cr}
\end{equation*}
so that $\opIm(\Phi)=\Omega_{x_0}$, as desired.

Upon letting $r$ tend to $+\infty$ in \eqref{eqn:ExplicitFormulaForPhi}, we see that
\begin{equation*}
\hat{\Phi}_\infty(u) = \hat x_0 + u\ .
\end{equation*}
We now show that $\opIm(\Phi_\infty)=\Omega_{x_0,\infty}$. Indeed, for all $u$, since $\q (u,\hat x_0)=0$,
\begin{equation*}
\q(\hat{\Phi}_\infty(u),\hat{x}_0) = -1\ ,
\end{equation*}
so that $\hat{\Phi}_\infty(u)\in\Omega_{x_0,\infty}$. Conversely, given $y\in\Omega_{x_0,\infty}$, we may suppose that $\q(\hat{y},\hat{x}_0)=-1$, and we verify that $u:=\hat{y}-\hat{x}_0$ satisfies $\q (u,\hat x_0)=0$ and $\q(u,u)=1$. Since
\begin{equation*}
\hat{y}=\hat{\Phi}_\infty(u)\ ,
\end{equation*}
it follows that $\hat{y}\in\opIm(\hat{\Phi}_\infty)$, so that $\opIm(\hat{\Phi}_\infty)=\Omega_{x_0,\infty}$, as desired.

It remains only to study the induced metrics. Let $v$ be a tangent vector to $\opT^1_{x_0}\H_+$ at $u$. Differentiating \eqref{eqn:ExplicitFormulaForPhi} yields
\begin{equation*}
\eqalign{
D\hat{\Phi}(u,r)\cdot(v,0) &= \opSinh(r)v\ ,\ \text{and}\cr
D\hat{\Phi}(u,r)\cdot(0,1) &= \opSinh(r)\hat x_0 + \opCosh(r)u\ ,\cr}
\end{equation*}
so that
\begin{equation*}
\q\left(D\hat{\Phi}(u,r)\cdot(v,t),D\hat{\Phi}(u,r)\cdot(v,t)\right) := \opSinh(r)\q(v,v) + t^2\ ,
\end{equation*}
as desired. Finally, since
\begin{equation*}
D\hat{\Phi}_\infty(u)\cdot v = v\ ,
\end{equation*}
$\hat{\Phi}_\infty$ is conformal, and this completes the proof.
\end{proof}

\begin{remark}\label{rmk image spacelike chart}
In view of the proof of Theorem \ref{thm:perturb H+} that will be given below, it will be useful to observe that if $N_p$ and $N_q$ denote the respective north poles of $\S^p$ and $\S^q$, then in any Fermi parametrization $\Psi:\S^p_+\times\S^q\rightarrow\H_+$ such that $\Psi(N_p,N_q)=x_0$, $\Omega_{x_0,\infty}=\Psi(\S^{p-1}\times\S^q_+)$ where $\S^q_+$  the open hemisphere about $N_q$ in $\S^q$. Indeed, using the notation of Section \ref{subsection:FermiCoordinates}, it follows from \eqref{eqn:FermiParametrisationOfQuadricI} that, if $\hat{x}=:\hat\Psi(u,w)$, then
\begin{equation*}
\q(\hat{x},\hat{x}_0)=-\langle N_q,w\rangle_W=-\cos(d_{\S}(N_q,w))~,
\end{equation*}
which is negative if and only if $d_\S(N_q,w)<\pi/2$.
\end{remark}

We henceforth identify $\T_{x_0}\H_+$ with $\Rpq$ and $\T^1_{x_0}\H_+$ with the pseudo-sphere $\S^{p-1,q}$ of vectors $v\in\Rpq$ on which the quadratic form of $\Rpq$ takes the value $+1$. The tangent bundle of $\Spmq\times(0,\infty)$ trivially identifies with $\pi_1^*\T\Spmq\oplus\R$, where $\pi_1$ here denotes projection onto the first factor. We say that a vector field is \emph{horizontal} whenever it takes values in $\pi_1^*\T\Spmq$. Let $\partial_r$ denote the unit normal vector field in the $r$-direction. By \eqref{eqn:MetricInPolarCoordinates} every horizontal vector field is orthogonal to $\partial_r$.

\begin{lemma}\label{lemma:CovariantDerivativeInPolarCoordinates} Let $\nabla$ and $\hat{\nabla}$ denote respectively the Levi-Civita covariant derivatives of $h$ and $\hat{h}$. For all horizontal vector fields $\xi$ and $\nu$,
\begin{equation}\nexteqnno[CovariantDerivativeInPolarCoordinates]
\eqalign{
\hat{\nabla}_\xi\nu &= (\pi_1^*\nabla)_\xi\nu - \opCosh(r)\opSinh(r)h(\xi,\nu)\partial_r\ ,\cr
\hat{\nabla}_\xi\partial_r &= \opCoth(r)\xi\ ,\cr
\hat{\nabla}_{\partial_r}\xi &= \opCoth(r)\xi + [\partial_r,\xi]\ ,\ \text{and}\cr
\hat{\nabla}_{\partial_r}\partial_r &= 0\ .\cr}
\end{equation}
\proclabel{CovariantDerivativeInPolarCoordinates}
\end{lemma}

\begin{remark}
Note that in the product space $\Spmq\times(0,\infty)$ we may view horizontal vector fields as $r$-dependent families of tangent vector fields over $\S^{p-1,q}$, every horizontal vector field then satisfies $[\partial_r,\xi]=\partial_r\xi$.
\end{remark}

\myproof The fourth identity follows from the fact that vertical lines in this parametrization are unit speed geodesics. For all $r$, let $A_r$ denote the shape operator of $\Spmq\times\{r\}$, that is
\begin{equation*}
\hat\nabla_\xi\partial_r = A_r\cdot\xi\ .
\end{equation*}
A standard exercise in the geometry of quadrics shows that, for all $r$,
\begin{equation*}
A_r = \opCoth(r)\opId\ .
\end{equation*}
This proves the second identity, and the third identity follows since $\hat{\nabla}$ is torsion free. For all $r$, the restriction of $\hat{h}$ to $\Spmq\times\{r\}$ is a scalar multiple of $h$. Since $\partial_r$ is spacelike, it follows that
\begin{equation*}\eqalign{
\hat{\nabla}_\xi\nu &= (\pi_1^*\nabla)_\xi\nu + \hat h(\hat{\nabla}_\xi\nu,\partial_r)\partial_r\cr
&=(\pi_1^*\nabla)_\xi\nu - \hat h(\nu,\nabla_\xi\partial_r)\partial_r\cr
&=(\pi_1^*\nabla)_\xi\nu - \hat h(\nu,A_r\cdot\xi)\partial_r\cr
&=(\pi_1^*\nabla)_\xi\nu - \opCoth(r)\hat h(\nu,\xi)\ .\cr}
\end{equation*}
The first identity follows, and this completes the proof.\myqed

\subsection{The asymptotic model II - radially parametrized cones}\label{section:RadiallyParametrizedCones}

We continue to use the notation of the preceding section. Given a spacelike submanifold $X\subset\T^1_{x_0} \H_+$ we define
\begin{equation*}\eqalign{
\hat{X} &:= \Phi(X\times(0,\infty))\ ,\ \text{and}\cr
X_\infty &:= \Phi_\infty(X)\ ,\cr}\myeqnum{\nexteqnno[DefinitionOfConeAndIdealProjection]}
\end{equation*}
and we call these submanifolds respectively the \emph{cone} of $X$ and its \emph{ideal projection}. Trivially $X_\infty = \partial_\infty\hat{X}$, and we will always assume that the latter is a spacelike $(p-1)$-sphere.

We first study the infinitesimal geometry of cones.

\begin{lemma}\label{RadialCone}
Let $g$ and $\hat{g}$ denote respectively the metrics of $X$ and $\hat{X}$. Then
\begin{equation}
\Phi^*\hat{g} = \opSinh^2(r)g\oplus dr^2\ .\label{ConeMetric}
\end{equation}
\end{lemma}

\begin{proof} Indeed, in spacelike polar coordinates,
\begin{equation*}
\hat{X} = X\times(0,+\infty)\ ,
\end{equation*}
and the result follows by \eqref{eqn:MetricInPolarCoordinates}.
\end{proof}

\noindent Let $\N X$ and $\N \hat{X}$ denote respectively the normal bundles over $X$ and $\hat{X}$ in $\Spmq$ and $\H$. Since $\partial_r$ is tangent to $\hat{X}$, with the above identification,
\begin{equation}
\N \hat{X} = \pi_1^*\N X\ .\myeqnum{\nexteqnno[NormalBundleOverCone]}
\end{equation}

\begin{lemma}\label{SecondFFOfCone}
Let $\opII$ (respectively $\hat{\opII}$) denote the second fundamental forms of $X$ (respectively $\hat{X}$) with respect to $h$ (resp. $\hat{h}$), and let $H$ (resp. $\hat{H}$) denote its mean curvature vector. Then
\begin{equation}\eqalign{
\hat{\opII} &= \pi_1^*\opII\ ,\ \text{and}\cr
\hat{H} &= \frac{\pi_1^*H}{\opSinh^2(r)}\ .\cr}\label{SecondFFOfCone}
\end{equation}
\end{lemma}

\begin{proof} Indeed, by \eqnref{CovariantDerivativeInPolarCoordinates}, for all horizontal $\xi$ and $\nu$,
\begin{equation*}
\hat{\opII}(\xi,\nu) = \pi^N\big(\hat{\nabla}_\xi\nu\big) = \pi^N\big((\pi_1^*\nabla)_\xi\nu\big) = (\pi_1^*\opII)(\xi,\nu)\ .
\end{equation*}
Likewise, for all horizontal $\xi$,
\begin{equation*}
\hat{\opII}(\partial_r,\xi) = \hat{\opII}(\xi,\partial_r) = \pi^N\big(\hat{\nabla}_\xi\partial_r\big) = \pi^N(\opCoth(r)\xi) = 0\ .
\end{equation*}
Finally,
\begin{equation*}
\hat{\opII}(\partial_r,\partial_r) = \pi^N\big(\nabla_{\partial_r}\partial_r\big) = 0\ ,
\end{equation*}
and this proves the first identity. The second identity follows by \eqref{ConeMetric} upon taking the trace, and this completes the proof.
\end{proof}

\begin{lemma}
\noindent Let $\nabla^N$ and $\hat{\nabla}^N$ denote the respective covariant derivatives of $\N X$ and $\N \hat{X}$. For all $k$,
\begin{equation}
\|(\hat{\nabla}^N)^k(\opSinh(r)^{-1}\hat{\opII})\| = \opSinh^{-(k+2)}(r)\|\nabla^k\opII\|\circ\pi_1\ .\myeqnum{\nexteqnno[DecayOfSecondFF]}
\end{equation}
Likewise,
\begin{equation}
\Vert(\hat{\nabla}^N)^k(\opSinh(r)^{-1}\hat{H})\Vert = \opSinh^{-(k+2)}(r)\|\nabla^k H\|\circ\pi_1\ .\myeqnum{\nexteqnno[DecayOfMeanCurv]}
\end{equation}
\proclabel{DecayOfSecondFFAndMeanCurv}
\end{lemma}

\begin{proof}
From \eqnref{CovariantDerivativeInPolarCoordinates}, since $\partial_r$ is tangent to $\hat{X}$, for every section $\sigma$ of $\N \hat{X}$, and for every horizontal vector field $\xi$,
\begin{equation*}\eqalign{
\hat{\nabla}^N_\xi\sigma &= (\pi_1^*\nabla^N)_\xi\sigma\ ,\ \text{and}\cr
\hat{\nabla}^N_{\partial_r}\sigma &= \opCoth(r)\sigma + [\partial_r,\sigma]\ .\cr}
\end{equation*}
Consequently, for all such $\xi$,
\begin{equation*}
\hat{\nabla}^N_\xi(\opSinh(r)^{-1}\hat{\opII}) = (\pi_1^*\nabla^N)_\xi(\opSinh(r)^{-1}\hat{\opII}) = \opSinh(r)^{-1}(\pi_1^*(\nabla^N\opII))_\xi\ .
\end{equation*}
Likewise, since $\pi_1^*\opII$ is constant in the radial direction,
\begin{equation*}
\hat{\nabla}_{\partial_r}^N\left(\opSinh(r)^{-1}\hat{\opII}\right) = \hat{\nabla}_{\partial_r}^N\left(\opSinh(r)^{-1}\pi_1^*\opII\right) = 0 = \opSinh(r)^{-1}(\pi_1^*(\nabla^N\opII))_{\partial_r}\ .
\end{equation*}
Iterating this argument yields, for all $k$,
\begin{equation*}
(\hat{\nabla}^N)^k\left(\opSinh(r)^{-1}\hat{\opII}\right) = \opSinh(r)^{-1}\pi_1^*((\nabla^N)^k\opII)\ .
\end{equation*}
The first identity follows by \eqref{eqn:MetricInPolarCoordinates}, and the second follows upon taking the trace.\end{proof}

We now study the asymptotic geometry of cones. Denote
\begin{equation}
P_0 := \Rp\times\{0\}\subset\Rpq\ ,\myeqnum{\nexteqnno[DefinitionOfPzero]}
\end{equation}
and let $P$ denote the totally-geodesic subspace in $\H$ tangent to $P_0$ at $x_0$.

\begin{lemma}\label{AsymptoticsOfCone}
Let $X$ be a spacelike sphere in $\Spmq$, and let $\seq[m]{x_m}$ be a divergent sequence in $\hat{X}$. If $\seq[m]{\alpha_m}$ is a sequence of isometries of $\H$ such that, for all $m$,
\begin{equation}\eqalign{
\alpha_m(x_m) &= x_0\ ,\ \text{and}\cr
D\alpha_m(T_{x_m}\hat{X}) &= P_0\ ,\cr}\label{eq:AsymptoticsOfCone}
\end{equation}
then $\seq[m]{\alpha_m(\hat{X})}$ converges in the $C^\infty_\oploc$ sense to $P$, and $\seq[m]{\alpha_m(X_\infty)}$ converges in the Hausdorff sense to $\partial_\infty P$. In particular, $\hat{X}$ has bounded geometry.
\end{lemma}

\begin{proof} Note first that, upon suitably modifying $\hat X$ in an arbitrarily small neighbourhood of $x_0$, we may suppose that this submanifold is smooth and spacelike. For all $m$, denote $Y_m:=\alpha_m(X)$. By Lemma \ref{RadialCone} and Lemma \ref{lemma:EntireGraph} $(1)$, for all $m$, $Y_m$ is an entire graph of some smooth $1$-Lipschitz function $\varphi_m$, say. By \eqref{eqn:DecayOfSecondFF} and Lemma \ref{cor:loc unif spacelike} $(1)$, the sequence $\seq[m]{Y_m}$ is locally uniformly spacelike. By \eqref{eqn:DecayOfSecondFF} again, and reasoning as in the proof of Theorem \ref{theorem:PropernessOfTau}, we see that the derivatives of $\seq[m]{\varphi_m}$ to all orders are locally uniformly bounded, so that this sequence is compact in the $C^\infty_\oploc$ sense. Let $\varphi_\infty$ be an accumulation point and let $Y_\infty$ denote its graph. Since the intrinsic distance from $x_m$ to $x_0$ tends to infinity, by \eqref{eqn:DecayOfSecondFF} again, $Y_\infty$ is totally-geodesic, and since $T_{x_0}Y_\infty=P_0$, it follows that $Y_\infty=P$. The sequence $\seq[m]{Y_m}$ is thus $C^\infty_\oploc$-precompact with $P$ as its unique accumulation point, and therefore converges in the $C^\infty_\oploc$ sense to $P$, as desired. Finally, since $\mathcal{E}_+$ is precompact in the Hausdorff topology, this sequence also converges in the Hausdorff sense to $P$. It follows by continuity of the asymptotic boundary operator $\partial_\infty$ that $\seq[m]{\alpha_m(X_\infty)}=\seq[m]{\partial_\infty\alpha_m(\hat{X})}$ converges in the Hausdorff sense to $\partial_\infty P$, and this completes the proof.
\end{proof}

It remains only to describe how cones model the non-compact ends of maximal graphs. For all $r_0>0$, let $\hat{X}_{r_0}$ denote the image of $X\times(r_0,\infty)$ under $\Phi$. We call $\hat{X}_{r_0}$ the \emph{truncated cone} over $X$ of \emph{inner radius} $r_0$.

\begin{lemma}
Let $X$ be a smooth spacelike sphere in $\Spmq$ and let $M$ be a complete maximal $p$-submanifold in $\H_+$ such that $\partial_\infty M=X_\infty$. There exists a compact subset $K\subseteq M$, $r_0>0$ and a smooth section $\sigma\in\Gamma(\N \hat{X}_{r_0})$ such that $M\setminus K$ is the graph of $\sigma$ over $\hat{X}_{r_0}$. Furthermore, for all $k$, $\|J^k\sigma(x)\|$ tends to zero as $x$ tends to infinity. In particular, $M$ has bounded geometry.
\proclabel{AsymptoticsOfMaximalGraph}
\end{lemma}

\begin{proof} It will suffice to show that with $\seq[m]{x_m}$ and $\seq[m]{\alpha_m}$ as in Lemma \ref{AsymptoticsOfCone}, $\seq[m]{\alpha_m(M)}$ also converges in the $C^\infty_\oploc$ sense to $P$. Indeed, when this holds, there exists $r_0>0$ such that, for all $(x,r)\in\hat{X}_{r_0}$, a portion of $M$ is the graph of some normal section $\sigma_{x,r}$ over the unit ball about this point. We now claim that the family $(\sigma_{x,r})_{r>r_0}$ joins together to yield a normal section over $\hat{X}_{r_0}$. Indeed, note first that, upon choosing a suitable Fermi chart, any totally-geodesic timelike sphere can be represented as $\{0\}\times\S^q$. Since $M$ is an entire graph in any such chart, it intersects every such sphere exactly once, so that, for all $(x,r)$ and $(x',r')$, the sections $\sigma_{x,r}$ and $\sigma_{x',r'}$ coincide over the intersection $B_1(x,r)\minter B_1(x',r')$, and the family $(\sigma_{x,r})_{r>r_0}$ joins together to yield a normal section $\sigma$ over $\hat{X}_{r_0}$, as asserted. The graph of this section is trivially the complement of some compact subset of $M$, and the result follows.

We now prove the assertion. Note that any accumulation point $M_\infty$ of the sequence $\seq[m]{\alpha_m(M)}$ satisfies $\partial_\infty M_\infty=\partial_\infty P$, and therefore contains no complete lightlike ray. It follows by Theorem \ref{thm:degenerate} that $\seq[m]{\alpha_m(M)}$ is relatively compact in the $C^\infty_\oploc$ topology and that every accumulation point $M_\infty$ is a complete maximal graph satisfying $\partial_\infty M_\infty=\partial_\infty P$, so that, by uniqueness, $M_\infty=P$. In other words, $\seq[m]{\alpha_m}$ is $C^\infty_\oploc$ pre-compact with $P$ as its unique accumulation point, and therefore converges in the $C^\infty_\oploc$ sense to $P$, as desired.\end{proof}

\newsubhead{Elliptic estimates I - differential operators over the line}[DifferentialOperatorsOverTheLine]

We now determine elliptic estimates for the actions of the Jacobi operator on Sobolev and H\"older spaces defined over cones and maximal graphs. As explained in the introduction to this section, pre-elliptic estimates will first be obtained using general properties of elliptic operators over manifolds of bounded geometry (see, for example, Lemmas \ref{proc:SobolevEstimateB} and \ref{proc:HoelderEstimateB}), and the main challenge will lie in transforming these pre-elliptic estimates into elliptic estimates. This will follow from a standard estimate for solutions of a certain family of ordinary differential operators defined over an unbounded interval which we now study.

The main idea of this section is well illustrated by the operator
\begin{equation}
L:=\partial_x^2-1\ ,
\end{equation}
defined over $C^2([0,\infty))$. The general solution of the equation $Lu=0$ is
\begin{equation}\label{eqn:GeneralSolutionOneDimProblem}
u(x):=Ae^x + Be^{-x}\ ,
\end{equation}
so that if $u$ is a {\sl bounded} solution of $Lu=0$, then
\begin{equation}
\|u\|_{C^0}\leq\left|u(0)\right|\ .
\end{equation}
That is, when $u$ is bounded, its supremum over the {\sl non-compact} interval $[0,\infty)$ is bounded in terms of its value over the {\sl compact} set $\{0\}$. It is this simple yet powerful idea that underlies the estimates derived below.

It will be helpful to recall the following definition from the theory of partial differential equations (see \cite{CIL}). Let $f:(a,b)\rightarrow\R$ be a continuous function, and let $P$ be a second-order differential operator. For all $c\in(a,b)$ and $v\in\R$, we say that $(Pf)(c)\geq v$ in the \emph{viscosity sense} whenever it has the property that, for all smooth $g:(a,b)\rightarrow\R$ such that $g\geq f$ and $g(c)=f(c)$,
\begin{equation}
(Pg)(c)\geq v\ .
\end{equation}
The theory of viscosity solutions is designed to be the most general framework in which the pointwise maximum principle applies.

\begin{lemma}
Let
\begin{equation}
P:=P(x,\xi):=\xi^2+a(x)\xi+b(x)
\end{equation}
be a $C^0$ family of monic quadratic polynomials in $\xi$ such that
\begin{equation}
\mlim_{x\rightarrow+\infty}P(x,\xi) =: P_\infty(\xi) =: \xi^2 + a_\infty\xi + b_\infty\ .\myeqnum{\nexteqnno[BasicOneDimEstimateI]}
\end{equation}
If $P_\infty(0)<-\delta$, for some $\delta>0$, then there exists $C>0$ such that for any $A\geq 0$ and for any bounded function $u:[0,\infty)\rightarrow[0,\infty)$ satisfying
\begin{equation}
P(x,\partial_x)u \geq -A~,\myeqnum{\nexteqnno[BasicOneDimEstimateII]}
\end{equation}
in the viscosity sense,
\begin{equation}
\|u|_{[C,\infty)}\|_{L^\infty} \leq \frac{A}{\delta} + \left|u(C)\right|~.\myeqnum{\nexteqnno[BasicOneDimEstimateIII]}
\end{equation}
\proclabel{BasicOneDimEstimate}
\end{lemma}

\myproof
Let $\alpha>0$ be such that $P_\infty(\xi)< -\delta$ for all $\xi \in [-2\alpha,2\alpha]$. Let $C>0$ be such that, for all $x\geq C$, and for all $\xi\in[-\alpha,\alpha]$,
\begin{equation*}
P(x,\xi)< -\delta\ .
\end{equation*}
In particular, for all such $x$,
\begin{equation*}
b(x)=P(x,0)<0\ .
\end{equation*}
For all $R\in\R$, and for all $x\geq C$,
\begin{equation*}
\multiline{
P(x,\partial_x)e^{\alpha x+R} = P(x,\alpha)e^{\alpha x+R} \leq -\delta e^{\alpha x+R}\ ,\ \text{and}\cr
P(x,\partial_x)e^{-\alpha x+R} = P(x,-\alpha)e^{-\alpha x+R} \leq -\delta e^{-\alpha x+R}\ ,\cr}
\end{equation*}
so that
\begin{equation*}
P(x,\partial_x)\opCosh(\alpha x + R) \leq -\delta\opCosh(\alpha x + R)~.
\end{equation*}
We now define
\begin{equation*}
M(R) := \max \left\{\frac{\left|u(C)\right|}{\cosh(R)},\frac{A}{\delta}\right\}~.
\end{equation*}
The function $v(x,R):=u(x)-M(R)\cosh(\alpha x+R)$ then satisfies
\begin{enumerate}
\item $v(C,R)\leq 0$,
\item $\lim_{x\to \infty} v(x,R)=-\infty$ , and
\item for any $x\geq C$, in the viscosity sense,
\begin{eqnarray*}
P(x,\partial_x)v(x,R) & = & P(x,\partial_x)u(x) - M(R)P(x,\partial_x)\cosh(\alpha x+R) \\
& \geq & -A + \delta M(R)\cosh(\alpha x+R) \\
& \geq & 0~.
\end{eqnarray*}
\end{enumerate}

Fix $R$ and let $x_0$ be a maximum of $v(\cdot,R)$ in $[C,+\infty)$. If $x_0>C$, then, denoting by $f$ the constant function equal to $v(x_0,R)$,
\begin{equation*}
b(x_0)v(x_0,R) = P(x,\partial_x)f(x)|_{x=x_0} \geq P(x,\partial_x)\nu(x,R)|_{x=x_0} \geq 0\ .
\end{equation*}
Since $b(x_0)<0$, it follows that
\begin{equation*}
\sup_{x\geq C}\nu(x)\leq\nu(x_0)\leq 0\ ,
\end{equation*}
so that, for all $x\geq C$,
\begin{equation*}
u(x) \leq M(R)\opCosh(\alpha x + R)~.
\end{equation*}
Setting $R=-\alpha x$ now yields, for all $x\geq C$,
\begin{equation*}
u(x) \leq M(R) \leq \frac{A}{\delta}+\left|u(C)\right|\ ,
\end{equation*}
as desired.\myqed

\noindent The following corollary will be useful for obtaining estimates over H\"older spaces.

\begin{corollary}
Under the hypotheses of Lemma \ref{proc:BasicOneDimEstimate}, there exists $C>0$ such that, for all $A,B>0$, and for all bounded $u:[0,\infty)\rightarrow[0,\infty)$ satisfying
\begin{equation}
P(x,\partial_x)u \geq -A-B\|u\|_{L^\infty}^{\frac{1}{2}}\ ,\myeqnum{\nexteqnno[CorBasicOneDimEstimateI]}
\end{equation}
in the viscosity sense,
\begin{equation}
\|u\|_{L^\infty} \leq \frac{2A}{\delta} + \frac{B^2}{\delta^2} + 2\|u|_{[0,C]}\|_{L^\infty}\ .\myeqnum{\nexteqnno[CorBasicOneDimEstimateII]}
\end{equation}
\proclabel{CorBasicOneDimEstimate}
\end{corollary}

\myproof Indeed, by Lemma \procref{BasicOneDimEstimate}, there exists $C>0$ such that, for all such $A$, $B$ and $u$,
\begin{equation*}
\|u\|_{L^\infty} \leq \frac{A}{\delta} + \frac{B}{\delta}\|u\|_{L^\infty}^{\frac{1}{2}} + \|u|_{[0,C]}\|_{L^\infty}\ .
\end{equation*}
However, by the algebraic-geometric mean inequality,
\begin{equation*}
\frac{B}{\delta}\|u\|_{L^\infty}^{\frac{1}{2}} \leq \frac{B^2}{2\delta^2} + \frac{1}{2}\|u\|_{L^\infty}\ ,
\end{equation*}
and the result follows upon combining these relations.\myqed

\noindent The following variant of Lemma \ref{proc:BasicOneDimEstimate} which will serve to obtain estimates over Sobolev spaces.

\begin{lemma}
Let
\begin{equation}
P:=P(x,\xi):=\xi^2 + a(x)\xi + b(x)
\end{equation}
be a $C^1$ family of monic quadratic polynomials in $\xi$ such that
\begin{equation}\eqalign{
\mlim_{x\rightarrow+\infty}P(x,\xi) &=:P_\infty(\xi)=:\xi^2+a_\infty\xi + b_\infty\ ,\ \text{and}\cr
\mlim_{x\rightarrow+\infty}a'(x) &= 0\ .\cr}\myeqnum{\nexteqnno[BasicOneDimEstimateIIIA]}
\end{equation}
If $P_\infty(0)<-\delta$, for some $\delta>0$, then there exists $C>0$ such that, for all smooth, integrable $u,v:[0,\infty)\rightarrow[0,\infty)$, if
\begin{equation}
P(x,\partial_x)u \geq -v\ ,\myeqnum{\nexteqnno[BasicOneDimEstimateIIIB]}
\end{equation}
then
\begin{equation}
\|u\|_{L^1} \leq \frac{2}{\delta}\|v\|_{L^1} + \frac{2}{\delta}(2+|a_\infty|)\|u|_{[0,C]}\|_{W^{1,1}}\ .\myeqnum{\nexteqnno[BasicOneDimEstimateIIIC]}
\end{equation}
\proclabel{BasicOneDimEstimateIII}
\end{lemma}

\myproof Choose $x_0\in(0,\infty)$ such that, for all $x\geq x_0$,
\begin{equation}\label{eqn:SizeOfAAndB}
\left|a(x)-a_\infty\right|<1-\frac{\delta}{2},\ \left|a'(x)\right|<\frac{\delta}{4}\ \text{and}\ \left|b(x)-b_\infty\right|<\frac{\delta}{8}\ .
\end{equation}
Define
\begin{equation*}
f(x) := \int_{x_0}^x u(t)dt\ .
\end{equation*}
We will show that $f$ satisfies a second order differential relation of the type addressed by Lemma \ref{proc:BasicOneDimEstimate}. Indeed, upon integrating by parts, we obtain
\begin{equation*}\eqalign{
\int_{x_0}^xP(t,\partial_t)u(t)dt
&=\int_{x_0}^xu''(t) + a(t)u'(t) + b(t)u(t)dt\cr
&=[u'(t)]_{x_0}^x + [a(t)u(t)]_{x_0}^x - \int_{x_0}^xa'(t)u(t)dt + b(x)f(x) - \int_{x_0}^x(b(x)-b(t))u(t)dt\cr
&=f''(x)-u'(x_0)+a(x)f'(x)-a(x_0)u(x_0)\vphantom{\frac{1}{2}}\cr
&\qquad-\int_{x_0}^xa'(t)u(t)dt + b(x)f(x) - \int_{x_0}^x(b(x)-b(t))u(t)dt\cr
&=P(x,\partial_x)f - u'(x_0) - a(x_0)u(x_0) - \int_{x_0}^xa'(t)u(t)dt - \int_{x_0}^x(b(x)-b(t))u(t)dt\ .\cr}
\end{equation*}
Thus, bearing in mind \eqref{eqn:SizeOfAAndB},
\begin{equation*}
P(x,\partial_x)f(x) \geq -\|v\|_{L^1} - \left|u'(x_0)\right| - \left|a(x_0)\right|\left|u(x_0)\right| - \frac{\delta}{2}\|u|_{[x_0,\infty)}\|_{L^1}\ .
\end{equation*}
Noting that $f(x_0)=0$, it follows by Lemma \procref{BasicOneDimEstimate} that, upon increasing $x_0$ further if necessary,
\begin{equation*}
\|u|_{[x_0,\infty)}\|_{L^1} = \|f\|_{L^\infty} \leq \frac{1}{\delta}\big(\|v\|_{L^1} + \left|u'(x_0)\right| + \left|a(x_0)\right|\left|u(x_0)\right|\big) + \frac{1}{2}\|u|_{[x_0,\infty)}\|_{L^1}\ ,
\end{equation*}
so that
\begin{equation*}
\|u|_{[x_0,\infty)}\|_{L^1} \leq \frac{2}{\delta}\big(\|v\|_{L^1} + \left|u'(x_0)\right| + \left|a(x_0)\right|\left|u(x_0)\right|\big)\ .
\end{equation*}
It remains only to reformulate the right-hand side in terms of integral norms. Upon averaging, we obtain
\begin{equation*}
\|u|_{[x_0+1,\infty)}\|_{L^1} \leq \int_0^1\|u|_{[x_0+s,\infty)}\|_{L^1}ds
\leq\frac{2}{\delta}\|v\|_{L^1} + \frac{2}{\delta}\bigg(2-\frac{\delta}{2}+\left|a_\infty\right|\bigg)\|u|_{[x_0,x_0+1]}\|_{W^{1,1}}\ .
\end{equation*}
Finally,
\begin{equation*}
\|u|_{[0,x_0+1]}\|_{L^1} \leq \|u|_{[0,x_0+1]}\|_{W^{1,1}}\ ,
\end{equation*}
and the result follows upon combining these last two relations.\myqed

\noindent Applying the algebraic-geometric mean inequality as in the proof of Corollary \ref{proc:CorBasicOneDimEstimate}, we obtain the following useful variant of Lemma \ref{proc:BasicOneDimEstimateIII}.

\begin{corollary}
Under the hypotheses of Lemma \ref{proc:BasicOneDimEstimateIII}, there exists $C>0$ such that, for all $A,B>0$, and for all smooth, integrable $u,v:[0,\infty)\rightarrow[0,\infty)$, if
\begin{equation}
P(x,\partial_x)u \geq -u^{\frac{1}{2}}v^{\frac{1}{2}}\ ,\myeqnum{\nexteqnno[CorBBasicOneDimEstimateI]}
\end{equation}
then
\begin{equation}
\|u\|_{L^1} \leq \frac{4}{\delta^2}\|v\|_{L^1} + \frac{4}{\delta}(2+|a_\infty|)\|u|_{[0,C]}\|_{W^{1,1}}\ .\myeqnum{\nexteqnno[CorBBasicOneDimEstimateII]}
\end{equation}
\proclabel{CorBBasicOneDimEstimate}
\end{corollary}

\newsubhead{Elliptic estimates II - weighted Sobolev spaces}[WeightedSobolevSpaces]

We now introduce the weighted function spaces that we will use to construct perturbations of maximal graphs. We first study weighted Sobolev spaces, where we will be able to prove invertibility using duality arguments. This will then be used in the next section as the basis for proving invertibility in the H\"older space case. Recall that, given a spacelike sphere $X$ in $\T^1_{x_0} \H$, the corresponding cone is $\hat{X} = \Phi\big(X\times (0,+\infty)\big)$ and the truncated cone is $\hat{X}_{r_0} = \Phi\big( X\times [r_0,+\infty)\big)$. For $\omega\in\R$, let $\mu_\omega$ denote the operator of multiplication by $e^{\omega r}$, acting on $\Gamma(\N\hat X)$ or $\Gamma(\N \hat{X}_{r_0})$. For $k\in\mathbf{N}$, $l\in[1,\infty)$, and $\omega\in\mathbf{R}$, we define the $\omega$-\emph{weighted Sobolev norm} by
\begin{equation}
\|\sigma\|_{W^{k,l}_\omega} := \|\mu_{-\omega}\sigma\|_{W^{k,l}}\ .\myeqnum{\nexteqnno[WeightedSobolevNorm]}
\end{equation}
For all $k\in\mathbf{N}$, $l\in[1,\infty)$, and $\omega\in\mathbf{R}$, we define the $\omega$-\emph{weighted Sobolev space} $W^{k,l}_\omega(\N \hat{X}_{r_0})$ to be the completion of $C^\infty_0(\N\hat{X}_{r_0})$ with respect to the $\omega$-weighted Sobolev norm. For all such $k$, $l$ and $\omega$, let $W^{k,l}_{\omega,0}(\N \hat{X}_{r_0})$ denote the subspace of $W^{k,l}_\omega(\N \hat{X}_{r_0})$ consisting of those sections which vanish over $X\times\{r_0\}$.
The key to understanding the theory of weighted function spaces lies in the observation that, for all $k$, $p$ and $\omega$, the map
\begin{equation}
\mu_\omega:W^{k,l}(\N \hat{X}_{r_0})\rightarrow W^{k,l}_\omega(\N \hat{X}_{r_0})\myeqnum{\nexteqnno[LinearIsomorphism]}
\end{equation}
is an isometric isomorphism with inverse $\mu_{-\omega}$. Functional properties of any operator $L$ over the weighted space $W^{k,l}_\omega(\N \hat{X}_{r_0})$ are thus determined by studying those of its conjugate $L_\omega:=\mu_{-\omega}L\mu_\omega$ over the unweighted space $W^{k,l}(\N \hat{X}_{r_0})$.

The aim of this section is to prove the following result.

\begin{theorem}
For all $r_0>0$, and for every weight $\omega\in(-\sqrt{p},\sqrt{p})$, the operator
\begin{equation}
J:W^{2,2}_{\omega,0}(\N \hat{X}_{r_0})\rightarrow L^2_\omega(\N \hat{X}_{r_0})\myeqnum{\nexteqnno[InvertibilityOverEndSobolev]}
\end{equation}
is a linear isomorphism.
\proclabel{InvertibilityOverEndSobolev}
\end{theorem}

\noindent Since by Lemma \ref{proc:AsymptoticsOfMaximalGraph} every complete maximal graph with smooth spacelike asymptotic boundary is asymptotic to a cone, estimates for Jacobi operators over maximal graphs will be obtained upon perturbing this result.

To prove Theorem \procref{InvertibilityOverEndSobolev}, we first note that, by Lemma \ref{proc:JacobiOperator}, Lemma \ref{lemma:CovariantDerivativeInPolarCoordinates} and Lemma \ref{RadialCone},
\begin{equation}
J\sigma = \opSinh^{-2}(r)\Delta^N\sigma + (\nabla^N_{\partial_r})^2\sigma + (p-1)\opCoth(r)(\nabla^N_{\partial_r})\sigma - p\sigma + \sum_{m,n}\g(\hat{\opII}(e_m,e_n),\sigma)\hat{\opII}(e_m,e_n)\ ,\myeqnum{\nexteqnno[JacobiOperatorInPolarCoordinates]}
\end{equation}
where $\Delta^N$ here denotes the Laplace operator of $\N X$, and $e_1,\cdots,e_p$ is an arbitrary orthonormal frame of the tangent bundle of $\hat{X}$. We thus consider the simpler operator
\begin{equation}
L\sigma := \opSinh^{-2}(r)\Delta^N\sigma + P\big(r,\nabla_{\partial_r}^N\big)\sigma\ ,\myeqnum{\nexteqnno[ApproximateOperator]}
\end{equation}
where
\begin{equation}
P(r,\xi) := \xi^2 + (p-1)\opCoth(r)\xi - p\ ,\myeqnum{\nexteqnno[PolynomialPart]}
\end{equation}
and, for all $\omega\in\mathbf{R}$, we define the conjugate operator
\begin{equation}
L_\omega\sigma := \mu_{-\omega}L\mu_\omega\sigma = \opSinh^{-2}(r)\Delta^N\sigma + P\big(r,\nabla_{\partial_r}^N+\omega\big)\sigma\ .\myeqnum{\nexteqnno[ConjugateApproximateOperator]}
\end{equation}
The following result provides the key to transforming pre-elliptic estimates into elliptic estimates in the present Sobolev space case.

\begin{lemma}
For all $r_0>0$, and for every weight $\omega\in(-\sqrt{p},\sqrt{p})$, there exists $C>0$ such that, for all $\sigma\in W^{2,2}(\N \hat{X}_{r_0})$,
\begin{equation}
\|\sigma\|_{L^2} \leq C\big(\|\sigma|_{X\times[r_0,C]}\|_{W^{1,2}} + \|L_\omega\sigma\|_{L^2}\big)\ .\myeqnum{\nexteqnno[SobolevEstimateA]}
\end{equation}
\proclabel{SobolevEstimateA}
\end{lemma}

\begin{proof}Define $f,g:[r_0,\infty)\rightarrow[0,\infty)$ by
\begin{equation*}\eqalign{
f(r) &:= \int_X \|\sigma(x,r)\|^2\opSinh(r)^{p-1}dx\ ,\ \text{and}\cr
g(r) &:= \int_X \|L_\omega\sigma(x,r)\|^2\opSinh(r)^{p-1}dx\ .\cr}
\end{equation*}
It will suffice to show that $f$ satisfies a second-order ordinary differential relation of the type used in Corollary \ref{proc:CorBBasicOneDimEstimate}. Differentiating $f$ twice under the integral and applying the product rule yields
\begin{equation*}\eqalign{
f'(r) &= \int_X 2\langle\nabla^N_{\partial_r}\sigma,\sigma\rangle\opSinh(r)^{p-1}dx + (p-1)\opCoth(r)f(r)\ ,\ \text{and}\cr
f''(r) &= \int_X 2\langle(\nabla^N_{\partial_r})^2\sigma,\sigma\rangle\opSinh(r)^{p-1}dx + \int_X 2\langle\nabla^N_{\partial_r}\sigma,\nabla^N_{\partial_r}\sigma\rangle\opSinh(r)^{p-1}dx\cr
&\qquad\qquad+\int_X 2\langle(p-1)\opCoth(r)\nabla^N_{\partial_r}\sigma,\sigma\rangle\opSinh(r)^{p-1}dx-\frac{(p-1)}{\opSinh(r)^2}f(r) + (p-1)\opCoth(r)f'(r)\cr
&\geq \int_X 2\langle(\nabla^N_{\partial_r})^2\sigma,\sigma\rangle\opSinh(r)^{p-1}dx + \int_X 4\langle(p-1)\opCoth(r)\nabla^N_{\partial_r}\sigma,\sigma\rangle\opSinh(r)^{p-1}dx\cr
&\qquad\qquad -\frac{(p-1)}{\opSinh(r)^2}f(r)+(p-1)^2\opCoth(r)^2f(r)\ .\cr}
\end{equation*}
Consider now the function
\begin{equation*}
Q_\omega(r,\xi) := \xi^2 - ((p-1)\opCoth(r) - 2\omega)\xi + \bigg(2\omega^2 + \frac{(p-1)}{\opSinh(r)^2} - 2p\bigg),
\end{equation*}
Using the preceding formulae for the derivatives of $f$, we obtain
\begin{equation*}\eqalign{
Q_\omega(r,\partial_r)f &\geq
\int_X 2\left\langle(\nabla^N_{\partial_r})^2\sigma + \big((p-1)\opCoth(r)+2\omega\big)\nabla^N_{\partial_r}\sigma + \big(\omega^2 + \omega(p-1)\opCoth(r)-p\big)\sigma,\sigma\right\rangle\opSinh(r)^{p-1}dx\cr
&=\int_X 2\left\langle P\big(r,\nabla^N_{\partial_r}+\omega\big)\sigma,\sigma\right\rangle\opSinh(r)^{p-1}dx\ .\cr}
\end{equation*}
Applying \eqref{eqn:ConjugateApproximateOperator} now yields
\begin{equation*}
Q_\omega(r,\partial_r)f \geq \int_X 2\langle L_\omega\sigma,\sigma\rangle\opSinh(r)^{p-1}dx - \int_X 2\langle\Delta^N\sigma,\sigma\rangle\opSinh(r)^{p-3}dx\ ,
\end{equation*}
so that, upon integrating by parts,
\begin{equation*}\eqalign{
Q_\omega(r,\partial_r)f &\geq \int_X 2\langle L_\omega\sigma,\sigma\rangle\opSinh(r)^{p-1}dx + \int_X 2\langle\nabla^N\sigma,\nabla^N\sigma\rangle\opSinh(r)^{p-3}dx\cr
&\geq\int_X 2\langle L_\omega\sigma,\sigma\rangle\opSinh(r)^{p-1}dx\cr
&\geq-2f^{\frac{1}{2}}g^{\frac{1}{2}}\ .\vphantom{\frac{1}{2}}\cr}
\end{equation*}
Since $\omega\in(-\sqrt{p},\sqrt{p})$, $Q_\omega$ satisfies the hypotheses of Corollary \ref{proc:CorBBasicOneDimEstimate}, and there therefore exists $C>0$ such that
\begin{equation*}
\|\sigma\|^2_{L^2} = \|f\|_{L^1} \leq C\big(\|g\|_{L^1} + \|f|_{[r_0,C]}\|_{W^{1,1}}\big) = C\big(\|L_\omega\sigma\|^2_{L^2} + \|f|_{[r_0,C]}\|_{W^{1,1}}\big)\ .
\end{equation*}
Finally,
\begin{equation*}
\|f|_{[r_0,C]}\|_{W^{1,1}} = \|f|_{[r_0,C]}\|_{L^1} + \|f'|_{[r_0,C]}\|_{L^1} = \|\sigma|_{X\times[r_0,C]}\|^2_{L^2} + \|f'|_{[r_0,C]}\|_{L^1}\ ,
\end{equation*}
whilst
\begin{equation*}\eqalign{
\|f'|_{[r_0,C]}\|_{L^1}
&=\int_{r_0}^C\bigg|\int_X 2\langle\nabla^N_{\partial_r}\sigma,\sigma\rangle\opSinh^{(p-1)}(r)dx + (p-1)\opCoth(r)f(r)\bigg|dr\cr
&\leq 2\|\nabla^N_{\partial_r}\sigma|_{X\times[r_0,C]}\|_{L^2}\|\sigma_{X\times[r_0,C]}\|_{L^2} + (p-1)\opCoth(r_0)\|\sigma|_{X\times[r_0,C]}\|^2_{L^2}\ .\vphantom{\bigg(}\cr}
\end{equation*}
Upon applying the algebraic-geometric mean inequality to the first term, we thus obtain
\begin{equation*}
\|f'|_{[r_0,C]}\|_{L^1}\leq \|\nabla^N_{\partial_r}\sigma|_{[r_0,C]}\|^2_{L^2} + p\opCoth(r_0)\|\sigma|_{X\times[r_0,C]}\|^2_{L^2}\ ,
\end{equation*}
and the result follows upon combining these estimates.
\end{proof}

\begin{lemma}
\noindent For all $r_0>0$, and for any weight $\omega\in(-\sqrt{p},\sqrt{p})$, there exists $C>0$ such that, for all $\sigma\in W^{2,2}_{0,0}(\N \hat{X}_{r_0})$,
\begin{equation}
\|\sigma\|_{W^{2,2}}\leq C\big(\|\sigma|_{X\times[r_0,C]}\|_{W^{1,2}} + \|L_\omega\sigma\|_{L^2}\big)\ .\myeqnum{\nexteqnno[SobolevEstimateB]}
\end{equation}
\proclabel{SobolevEstimateB}
\end{lemma}

\myproof Indeed, since $\hat{X}_{r_0}$ is of bounded geometry, it follows by the classical theory of elliptic operators (see, for example, \cite[Lemma 7.18]{Evans}) that there exists $C>0$ such that, for all $(x,r)\in\hat{X}_{r_0}$,
\begin{equation*}
\|\sigma|_{B_1(x,r)}\|_{W^{2,2}} \leq C\big(\|\sigma|_{B_2(x,r)}\|_{L^2} + \|L_\omega\sigma|_{B_2(x,r)}\|_{L^2}\big)\ .
\end{equation*}
Let $\seq[m]{(x_m,r_m)}$ be a sequence in $\hat{X}_{r_0}$ such that the unit balls about its points cover $\hat{X}_{r_0}$. Since $\hat{X}_{r_0}$ has bounded geometry, we may suppose that there exists $N>0$ such that any point $(y,s)$ of $\hat{X}_{r_0}$ lies at a distance of less than $2$ from at most $N$ elements of this sequence. Then
\begin{equation*}
\|\sigma\|^2_{W^{2,2}}\leq\sum_{m=1}^\infty\|\sigma|_{B_1(x_m,r_m)}\|^2_{W^{2,2}}
\leq C\sum_{m=1}^\infty\big(\|\sigma|_{B_2(x_m,r_m)}\|^2_{L^2}+\|L_\omega\sigma|_{B_2(x_m,r_m)}\|^2_{L^2}\big)
\leq NC\big(\|\sigma\|^2_{L^2}+\|L_\omega\sigma\|^2_{L^2}\big)\ ,\vphantom{\frac{1}{2}}
\end{equation*}
and the result now follows upon combining this estimate with \eqnref{SobolevEstimateA}.\myqed

\begin{lemma}
\noindent For all $r_0>0$, and for any weight $\omega\in(-\sqrt{p},\sqrt{p})$, there exists $C>0$ such that, for all $\sigma\in W^{2,2}_{0,0}(\N \hat{X}_{r_0})$,
\begin{equation}
\|\sigma\|_{W^{2,2}} \leq C\big(\|\sigma|_{X\times[r_0,C]}\|_{W^{1,2}} + \|J_\omega\sigma\|_{L^2}\big)\ .\myeqnum{\nexteqnno[SobolevEstimateC]}
\end{equation}
\proclabel{SobolevEstimateC}
\end{lemma}

\myproof Let $\chi:\mathbf{R}\rightarrow[0,1]$ be a smooth, increasing function equal to $0$ over $(-\infty,-1]$ and equal to $1$ over $[0,\infty)$. For all $r_1\in\mathbf{R}$, denote $\chi_{r_1}(r):=\chi(r-r_1)$, and denote
\begin{equation*}
L_{\omega,r_1} := \chi_{r_1} J_\omega + (1-\chi_{r_1})L_\omega = L_\omega + \chi_{r_1}(J_\omega - L_\omega)\ .
\end{equation*}
By \eqref{eqn:JacobiOperatorInPolarCoordinates} and \eqref{eqn:ApproximateOperator}, $L_{\omega,r_1}$ converges to $L_\omega$ in the operator norm as $r_1$ tends to infinity. It follows by stability of elliptic estimates that, for sufficiently large $r_1$, an estimate of the form \eqnref{SobolevEstimateB} also holds for $L_{\omega,r_1}$. Choose $r_2>r_1+1$, and note that, for all $\sigma$, over $X\times[r_2,\infty)$,
\begin{equation*}
L_{\omega,r_1}\sigma = J_\omega\sigma\ .
\end{equation*}
It follows that, for suitable constants $C_1,C_2>0$, and for all $\sigma$,
\begin{equation*}\eqalign{
\|\sigma\|_{W^{2,2}}
&\leq \|(1-\chi_{r_2})\sigma\|_{W^{2,2}} + \|\chi_{r_2}\sigma\|_{W^{2,2}}\cr
&\leq C_1\big(\|\sigma|_{X\times[0,r_2]}\|_{W^{2,2}} + \|\chi_{r_2}\sigma|_{X\times[0,C_1]}\|_{W^{1,2}} + \|L_{\omega,r_1}\chi_{r_2}\sigma\|_{L^2}\big)\cr
&\leq C_2\big(\|\sigma|_{X\times[0,r_2]}\|_{W^{2,2}} + \|\chi_{r_2}\sigma|_{X\times[0,C_1]}\|_{W^{1,2}} + \|[L_{\omega,r_1},\mu_{r_2}]\sigma\|_{L^2} + \|\chi_{r_2}L_{\omega,r_1}\sigma\|_{L^2}\big)\ ,\cr}
\end{equation*}
where here $\mu_{r_2}$ denotes the operator of multiplication by $\chi_{r_2}$. Since this is a first order differential operator, upon increasing $C_2$ if necessary, we obtain
\begin{equation*}\eqalign{
\|\sigma\|_{W^{2,2}}&\leq C_2\big(\|\sigma|_{X\times[0,r_2]}\|_{W^{2,2}} + \|\sigma|_{X\times[0,C_2]}\|_{W^{1,2}} + \|\chi_{r_2}L_{\omega,r_1}\sigma\|_{L^2}\big)\cr
&= C_2\big(\|\sigma|_{X\times[0,r_2]}\|_{W^{2,2}} + \|\sigma|_{X\times[0,C_2]}\|_{W^{1,2}} + \|J_\omega\sigma\|_{L^2}\big)\ .\cr}
\end{equation*}
Finally, by the classical theory of elliptic operators, there exists $C_3>0$ such that
\begin{equation*}
\|\sigma|_{X\times[0,r_2]}\|_{W^{2,2}} \leq C_3\big(\|\sigma|_{X\times[0,r_2+1]}\|_{L^2} + \|J_\omega\sigma\|_{L^2}\big)\ ,
\end{equation*}
and the result follows upon combining these estimates.\myqed

\noindent We now prove Theorem \procref{InvertibilityOverEndSobolev}.

\begin{proof}[Proof of Theorem \procref{InvertibilityOverEndSobolev}.] By the Rellich-Kondrachov compactness theorem (see, for example, \cite[Section 5.7]{Evans}), the restriction operator
\begin{equation*}
R:W_{0,0}^{2,2}(\N \hat{X}_{r_0})\rightarrow W^{1,2}(\N \hat{X}_{r_0}|_{X\times[r_0,C]})
\end{equation*}
is compact. The estimate \eqnref{SobolevEstimateC} is thus of elliptic type, so that, as in \cite[Chapter 21, Theorem 4]{Lax}, $J_\omega$ has finite-dimensional kernel and closed image. Since, for all $\omega$, $J_\omega^*=J_{-\omega}$, it follows that $J_\omega$ is Fredholm. Since $J_0$ is self-adjoint, its Fredholm index is equal to zero, and since $(J_\omega)_{\omega\in(-\sqrt{p},\sqrt{p})}$ is a continuous family of Fredholm operators, it follows that the Fredholm index of $J_\omega$ is zero for all $\omega\in(-\sqrt{p},\sqrt{p})$.

It remains only to prove that $J_\omega$ is a linear isomorphism for all $\omega$. Consider first the case where $\omega\in(-\sqrt{p},0]$. It will now be preferable to work with the unconjugated operator $J$. Let $\sigma\in W^{2,2}_{\omega,0}(\N \hat{X}_{r_0})$ be an element of the kernel of $J$. We claim that $\sigma$ vanishes. Indeed, suppose that contrary. By elliptic regularity, $\sigma$ is smooth. Note also that, since $\omega\leq 0$, $\sigma$ is an element of the unweighted Sobolev space $W^{2,2}(\N\hat{X}_{r_0})$ and, since $\hat{X}_{r_0}$ has bounded geometry, it follows by standard elliptic bootstrapping arguments that $\|\sigma\|_{L^\infty}$ is finite. Since $\hat{X}_{r_0}$ is also complete, it follows by Omori's maximum principle that, for all $\epsilon>0$, there exists $(x,r)\in\hat{X}_{r_0}$ such that
\begin{equation*}\eqalign{
\opHess(\|\sigma\|^2)(x,r) &\leq \epsilon\opId\ ,\ \text{and}\cr
\|\sigma(x,r)\|^2 &\geq \|\sigma\|_{L^\infty}^2 - \epsilon\ ,\cr}
\end{equation*}
so that, by \eqref{eqn:JacobiOperator}, at this point,
\begin{equation*}
\epsilon p
\geq \Delta\|\sigma\|^2
= 2\langle\Delta^N\sigma,\sigma\rangle + 2\|\nabla^N\sigma\|^2\vphantom{\frac{1}{2}}
= 2p\|\sigma\|^2 + 2\|\nabla^N\sigma\|^2 + \sum_{m,n}\langle\opII(e_m,e_n),\sigma\rangle^2
\geq 2p\|\sigma\|_{L^\infty}^2 - 2p\epsilon\ .
\end{equation*}
Upon choosing $\epsilon<\frac{2}{3}\|\sigma\|_{L^\infty}^2$, we obtain a contradiction, and it follows that $\sigma$ vanishes, as asserted. The operator $J_\omega$ is thus injective and therefore invertible. Finally, by duality, for all $\omega\in[0,\sqrt{p})$, $J_\omega=J_{-\omega}^*$ is also invertible, and this completes the proof.
\end{proof}

\newsubhead{Elliptic estimates III - weighted H\"older spaces}[WeightedHoelderSpaces]

We now address the case of weighted H\"older spaces. For $k\in\mathbf{N}$, $\alpha\in(0,1)$, and $\omega\in\mathbf{R}$, we define the $\omega$-\emph{weighted H\"older norm} by
\begin{equation}
\|\sigma\|_{C^{k,\alpha}_\omega} := \|\mu_{-\omega}\sigma\|_{C^{k,\alpha}}\ ,\myeqnum{\nexteqnno[DefinitionOfWeightedHoelderNorm]}
\end{equation}
where, as before, for $\omega\in\mathbf{R}$, $\mu_\omega$ denotes the operator of multiplication by $e^{\omega r}$. For $k\in\mathbf{N}$, $\alpha\in(0,1)$, and $\omega\in\mathbf{R}$, we define the $\omega$-\emph{weighted H\"older space} $C^{k,\alpha}_\omega(\N \hat{X}_{r_0})$ to be the space of all $k$-times differentiable sections of $\N \hat{X}_{r_0}$ with finite $\omega$-weighted H\"older norm. For all such $k$, $\alpha$ and $\omega$, we denote by $C^{k,\alpha}_{\omega,0}(\N \hat{X}_{r_0})$ the closed subspace consisting of those sections which vanish along $X\times\{r_0\}$.

The aim of this section is to prove the following result.

\begin{theorem}
For all $r_0>0$, and for every weight $\omega\in(-p,0]$, the operator
\begin{equation}
J:C^{2,\alpha}_{\omega,0}(\N \hat{X}_{r_0})\rightarrow C^{0,\alpha}_\omega(\N \hat{X}_{r_0})\ ,\myeqnum{\nexteqnno[InvertibilityOverEndHoelder]}
\end{equation}
is a linear isomorphism.
\proclabel{InvertibilityOverEndHoelder}
\end{theorem}

\begin{remark}
Upon refining the arguments used in Section \ref{subhead:DifferentialOperatorsOverTheLine} and below, we may show that, for all $\omega\in(-p,1)$, every element of the kernel of $J$ in $C^{2,\alpha}_{\omega,0}(\hat{X}_{r_0})$ has exponential decay at infinity. The maximum principle then shows that, for all such $\omega$, $J$ has trivial kernel over $C^{2,\alpha}_{\omega,0}(\hat{X}_{r_0})$, and the isomorphism property then follows. Theorem \ref{proc:InvertibilityOverEndHoelder} as stated is, however, quite sufficient for our purposes.
\end{remark}

\noindent We continue to use the framework of the preceeding section. The following result provides the key to transforming pre-elliptic estimates into elliptic estimates in the H\"older space case.

\begin{lemma}
For all $r_0>0$, and for every weight $\omega\in(-p,1)$, there exists $C>0$ such that, for all $\sigma\in C^{2,\alpha}(\N \hat{X}_{r_0})$,
\begin{equation}
\|\sigma\|_{C^0} \leq C\big(\|\sigma|_{X\times[0,C]}\|_{C^0} + \|L_\omega\sigma\|_{C^0}\big)\ .\myeqnum{\nexteqnno[HoelderEstimateA]}
\end{equation}
\proclabel{HoelderEstimateA}
\end{lemma}

\myproof Consider the function $f:[r_0,\infty)\rightarrow[0,\infty)$ given by
\begin{equation*}
f(r) := \msup_{x\in X}\|\sigma(x,r)\|^2\ .
\end{equation*}
It will suffice to show that $f$ satisfies in the viscosity sense a second-order ordinary differential relation of the type used in Corollary \ref{proc:CorBasicOneDimEstimate}. Indeed, by the product rule,
\begin{equation*}\eqalign{
P(r,\partial_r+\omega)\|\sigma\|^2 &= 2\left\langle P\big(r,\nabla^N_{\partial_r}+\omega\big)\sigma,\sigma\right\rangle + 2\|\nabla^N_{\partial_r}\sigma\|^2 - P(r,\omega)\|\sigma\|^2\cr
& \geq 2\left\langle P\big(r,\nabla^N_{\partial_r}+\omega\big)\sigma,\sigma\right\rangle - P(r,\omega)\|\sigma\|^2\ .\cr}
\end{equation*}
Likewise, at every point $(x,r)\in\hat{X}_{r_0}$ maximizing $\|\sigma\|^2$ over $X\times\{r\}$,
\begin{equation*}
0\geq \Delta\|\sigma\|^2 = 2\langle\Delta^N\sigma,\sigma\rangle + 2\|\nabla^N\sigma\|^2 \geq 2\langle\Delta^N\sigma,\sigma\rangle\ .
\end{equation*}
Consider now a smooth function $g$ such that $g\geq f$ and $g(r)=f(r)$. Then, at any point $(x,r)$ such that $f(r)=\|\sigma(x,r)\|^2$,
\begin{equation*}\eqalign{
P(r,\partial_r+\omega)g
&\geq P(r,\partial_r+\omega)\|\sigma\|^2\cr
&\geq 2\big\langle\big(\opSinh^{-2}(r)\Delta^N + P\big(r,\nabla^N_{\partial_r}+\omega\big)\big)\sigma,\sigma\big\rangle - P(r,\omega)\|\sigma\|^2\cr
&=2\langle L_\omega\sigma,\sigma\rangle - P(r,\omega)g .\cr}
\end{equation*}
Consequently, in the viscosity sense,
\begin{equation*}
(P(r,\partial_r+\omega)+P(r,\omega))f \geq -2\|L_\omega\sigma\|_{C^0}f^{\frac{1}{2}}\ .
\end{equation*}
Since $\omega\in(-p,1)$, $Q(r,\xi):=P(r,\xi+\omega) - P(r,\omega)$ satisfies the hypothesis of Corollary \ref{proc:CorBasicOneDimEstimate}. There therefore exists $C>0$ such that
\begin{equation*}
\|\sigma\|_{L^\infty}^2
=\msup_{r>0}\left|f(r)\right|
\leq C^2\big(\|\sigma|_{X\times[r_0,C]}\|^2_{C^0} + \|L_\omega\sigma\|^2_{C^0}\big)\vphantom{\bigg(}
\leq C^2\big(\|\sigma|_{X\times[r_0,C]}\|_{C^0} + \|L_\omega\sigma\|_{C^0}\big)^2\ ,\vphantom{\bigg(}
\end{equation*}
as desired.\myqed

\begin{lemma}
For all $r_0>0$, and for every weight $\omega\in(-p,1)$, there exists $C>0$ such that, for all $\sigma\in C^{2,\alpha}_{0,0}(\N \hat{X}_{r_0})$,
\begin{equation}
\|\sigma\|_{C^{2,\alpha}} \leq C\big(\|\sigma|_{X\times[0,C]}\|_{C^0} + \|L_\omega\sigma\|_{C^{0,\alpha}}\big)\ .\myeqnum{\nexteqnno[HoelderEstimateB]}
\end{equation}
\proclabel{HoelderEstimateB}
\end{lemma}

\myproof Since $\hat{X}_{r_0}$ is of bounded geometry, by the classical theory of elliptic operators, there exists $C_1>0$ such that, for all $(x,r)\in\hat{X}_{r_0}$,
\begin{equation*}
\|\sigma|_{B_1(x)}\|_{C^{2,\alpha}}\leq C_1\big(\|\sigma|_{B_2(x)}\|_{C^0}+\|L_\omega\sigma|_{B_2(x)}\|_{C^{0,\alpha}}\big)\ .
\end{equation*}
Thus, for a suitable constant $C_2>0$,
\begin{equation*}
\|\sigma\|_{C^{2,\alpha}}
\leq C_2\sup_{x\in\hat{X}_{r_0}}\|\sigma|_{B_1(x)}\|_{C^{2,\alpha}}
\leq C_1C_2\sup_{x\in\hat{X}_{r_0}}\big(\|\sigma|_{B_2(x)}\|_{C^0}+\|L_\omega\sigma|_{B_2(x)}\|_{C^{0,\alpha}}\big)
\leq C_1C_2\big(\|\sigma\|_{C^0} + \|L_\omega\sigma\|_{C^{0,\alpha}}\big)\ ,
\end{equation*}
and the result follows upon combining this estimate with \eqnref{HoelderEstimateA}.\myqed

\begin{lemma}
For all $r_0>0$, and for every weight $\omega\in(-p,1)$, there exists $C>0$ such that, for all $\sigma\in C^{2,\alpha}_0(\N \hat{X}_{r_0})$,
\begin{equation}
\|\sigma\|_{C^{2,\alpha}} \leq C\big(\|\sigma|_{X\times[0,C]}\|_{C^0} + \|J_\omega\sigma\|_{C^{0,\alpha}}\big)\ .\myeqnum{\nexteqnno[HoelderEstimateC]}
\end{equation}
\proclabel{HoelderEstimateC}
\end{lemma}

\myproof Let $\chi:\mathbf{R}\rightarrow[0,1]$ be a smooth, increasing function equal to $0$ over $(-\infty,-1]$ and equal to $1$ over $[0,\infty)$. For all $r_1\in\mathbf{R}$, denote $\chi_{r_1}(r):=\chi(r-r_1)$, and denote
\begin{equation*}
L_{\omega,r_1} := \chi_{r_1} J_\omega + (1-\chi_{r_1})L_\omega\ .
\end{equation*}
By \eqref{eqn:JacobiOperatorInPolarCoordinates} and \eqref{eqn:ApproximateOperator}, $L_{\omega,r_1}$ converges to $L_\omega$ in the operator norm as $r_1$ tends to infinity. It follows by stability of elliptic estimates that, for sufficiently large $r_1$, an estimate of the form \eqnref{HoelderEstimateB} also holds for $L_{\omega,r_1}$. Choose $r_2>r_1+1$, and note that, for all $\sigma$, over $X\times[r_2,\infty)$,
\begin{equation*}
L_{\omega,r_1}\sigma = J_\omega\sigma\ .
\end{equation*}
It follows that, for suitable constants $C_1,C_2>0$, and for all $\sigma$,
\begin{equation*}\eqalign{
\|\sigma\|_{C^{2,\alpha}}
&\leq \|(1-\chi_{r_2})\sigma\|_{C^{2,\alpha}} + \|\chi_{r_2}\sigma\|_{C^{2,\alpha}}\cr
&\leq C_1\big(\|\sigma|_{X\times[0,r_2]}\|_{C^{2,\alpha}} + \|\chi_{r_2}\sigma|_{X\times[0,C_1]}\|_{C^0} + \|L_{\omega,r_1}\chi_{r_2}\sigma\|_{C^{0,\alpha}}\big)\cr
&\leq C_2\big(\|\sigma|_{X\times[0,C_2]}\|_{C^{2,\alpha}} + \|\sigma|_{X\times[0,C_2]}\|_{C^{1,\alpha}} + \|L_{\omega,r_1}\sigma|_{X\times[r_2,\infty)}\|_{C^{0,\alpha}}\big)\cr
&= C_2\big(\|\sigma|_{X\times[0,C_2]}\|_{C^{2,\alpha}} + \|J_\omega\sigma\|_{C^{0,\alpha}}\big)\ .\cr}
\end{equation*}
Finally, by the classical theory of elliptic operators, there exists $C_3>0$ such that
\begin{equation*}
\|\sigma|_{X\times [0,C_2]}\|_{C^{2,\alpha}} \leq C_3\big(\|\sigma|_{X\times[0,C_2+1]}\|_{C^0} + \|J_\omega\sigma|_{X\times[0,C_2+1]}\|_{C^{0,\alpha}}\big)\ ,
\end{equation*}
and the result follows upon combining these estimates.\myqed

\noindent We now prove Theorem \ref{proc:InvertibilityOverEndHoelder}.

\begin{proof}[Proof of Theorem \ref{proc:InvertibilityOverEndHoelder}.]
Indeed, choose $\omega\in(-p,0]$. By the Arzelà-Ascoli theorem, the restriction operator
\begin{equation*}
R:C^{2,\alpha}_{\omega,0}(\N \hat{X}_{r_0})\rightarrow C^0(\N \hat{X}_{r_0}|_{X\times[0,C]})
\end{equation*}
is compact. The estimate \eqnref{HoelderEstimateC} is thus of elliptic type, so that, as in \cite[Chapter 21, Theorem 4]{Lax}, $J$ has finite-dimensional kernel and closed image. Furthermore, as in the proof of Theorem \ref{proc:InvertibilityOverEndSobolev}, for all such $\omega$, $J_\omega$ is injective, so that, by the closed graph theorem, it is a linear isomorphism onto its image.

It remains only to prove surjectivity. Since $J_\omega$ is a linear isomorphism onto its image, there exists $C>0$ such that, for all $\sigma\in C^{2,\alpha}_{0,0}(\N \hat{X}_{r_0})$,
\begin{equation}\label{ProofOfInvertibilityOverHoelder}
\|\sigma\|_{C^{2,\alpha}_{0}} \leq C\|J_\omega\sigma\|_{C^{0,\alpha}_{0}}\ .
\end{equation}
Now let $\chi:\mathbf{R}\rightarrow[0,1]$ be a smooth function equal to $1$ over $(-\infty,0]$ and equal to $0$ over $[1,\infty)$, and for all $r>0$, define $\chi_r(x):=\chi(x-r)$. Choose $\tau\in C^{0,\alpha}_{0,0}(\N \hat{X}_{r_0})$, and note that, for all $r$, since $\chi_r$ has compact support, $\tau\chi_r\in L^2(\N \hat{X}_{r_0})$. It follows by Theorem \procref{InvertibilityOverEndSobolev} that, for all $r$, there exists $\sigma_r\in W^{2,2}_{0,0}(\N \hat{X}_{r_0})$ such that $J_\omega\sigma_r=\tau\chi_r$. Since $\hat{X}_{r_0}$ is of bounded geometry, it follows by the Sobolev embedding theorem that $\sigma_r\in L^\infty(\hat{X}_{r_0})$, then by elliptic regularity that $\sigma_r\in C^{2,\alpha}_\oploc(\hat{X}_{r_0})$, and then by the Schauder estimates (see, for example, \cite[Chapter 6]{GilbTrud}) that $\sigma_r\in C^{2,\alpha}(\hat{X}_{r_0})$. Note, however, that this process does not yield {\sl uniform} $C^{2,\alpha}$ bounds on $(\sigma_r)_{r>0}$, since the Sobolev norm of $(\tau\chi_r)_{r>0}$ may diverge as $r$ tends to infinity. However, by \eqref{ProofOfInvertibilityOverHoelder}, for all $r$,
\begin{equation*}
\|\sigma_r\|_{C^{2,\alpha}_0} \leq C\|\tau\chi_r\|_{C^{0,\alpha}_0}\ .
\end{equation*}
Note that $(\tau\chi_r)_{r>0}$ is uniformly bounded in $C^{0,\alpha}_0(\N \hat{X}_{r_0})$, so that, by the Arzel\`a-Ascoli theorem again, $(\sigma_r)_{r\geq 0}$ is relatively compact in the $C^{2,\beta}_\oploc$ topology for all $\beta<\alpha$. By semicontinuity of the H\"older seminorm, every accumulation point $\sigma$ is an element of $C^{2,\alpha}_{0,0}(\N \hat{X}_{r_0})$. Since every such point trivially satisfies $J_\omega\sigma=\tau$, surjectivity follows, and this completes the proof.
\end{proof}

\newsubhead{Elliptic estimates IV - maximal graphs}[MaximalGraphs]

We now use a perturbation argument to adapt Theorems \procref{InvertibilityOverEndSobolev} and \procref{InvertibilityOverEndHoelder} to weighted function spaces defined over maximal graphs. Let $M$ be a maximal graph in $\H_+$ with smooth spacelike boundary $\partial_\infty M=X_\infty$. By Lemma \procref{AsymptoticsOfMaximalGraph}, there exists $r_0>0$ and a compact subset $K$ of $M$ such that $M\setminus K$ coincides with the graph of some smooth normal section $\sigma$ over $\hat{X}_{r_0}$. In particular, this yields an explicit parametrization $\hat{e}^{\sigma}:\hat{X}_{r_0}\rightarrow M\setminus K$. For all $(x,r)\in\hat{X}_{r_0}$, let $\hat{q}^\sigma(x,r)$ denote the composition of parallel transport from $(x,r)$ to $\hat{e}^\sigma(x,r)$ with orthogonal projection onto the normal bundle of $M$. Note that, since parallel transport defines isometries, and since orthogonal projection from a $q$-dimensional negative-definite subspace in $\Rpq$ to any other such subspace is always a linear isomorphism, $\hat{q}^\sigma$ always defines a bundle isomorphism from $\N\hat{X}_{r_0}$ to $\N (M\setminus K)$. For any section $\tau$ of $\N M$, we denote
\begin{equation}
(\hat{e}^\sigma)^*\tau := (\hat{q}^\sigma)^{-1}\circ\tau\circ\hat{e}^\sigma \in \Gamma(\N \hat X_{r_0})~.
\end{equation}
For all suitable $k$, $p$ and $\omega$, and for every smooth section $\tau$ of $\N M$, we define
\begin{equation}
\|\tau\|_{W^{k,l}_\omega(M\setminus K)} := \|(\hat{e}^\sigma)^*\tau\|_{W^{k,l}_{\omega}(\hat{X}_{r_0})}\ .
\end{equation}
Let $\chi:M\rightarrow[0,1]$ be a smooth, compactly supported function, equal to $1$ over $K$. We define the \emph{$\omega$-weighted Sobolev norm} for sections of $\N M$ by
\begin{equation}
\|\tau\|_{W^{k,l}_\omega} := \|\chi\tau\|_{W^{k,l}} + \|(1-\chi)\tau\|_{W^{k,l}_\omega(M\setminus K)}\ ,
\end{equation}
and we define the \emph{$\omega$-weighted Sobolev space} $W^{k,l}_\omega(M)$ to be the completion of $C^\infty(\N M)$ with respect to this norm.

\begin{lemma}
Let $M$ be a complete maximal $p$-submanifold such that $\partial_\infty M=X_\infty$ is smooth and spacelike. For every weight $\omega\in(-\sqrt{p},\sqrt{p})$, there exists a compact subset $K'\subseteq M$ and $C>0$ such that, for all $\tau\in W^{2,2}_\omega(\N M)$,
\begin{equation}
\|\tau\|_{W^{2,2}_\omega} \leq C\big(\|\tau|_{K'}\|_{L^2} + \|J\tau\|_{L^2_\omega}\big)\ .\myeqnum{\nexteqnno[SobEstJacOp]}
\end{equation}
\proclabel{SobEstJacOp}
\end{lemma}

\myproof Indeed, let $K$ be as above. We may suppose that $K$ has smooth boundary and, by Lemma \procref{AsymptoticsOfMaximalGraph}, that $\sigma$ and all its derivatives are as small as we wish. It follows that, up to reparametrization, the restriction of $J$ to $M\setminus K$ is a perturbation of the Jacobi operator of $\hat{X}_{r_0}$ so that, by Theorem \procref{InvertibilityOverEndSobolev}, it defines a linear isomorphism from $W^{2,2}_{\omega,0}(\N M|_{M\setminus K})$ into $L^2_\omega(\N M|_{M\setminus K})$. Let $\chi:M\rightarrow[0,1]$ be a smooth, compactly supported function equal to $1$ over $K$. For suitable $C_1,C_2>0$, and for all $\tau\in W^{2,2}_\omega(\N M)$,
\begin{equation*}\eqalign{
\|\tau\|_{W^{2,2}_\omega}
&\leq \|\chi\tau\|_{W^{2,2}} + \|(1-\chi)\tau\|_{W^{2,2}_\omega}\cr
&\leq \|\chi\tau\|_{W^{2,2}} + C_1\|J(1-\chi)\tau\|_{L^2_\omega}\cr
&\leq C_2\big(\|\tau|_{\opSupp(\chi)}\|_{W^{2,2}} + \|J\tau\|_{L^2_\omega}\big)\ .\cr}
\end{equation*}
Finally, by the classical theory of elliptic operators, there exists $C_3>0$ and a compact set $K'$ containing a neighbourhood of $\opSupp(\chi)$ such that
\begin{equation*}
\|\tau|_{\opSupp(\chi)}\|_{W^{2,2}} \leq C_3\big(\|\tau|_{K'}\|_{L^2} + \|J\tau|_{K'}\|_{L^2}\big)\ ,
\end{equation*}
and the result follows upon combining these estimates.\myqed

\noindent With $\hat{e}^\sigma$ as before, for all suitable $k$, $\alpha$ and $\omega$, and for every $k$-times differentiable section $\tau$ of $\N M$, we define
\begin{equation}
\|\tau\|_{C^{k,\alpha}_\omega(M\setminus K)} := \|(\hat{e}^\sigma)^*\tau\|_{C^{k,\alpha}_{\omega}(\hat{X}_{r_0})}\ .
\end{equation}
With $\chi$ as before, we define the \emph{$\omega$-weighted H\"older norm} for $k$-times differentiable sections of $\N M$ by
\begin{equation}
\|\tau\|_{C^{k,\alpha}_\omega} := \|\chi\tau\|_{C^{k,\alpha}} + \|(1-\chi)\tau\|_{C^{k,\alpha}_\omega(M\setminus K)}\ ,
\end{equation}
and we define the \emph{$\omega$-weighted H\"older space} $C^{k,\alpha}_\omega(M)$ to be the space of all $k$-times differentiable sections of $\N M$ for which this norm is finite. Repeating the proof of Lemma \procref{SobEstJacOp} in the H\"older space case yields the following estimate.

\begin{lemma}
\noindent Let $M$ be a complete maximal $p$-submanifold such that $\partial_\infty M=X_\infty$ is smooth and spacelike. For every weight $\omega\in(-p,1)$, and for all $\alpha\in(0,1)$, there exists a compact subset $K'\subseteq M$ and $C>0$ such that, for all $\sigma\in C^{2,\alpha}_\omega(\N M)$,
\begin{equation}
\|\sigma\|_{C^{2,\alpha}_\omega} \leq C\big(\|\sigma|_{K'}\|_{C^{0,\alpha}} + \|J\sigma\|_{C^{0,\alpha}_\omega}\big)\ .\myeqnum{\nexteqnno[SobEstJacOp]}
\end{equation}
\proclabel{HolEstJacOp}
\end{lemma}

\noindent Repeating the proofs of Theorems \procref{InvertibilityOverEndSobolev} and \procref{InvertibilityOverEndHoelder} in the present context yields the following results, which concludes our study of the analytic properties of the action of the Jacobi operator on Sobolev and H\"older spaces over cones and maximal graphs.

\begin{theorem}
Let $M$ be a complete maximal $p$-submanifold such that $\partial_\infty M=X_\infty$ is smooth and spacelike. For every weight $\omega\in(-\sqrt{p},\sqrt{p})$, the operator $J$ defines a linear isomorphism from $W^{2,2}_\omega(M)$ into $L^2_\omega(M)$.
\proclabel{InvSob}
\end{theorem}

\begin{theorem}
Let $M$ be a complete maximal $p$-submanifold such that $\partial_\infty M=X_\infty$ is smooth and spacelike. For every weight $\omega\in(-p,0]$, and for all $\alpha\in(0,1)$, the operator $J$ defines a linear isomorphism from $C^{2,\alpha}_\omega(M)$ into $C^{0,\alpha}_\omega(M)$.
\proclabel{InvHol}
\end{theorem}

\subsection{Perturbing maximal ends}\label{PerturbationsOfMinimalEnds}

We now describe the families of perturbations of maximal graphs to which the implicit function theorem will be applied. Our construction is inevitably rather technical, reflecting the main difference between the perturbation theories of compact manifolds and non-compact manifolds, namely that in the non-compact case functional norms are sensitive to the parametrizations used. It is for this reason that we will describe our parametrizations in some detail.

The first step concerns the construction of explicit parametrizations of families of cones. We proceed as follows. Let $(X_t)_{t\in(-\epsilon,\epsilon)}$ be a smooth family of spacelike graphs in $\T^1_{x_0}\H\cong\Spmq$ such that $X_0=X$. Let $(e_t)_{t\in(-\epsilon,\epsilon)}$ be a smooth family of smooth functions from $X$ into $\Spmq$ such that, for all $t$, $e_t$ parametrises $X_t$. For all $t$, define $\hat{e}_t:X\times(0,\infty)\rightarrow\H_+$ by
\begin{equation}
\hat{e}_t(x,r) := \Phi(e_t(x),r)\ ,\myeqnum{\nexteqnno[ParamOfConeFamily]}
\end{equation}
so that $\hat{e}_t$ parametrises $\hat{X}_t$.

We now study the variations of the mean curvature vectors of these cones. Note that the mean curvature vector always lies in the tangent bundle of the ambient space. However, we alert the reader to the fact that, since the derivative of $e_t$ with respect to $t$ becomes large at infinity, it will be preferable to identify different fibres of this bundle in the following non-standard manner.

First, for all $t\in (-\epsilon,\epsilon)$, and for all $x\in X$, let $p_t(x)$ denote parallel transport from $x$ to $e_t(x)$ along the unique geodesic joining these two points, let $q_t(x)$ denote the composition of this function with the orthogonal projection onto the normal bundle of $e_t$ and note that, as before, $q_t$ defines a bundle isomorphism from $\N X$ into $\N X_t$ for all $t$. For all $t$, let $H_t\in \Gamma(\N X)$ denote the image under $q_t^{-1}$ of the mean curvature vector of $X_t$.

For all $t$, we identify $\N \hat{X}_t$ with $\pi_1^*\N X_t$ via parallel transport $\tau_\oprad$ along radial lines, where $\pi_1$ is the projection onto the first factor on  $\hat{X}_t\cong X_t \times (0,+\infty)$.  We denote by $\hat{q}_t:\N \hat{X}_0\rightarrow \N \hat{X}_t$ the bundle isomorphism induced by this identification, that is, the unique bundle isomorphism such that the following diagram commutes.
\begin{equation}\diagram{
\N \hat{X}_0&\sharr{\tau_\oprad}{}&\pi_1^*\N X_0\cr
\varr{\hat{q}_t}{}& &\varr{\pi_1^*q_t}{}\cr
\N \hat{X}_t&\sharr{\tau_\oprad}{}&\pi_1^*\N X_t\cr}~.\myeqnum{\nexteqnno[CanonicalHomeomorphism]}
\end{equation}
For all $t$, let $\hat{H}_t$ denote the image under $\hat{q}_t^{-1}$ of the mean curvature vector of $\hat{X}_t$. Note now that, whilst radial projection contracts by a factor of $\opSinh(r)$, parallel transport does not. The identification of $\N \hat{X}_t$ with $\pi_1^*\N X_t$ via parallel transport along radial lines thus differs from that used in Section \ref{sec:cones} by a factor of $\opSinh(r)$ so that, by \eqref{SecondFFOfCone}, for all $t$,
\begin{equation}
\hat{H}_t = \frac{\pi_1^*H_t}{\opSinh(r)}\ .\myeqnum{\nexteqnno[MeanCurvOfConeFamily]}
\end{equation}

\begin{lemma}
For all $r_0>0$ and for all $\alpha\in(0,1)$, the map sending $t$ to $\hat{H}_t$ defines a smooth function from $(-\epsilon,\epsilon)$ into $C_{-1}^{2,\alpha}(\N \hat{X}_{r_0})$.
\proclabel{ConeMCIsSmooth}
\end{lemma}

\myproof Indeed, $(H_t)_{t\in(-\epsilon,\epsilon)}$ trivially defines a smooth function from $(-\epsilon,\epsilon)$ into $C^{2,\alpha}(\N X_{r_0})$. It follows that $(\pi_1^*H_t)_{t\in(-\epsilon,\epsilon)}$ defines a smooth function from $(-\epsilon,\epsilon)$ into $C_0^{2,\alpha}(\N \hat{X}_{r_0})$. Since multiplication by $\frac{1}{\opSinh(r)}$ defines a bounded linear map from $C_0^{2,\alpha}(\N \hat{X}_{r_0})$ into $C^{2,\alpha}_{-1}(\N \hat{X}_{r_0})$, the result follows.\myqed

The second step involves constructing explicit parametrizations of families of graphs over families of cones. For all $t$, and for all $\sigma\in C_0^{2,\alpha}(\N \hat{X}_{r_0})$, we define $\hat{e}_t^\sigma:\hat{X}\rightarrow\H$ by
\begin{equation}
\hat{e}_t^\sigma(x,r) := \Exp_{\hat{e}_t(x,r)}(\hat{q}_t\circ\sigma(x,r))\ .\myeqnum{\nexteqnno[FirstPertFamily]}
\end{equation}
\noindent Let $\Cal{U}:=\Cal{U}_0^{2,\alpha}$ be a neighbourhood of the zero section in $C_0^{2,\alpha}(\N \hat{X}_{r_0})$ such that, for all $\sigma\in\Cal{U}$, $\hat{e}_t^\sigma$ is a spacelike embedding. We now study, as before, the variations of the mean curvature vectors of these embeddings. Note that, in contrast to the previous case, the derivative of $\hat{e}_t^\sigma$ with respect to $\sigma$ remains small at infinity, and we thus use parallel transport to identify nearby fibres of the tangent bundle of the ambient space in the standard manner. For all $t$, and for all $\sigma\in\mathcal{U}$, let $\hat{q}_t^\sigma$ denote the composition of parallel transport from $\hat{e}_t(x,r)$ to $\hat{e}_t^\sigma(x,r)$ with orthogonal projection onto the normal bundle of $\hat{e}_t^\sigma$. As before, for all $(t,\sigma)\in(-\epsilon,\epsilon)\times\Cal{U}$, $\hat{q}^\sigma_t$ is a bundle isomorphism, we denote by $\hat{H}_t^\sigma$ the image under $(\hat{q}_t^\sigma\circ\hat{q}_t)^{-1}$ of the mean curvature vector of $\hat{e}_t^\sigma$, and we define
\begin{equation}
\Cal{F}(t,\sigma) := \hat{H}_t^\sigma - \hat H_t\ .\myeqnum{\nexteqnno[CurvErrorFirstPert]}
\end{equation}
We recall the following technical lemma of differential geometry.

\begin{lemma}\label{lemma:AlmostTaylor}
Let $\Omega\subseteq\R^m$ be an open subset, and let $f:\Omega\rightarrow\mathbf{R}$ be a smooth function. For all $x_0\in\Omega$, there exist smooth functions $f_1,\cdots,f_m:\Omega\rightarrow\mathbf{R}$ such that
\begin{equation}
f(x) = f(x_0) + \sum_{i=1}^m(x^i-x_0^i)f_i(x)\ .
\end{equation}
Furthermore, for all $i$, $f_i(0)=(\partial_if)(0)$.
\end{lemma}

\begin{proof} It suffices to address the case where $\Omega$ is convex. For all $i$, we define
\begin{equation*}
f_i(x) := \int_0^1 \frac{\partial f}{\partial x_i}(x_0 + t(x-x_0))dt\ ,
\end{equation*}
and we readily verify that these functions have the desired properties.
\end{proof}

\noindent For all $\omega\leq 0$, we denote
\begin{equation*}
\mathcal{U}^{2,\alpha}_\omega := \mathcal{U}\minter C^{2,\alpha}_\omega(\N \hat{X}_{r_0})\ .
\end{equation*}

\begin{lemma}
For all $r_0>0$, for all $\alpha\in(0,1)$, and for every weight $\omega\leq 0$, $\Cal{F}$ defines a smooth function from $(-\epsilon,\epsilon)\times\Cal{U}^{2,\alpha}_\omega$ into $C^{0,\alpha}_\omega(\N \hat{X}_{r_0})$.
\proclabel{FirstPertIsSmooth}
\end{lemma}

\myproof Since $\Cal{F}(t,0)=0$, and since mean curvature is quasilinear, upon reducing $\Cal{U}$ if necessary, and upon applying Lemma \ref{lemma:AlmostTaylor} pointwise to the $2$-jets of $\sigma$, we find that there exist smooth functions $a^{ij}$, $b^i$ and $c$ such that, for all $(t,\sigma)$,
\begin{equation*}\eqalign{
\Cal{F}(t,\sigma)(x,r) &= a^{ij}(t,x,r,J^1\sigma(x,r))\opHess(\sigma)_{ij}(x,r) + b^1(t,x,r,J^1\sigma(x,r))(\nabla\sigma)_i(x,r)\cr
&\qquad\qquad\qquad\qquad+c(t,x,r,J^1\sigma(x,r))\sigma(x,r)\ ,\cr}\myeqnum{\nexteqnno[PseudoLinearizationOfF]}
\end{equation*}
where here $J^1\sigma$ denotes the $1$-jet of $\sigma$. Consider now the final term in this expression. The canonical inclusion $\mathcal{U}^{2,\alpha}_\omega\rightarrow\mathcal{U}^{2,\alpha}_0$ is continuous and linear, and therefore smooth. Since the derivatives of $c$ to all orders are bounded, the non-linear operator
\begin{equation*}
\sigma\mapsto c(t,x,r,J^1\sigma(x,r))
\end{equation*}
defines a smooth function from $(-\epsilon,\epsilon)\times\Cal{U}^{2,\alpha}_0$ into $C^{0,\alpha}_0(\hat{X}_{r_0})$. Finally, the product $C^{0,\alpha}_0(\hat{X}_{r_0})\oplus C^{0,\alpha}_\omega(\N \hat{X}_{r_0})\rightarrow C^{0,\alpha}_\omega(\N \hat{X}_{r_0})$ is continuous and bilinear, and thus also smooth. It follows that the composition
\begin{equation*}
\sigma\mapsto c(t,x,r,J^1\sigma(x,r))\sigma(x,r)
\end{equation*}
defines a smooth map from $(-\epsilon,\epsilon)\times\Cal{U}^{2,\alpha}_\omega$ into $C^{0,\alpha}_\omega(\N \hat{X})$. The remaining terms, though more technical, are addressed in a similar manner, and this completes the proof.\myqed

The third step of our construction involves iterating the preceding step to produce explicit parametrizations of families of graphs over families of graphs over families of cones. Thus, for all $t$, for all $\sigma\in\Cal{U}$, and for all $\tau\in C_0^{2,\alpha}(\N \hat{X})$, define $\hat{e}_t^{(\sigma,\tau)}:\hat{X}\rightarrow\H_+$ by
\begin{equation}
\hat{e}_t^{(\sigma,\tau)}(x,r) := \Exp_{\hat{e}_t^\sigma(x,r)}(\hat{q}_t^\sigma\circ\hat{q}_t\circ\tau(x,r))\ .\myeqnum{\nexteqnno[SecondPertFamily]}
\end{equation}
Upon reducing $\Cal{U}$ if necessary, we may suppose that, for all $t$, and for all $\sigma,\tau\in\Cal{U}$, $\hat{e}_t^{(\sigma,\tau)}$ is a spacelike embedding. For all such $(t,\sigma,\tau)$, let $\hat{q}_t^{(\sigma,\tau)}$ denote the composition of parallel transport from $\hat{e}_t^\sigma(x,r)$ to $\hat{e}_t^{(\sigma,\tau)}(x,r)$ with orthogonal projection onto the normal bundle of $\hat{e}_t^{(\sigma,\tau)}$. Once again, for all $(t,\sigma,\tau)\in(-\epsilon,\epsilon)\times\Cal{U}\times\Cal{U}$, $\hat{q}^{(\sigma,\tau)}_t$ is a bundle isomorphism, we denote by $\hat{H}_t^{(\sigma,\tau)}$ the image under $(\hat{q}_t^{(\sigma,\tau)}\circ\hat{q}_t^{\sigma}\circ\hat{q}_t)^{-1}$ of the mean curvature vector of $\hat{e}_t^{(\sigma,\tau)}$, and we define
\begin{equation}
\Cal{G}(t,\sigma,\tau) := \hat{H}_t^{(\sigma,\tau)} - \hat{H}_t\ .\label{CurvErrorSecondPert}
\end{equation}
Repeating the proof of Lemma \procref{FirstPertIsSmooth} yields the following result.

\begin{lemma}
For all $r_0>0$, for all $\alpha\in(0,1)$, and for every weight $\omega\leq 0$, $\Cal{G}$ defines a smooth function from $(-\epsilon,\epsilon)\times\Cal{U}^{2,\alpha}_\omega\times\Cal{U}^{2,\alpha}_\omega$ into $C^{0,\alpha}_\omega(\N \hat{X})$.
\proclabel{SecondPertIsSmooth}
\end{lemma}

Finally, we sketch how this construction extends to treat perturbations of complete maximal graphs. At this stage, although it is not strictly necessary, we will take the weight $\omega$ to be equal to zero. Let $M:=M_0$ be a complete maximal graph such that $\partial_\infty M=X_{0,\infty}$. By Lemma \ref{proc:AsymptoticsOfMaximalGraph}, there exists a compact subset $K\subseteq M$, $r_0>0$, and a section $\sigma_0\in C^{2,\alpha}_0(\N \hat{X}_{0,r_0})$ whose graph coincides with $M\setminus K$. In addition, upon increasing $r_0$ if necessary, we may suppose that $\sigma_0$ is as small as we wish. Now let $(\tilde{M}_t)_{t\in(-\epsilon,\epsilon)}$ be a smooth family of complete, smooth, not necessarily maximal, spacelike graphs such that $\tilde{M}_0=M$ and such that, for all $t$, the graph of $\sigma_0$ over $\N \hat{X}_{t,r_0}$ coincides with the complement of some compact subset of $\tilde{M}_t$. Let $(\hat{e}^M_t)_{t\in(-\epsilon,\epsilon)}$ be a smooth family of smooth embeddings such that
\begin{enumerate}
\item $\hat{e}^M_0$ is the identity,
\item for all $t$, $\hat{e}^M_t$ parametrizes $\tilde{M}_t$, and
\item for all $t$, $\hat{e}^M_t$ coincides with $\hat{e}^{\sigma_0}_t\circ(\hat{e}^{\sigma_0}_0)^{-1}$ over $\tilde{M}_0\setminus K$.
\end{enumerate}
We underline here that it is of little importance how $(\hat{e}^M_t)_{t\in(-\epsilon,\epsilon)}$ extends over the rest of $(\tilde{M}_t)_{t\in(-\epsilon,\epsilon)}$, as long as it is smooth. For all sufficiently small $(t,\tau)\in(-\epsilon,\epsilon)\times C^{2,\alpha}(\N M)$, we define $\hat{q}^{M,\tau}_t$ and $\hat{H}^{M,\tau}_t$ in the same way as before, and we define, for all such $(t,\tau)$,
\begin{equation}\label{eq:defi operator H}
\mathcal{H}(t,\tau) := \hat{H}_t^{M,\tau}\ .
\end{equation}
Note that, by construction,
\begin{equation*}
\mathcal{H}(0,0) = 0\ .
\end{equation*}

\begin{lemma}\label{lemma:H smooth}
For all $\alpha\in(0,1)$, upon reducing $\mathcal{U}$ if necessary, $\Cal{H}$ defines a smooth function from $(-\epsilon,\epsilon)\times\Cal{U}^{2,\alpha}_0$ into $C^{0,\alpha}_0(\N M)$.
\proclabel{ThirdPertIsSmooth}
\end{lemma}

\begin{remark}
We will see in Section \ref{sec:conclude2} that $\sigma_0\in C^{2,\alpha}_{-1}(\N \hat{X}_{0,r_0})$, and the argument of the following proof then shows that $\Cal{H}$ in fact defines a smooth function from $(-\epsilon,\epsilon)\times\Cal{U}^{2,\alpha}_\omega$ into $C^{0,\alpha}_\omega(\N M)$ for all $\omega\in[-1,0]$. However, Lemma \ref{proc:ThirdPertIsSmooth} as stated is already sufficient for proving stability.
\end{remark}

\begin{proof}
Indeed, outside some compact subset of $M$,
\begin{equation*}
(\hat{e}^{\sigma_0}_0)^*\mathcal{H}(t,\tau) := (\hat{q}^{\sigma_0}_0)^{-1}\circ\mathcal{H}(t,\tau)\circ\hat{e}^{\sigma_0}_0 = \Cal{G}(t,\sigma_0,\tau) + \hat{H}_t\ ,
\end{equation*}
and the result now follows by Lemmas \ref{proc:ConeMCIsSmooth} and \ref{proc:SecondPertIsSmooth}
\end{proof}

\noindent We are now ready to prove our perturbation result.

\begin{proof}[Proof of Theorem \ref{thm:perturb H+}] Let $(\Lambda_t)_{t\in(-\epsilon,\epsilon)}$ be a family of  spacelike spheres in $\bH_+$ varying continuously in the $C^\infty$ topology. We first show that, upon reducing $\epsilon$ if necessary, we may suppose that, for all $t$, $\Lambda_t$ lies in the image of some fixed spacelike polar coordinate chart. Indeed, in any Fermi chart, $\Lambda$ is the graph of some strictly $1$-Lipschitz function $\varphi:\S^{p-1}\rightarrow\S^q$. By \cite[Lemma 2.8]{heinonen}, this function takes values in an open hemisphere so that, by Lemma \ref{eqn:GeometriesOfPhiAndPhiInfinity} and the subsequent remark, $\Lambda_0$ lies in the image of some spacelike polar coordinate chart. Upon reducing $\epsilon$ if necessary, we may suppose that $\Lambda_t$ is also contained in the same coordinate chart, as asserted.

Now, let $M_0$ be a complete maximal graph with $\partial_\infty M_0=\Lambda_0$. Let $(\tilde{M}_t)_{t\in(-\epsilon,\epsilon)}$ denote the smooth family of perturbations of $M_0$ constructed above, and let $\mathcal{H}:(-\epsilon,\epsilon)\times\Cal{U}^{2,\alpha}_0\to C^{0,\alpha}_0(\N M)$ denote the operator defined in \eqref{eq:defi operator H}. As shown above, $\mathcal{H}$ defines a smooth function between Banach spaces. Furthermore, its partial derivative with respect to the second component is the Jacobi operator $J$. By Theorem \procref{InvHol}, $J:C^{2,\alpha}_0(\N M)\rightarrow C^{0,\alpha}_0(\N M)$ is a linear isomorphism. It follows by the implicit function theorem that, upon reducing $\epsilon$ and $\Cal{U}$ if necessary, we may suppose that for all $t\in(-\epsilon,\epsilon)$, there exists a unique element $\tau_t\in\Cal{U}^{2,\alpha}_0$ such that
\begin{equation*}
\hat{H}_t^{M,\tau_t} = \Cal{H}(t,\tau_t) = 0\ .
\end{equation*}
For all such $t$, the image of $\hat{e}_t^{M,\tau_t}$ is a maximal graph $M_t$, say. Since $(\tau_t)_{t\in(-\epsilon,\epsilon)}$ tends to zero in $C^{2,\alpha}_0(\N M)$ as $t$ tends to zero, upon reducing $\epsilon$ if necessary, we may suppose that, for all $t$, $\hat{e}_t^{M,\tau_t}$ has bounded second fundamental form so that, by Lemma \ref{lemma:EntireGraph}, it is complete. Finally, since, for all $t$, $\tau_t$ is bounded, $\partial_\infty M_t=\partial_\infty\tilde{M}_t=\Lambda_t$, and this completes the proof.
\end{proof}

\section{Proofs of main results}\label{sec:conclude}

\subsection{Smoothly approximating 1-Lipschitz maps}

In order to conclude the proofs of Theorem \ref{thm:introhomeo}, we require the following refinement of \cite[Lemma 2.8]{heinonen}.

\begin{lemma}\label{lemma:contained hemisphere}
Let $\varphi:\S^{p-1}\to\S^q$ be a 1-Lipschitz map. If the image of $\varphi$ does not contain antipodal points, then it is contained in an open hemisphere of $\S^q$.
\end{lemma}
\begin{proof}
By the Hahn-Banach separation theorem, it suffices to show that $0$ is not in the convex hull of $\varphi(\S^{p-1})$ in $\textbf R^{q+1}$. Suppose the contrary, namely that there exist $x_1,\ldots,x_m\in\S^{p-1}$ and $\lambda_1,\ldots,\lambda_m\in (0,1)$ with $\sum\lambda_i=1$ such that
\begin{equation*}
\sum_{i=1}^m \lambda_i\varphi(x_i)=0~.
\end{equation*}
Since $\opCos$ is decreasing on $[0,\pi]$, for all $i,j$,
\begin{equation*}
\langle\varphi(x_i),\varphi(x_j)\rangle = \opCos(d_S(\varphi(x_i),\varphi(x_j))) \geq \opCos(d_S(x_i,x_j)) = \langle x_i,x_j\rangle\ ,
\end{equation*}
so that,
\begin{equation*}\eqalign{
0&=\bigg\langle\sum_{i=1}^m\lambda_i\varphi(x_i),\sum_{i=1}^m\lambda_i\varphi(x_i)\bigg\rangle\cr
&=\sum_{i=1}^m\lambda_i^2+2\sum_{1\leq i<j\leq m}\lambda_i\lambda_j\langle \varphi(x_i),\varphi(x_j)\rangle\cr
&\geq\sum_{i=1}^m\lambda_i^2+2\sum_{1\leq i<j\leq m}\lambda_i\lambda_j\langle x_i,x_j\rangle\cr
&=\bigg\langle\sum_{i=1}^m\lambda_i x_i,\sum_{i=1}^m\lambda_i x_i\bigg\rangle
\geq 0~.\vphantom{\frac{1}{2}}\cr}
\end{equation*}
Equality therefore holds, so that $0\in\R^p$ also lies in the convex hull of $\{x_1,\cdots,x_m\}$. Furthermore, for any other point $y$ in $\S^{p-1}$,
\begin{equation*}
0=\bigg\langle\sum_{i=1}^m\lambda_i\varphi(x_i),\varphi(y)\bigg\rangle=\sum_{i=1}^m\lambda_i\langle\varphi(x_i),\varphi(y)\rangle\geq\sum_{i=1}^m \lambda_i \langle x_i,y \rangle=\sum_{i=1}^m\langle\lambda_i x_i,y\rangle=0\ ,
\end{equation*}
so that
\begin{equation*}
\sum_{i=1}^m\lambda_i(\langle\varphi(x_i),\varphi(y)\rangle-\langle x_i,y \rangle)=0\ .
\end{equation*}
Since every summand on the left-hand side is non-negative, it follows that, for all $i$,
\begin{equation*}
\langle\varphi(x_i),\varphi(y)\rangle=\langle x_i,y\rangle\ .
\end{equation*}
Setting $y=-x_1$ now yields
\begin{equation*}
\langle \varphi(x_1),\varphi(-x_1)\rangle=\langle x_1,-x_1 \rangle=-1\ ,
\end{equation*}
so that $\varphi(\S^{p-1})$ contains antipodal points. This is absurd, and the result follows.
\end{proof}

The first consequence of Lemma \ref{lemma:contained hemisphere} is that 1-Lipschitz maps not containing antipodal points can be continuously deformed to constant maps.

\begin{lemma}\label{cor:deformation to constant}
Let $\varphi:\S^{p-1}\to\S^q$ be a 1-Lipschitz map. If the image of $\varphi$ does not contain antipodal points, then there exists a continuous family of maps $(\varphi_t)_{t\in[0,1]}:\S^{p-1}\to\S^q$ such that
\begin{enumerate}
\item $\varphi_0=\varphi$,
\item $\varphi_1$ is constant, and
\item for all $t$, $\varphi_t$ is $(1-t)$-Lipschitz.
\end{enumerate}
Morerover, if, in addition, $\varphi$ is smooth, then we can assume that $(\varphi_t)_{t\in[0,1]}$ is also smooth.
\end{lemma}

\begin{proof}
By applying Lemma \ref{lemma:contained hemisphere}, let $y$ denote the centre of an open hemisphere $\Upp^q$ containing $\varphi(\S^{p-1})$. We work in geodesic coordinates of $\S^q$ about $y$, so that $\Upp^q$ identifies with the ball $B^q_{\pi/2}(0)$. By compactness,
\begin{equation*}
\varphi(\S^{p-1})\subset B_r^q(0)\ ,
\end{equation*}
for some $r<\pi/2$. For all $s$, define $\Phi_s:B_r^q(0)\rightarrow B_r^q(0)$ by
\begin{equation*}
\Phi_s(x) := sx\ .
\end{equation*}
A quick computation shows that, for all $s$, the map $\Phi_s$ is $\left(\frac{\opSin(rs)}{\opSin(r)}\right)$-Lipschitz with respect to the spherical metric over $B^q_r(0)$. We thus define
\begin{equation*}
\varphi_t := \Phi_{\frac{\opArcSin((1-t)\opSin(r))}{r}}\circ\varphi\ ,
\end{equation*}
and the result follows.
\end{proof}

The second consequence of Lemma \ref{lemma:contained hemisphere} is that $1$-Lipschitz maps not containing antipodal points can be approximated by smooth, strictly $1$-Lipschitz maps.

\begin{lemma}\label{lemma:ApproximationBySmooth}
Let $\varphi:\S^{p-1}\rightarrow\S^q$ be a $1$-Lipschitz map. If the image of $\varphi$ does not contain antipodal points, then, for all $\delta>0$, there exists a smooth, strictly $1$-Lipschitz map $\psi:\S^{p-1}\rightarrow\S^q$ such that
\begin{equation}
\|\varphi-\psi\|_{C^0} < \delta\ .
\end{equation}
\end{lemma}

\begin{proof}
By Lemma \ref{cor:deformation to constant}, there exists $\epsilon>0$, and a $(1-\epsilon)$-Lipschitz map $\varphi':\S^{p-1}\rightarrow\S^q$ such that $\|\varphi-\varphi'\|_{C^0}<\delta/2$. We now identify $\S^q$ with the unit sphere in $\R^{q+1}$, and we denote by $\pi:\R^{q+1}\setminus\{0\}\rightarrow\S^q$ radial projection. The function $\varphi''(x):=\varphi'(x)/(1-\epsilon)$ is a $1$-Lipschitz map into $\S^q/(1-\epsilon)$. Applying a suitable smoothing convolution yields a smooth, $1$-Lipschitz map $\varphi''':\S^{p-1}\rightarrow\R^{q+1}$ such that $\|\varphi''-\varphi'''\|_{C^0}<\delta/2$. Since $\pi$ is strictly contracting outside $\S^q$,
\begin{equation*}
\|\varphi'-(\pi\circ\varphi''')\|_{C^0} = \|(\pi\circ\varphi'')-(\pi\circ\varphi''')\|_{C^0} < \frac{\delta}{2}\ .
\end{equation*}
It follows that $\psi:=(\pi\circ\varphi''')$ is a smooth, strictly $1$-Lipschitz map, and satisfies
\begin{equation*}
\|\varphi-\psi\|_{C^0} \leq \|\varphi-\varphi'\| + \|\varphi'-\psi\| < \delta\ ,
\end{equation*}
as desired.
\end{proof}

\subsection{Proof of Theorem \ref{thm:introhomeo}}

\begin{proof}[Proof of Theorem \ref{thm:introhomeo}]
We first show that $\partial_\infty$ maps $\M$ bijectively onto $\Bd$. It suffices to work in the double cover $\H_+$ and show that $\partial_\infty^+$ maps $\M_+$ bijectively onto $\Bd_+$. By Lemmas \ref{lemma:EntireGraphsAreGraphs}, \ref{lemma:EntireGraph} and \ref{lemma:AdmissableBoundaryCriteria}, $\partial_\infty$ maps $\M_+$ into $\Bd_+$. We now show that every admissible non-negative $(p-1)$-sphere $\Lambda$ in $\bH_+$ is the asymptotic boundary of a unique complete maximal graph in $\H_+$.

We first prove existence. Suppose first that $\Lambda$ is smooth and spacelike. By Lemma \ref{lemma:contained hemisphere}, we may choose a Fermi chart with respect to which $\Lambda$ is the graph of some smooth, strictly $1$-Lipschitz map $\varphi:\S^{p-1}\to\Upp^q$, say. Let $(\varphi_t)_{t\in[0,1]}$ be as in Lemma \ref{cor:deformation to constant} and, for all $t$, let $\Lambda_t$ denote the graph of $\varphi_t$ in $\bH_+\cong \S^{p-1}\times\S^q$. Let $E\subseteq[0,1]$ denote the set of all $t$ for which $\Lambda_t$ is the asymptotic boundary of some complete maximal graph in $\bH_+$. By Theorem \ref{thm:perturb H+}, $E$ is open, by Theorem \ref{theorem:Properness}, $E$ is closed, and, since $\Lambda_0$ is the boundary of a totally-geodesic $p$-dimensional hyperbolic plane, $E$ is non-empty. It follows by connectedness that $E=[0,1]$, and this proves existence when $\Lambda$ is smooth.

By Corollary \ref{theorem:Properness}, $\partial_\infty^+:\M_+\rightarrow\Bd_+$ is proper, and, by Lemma \ref{lemma:ApproximationBySmooth}, the set of smooth, spacelike spheres in $\partial_\infty \H_+$ is dense in $\Bd_+$, and surjectivity of $\partial_\infty^+$ follows. Since uniqueness is proven in Theorem \ref{thm:injective}, it follows that $\partial_\infty^+$ is a bijection, as desired.

It remains only to show that $\partial_\infty$ is a homeomorphism. Let $\seq[m]{M_m}$ be a sequence of maximal graphs converging in the $C^\infty_\oploc$ sense to the maximal graph $M_\infty$. Choose a Fermi chart and, for all $m\in\mathbf{N}\munion\{\infty\}$, let $\varphi_m$ denote the $1$-Lipschitz function whose graph is $M_m$. By the Arzelà-Ascoli Theorem, $\seq[m]{\varphi_m}$ is compact in the uniform topology. Since $\varphi_\infty$ is its unique limit, this sequences converges uniformly towards $\varphi_\infty$, and it follows that $\seq[m]{M_m}$ also converges in the Hausdorff sense to $M_\infty$. From this is follows that $\seq[m]{\partial_\infty M_m}$ likewise converges in the Hausdorff sense to $\partial_\infty M_\infty$, and this proves continuity. Finally, by Theorem \ref{theorem:Properness}, the asymptotic boundary map is proper. Its inverse is therefore also continuous, and this completes the proof.
\end{proof}

\subsection{Proof of Theorem \ref{thm:DecayOfSecondFundamentalForm}}\label{sec:conclude2}

We first note that, in order not to trouble the reader with technical questions of low regularity, we have up to this point only worked with objects that are smooth. However, by elliptic regularity, in the context of elliptic, second-order partial differential equations, the properties of being smooth and of being $C^{2,\alpha}$ are equivalent. There is, nonetheless, one important caveat, namely, when we wish to study functionals over the space of $C^{2,\alpha}$ sections of the normal bundle of some cone $\hat{X}$, we require that this bundle also be of type $C^{2,\alpha}$. For this to hold, it is necessary and sufficient that the sphere $X$ be of type $C^{3,\alpha}$. For this reason, we henceforth work with asymptotic boundaries of type $C^{3,\alpha}$, to which the theory developed in the previous sections applies without modification.

It will now be sufficient to work in the double cover $\H_+$. Let $M$ be a compete maximal $p$-submanifold of $\H_+$ with $C^{3,\alpha}$ asymptotic boundary $M_\infty$. Let $x_0\in\H_+$ be such that $M_\infty$ lies in the image of its spacelike polar coordinate chart. Let $X\subseteq\T^1_{x_0}\H_+$ denote the preimage of $M_\infty$ in this chart, and let $\hat{X}$ denote its cone. Let $r_0>0$ and $\sigma\in\Gamma(N\hat{X}_{r_0})$ be as in Lemma \ref{proc:AsymptoticsOfMaximalGraph}. Note that $\|\sigma|_{B_1(x)}\|_{C^{2,\alpha}}$ tends to zero as $x$ tends to infinity so that, in particular,
\begin{equation}\label{eqn:SigmaIsBounded}
\sigma\in C^{2,\alpha}_0(N\hat{X}_{r_0})\ .
\end{equation}

Theorem \ref{thm:DecayOfSecondFundamentalForm} is a straightforward consequence of the following refinement of \eqref{eqn:SigmaIsBounded}.

\begin{lemma}\label{lemma:SigmaDecaysExponentially}
For all $\alpha\in(0,1)$,
\begin{equation}\label{eqn:SigmaDecaysExponentially}
\sigma\in C^{2,\alpha}_{-1}(N\hat{X}_{r_0})\ .
\end{equation}
\end{lemma}

\begin{proof} Let $\hat{H}$ denote the mean curvature vector of $\hat{X}$, and note that, by \eqref{eqn:DecayOfMeanCurv},
\begin{equation*}\label{eqn:EstimateOfH}
\hat{H}\in C^{0,\alpha}_{-1}(N\hat{X}_{r_0})\ .
\end{equation*}
Consider now a general section $\tau\in C^{2,\alpha}_0(N\hat{X}_{r_0})$. Since $\hat{X}$ is asymptotically totally-geodesic, for $\|\tau\|_{C^{2,\alpha}}$ sufficiently small, the graph of $\tau$ is an embedded, spacelike submanifold. For all such $\tau$, we denote
\begin{equation*}
\hat{e}^\tau(x,r):=\opExp_{(x,r)}(\tau(x,r))\ ,
\end{equation*}
we denote by $\hat{q}^\tau(x,r)$ the composition of parallel transport from $(x,r)$ to $\hat{e}^\tau(x,r)$ with orthogonal projection onto the normal bundle of $\hat{X}$, and we denote by $\hat{H}^\tau$ the image under $(\hat{q}^\tau)^{-1}$ of the mean curvature vector of $\hat{e}^\tau$. Since the mean curvature functional is quasi-linear, upon applying Lemma \ref{lemma:AlmostTaylor} pointwise to the $2$-jets of $\tau$, we obtain
\begin{equation*}
\hat{H}^\tau = \hat{H} + L(\tau)\tau\ ,
\end{equation*}
where, for any section $\rho$,
\begin{equation*}
L(\tau)\rho := a^{ij}(x,r,J^1\tau(x,r))\opHess(\rho)_{ij}(x,r) + b^i(x,r,J^1\tau(x,r))(\nabla\rho)_i(x,r) + c(x,r,J^1\tau(x,r))\rho(x)\ ,
\end{equation*}
for suitable functions $a$, $b$ and $c$. In particular, since the graph of $\sigma$ is maximal,
\begin{equation*}
L(\sigma)\sigma = -\hat{H}\ .
\end{equation*}
Note now that $L(0)$ is simply the Jacobi operator $J$ of $\hat{X}$. Since $\|\sigma|_{B_1(x)}\|_{C^{2,\alpha}}$ tends to zero as $x$ tends to infinity, it follows that, for all $\tau$ supported in the complement of a sufficiently large compact subset $K$ of $\hat{X}_{r_0}$,
\begin{equation*}
L(\sigma)\tau = J'\tau\ ,
\end{equation*}
for some perturbation $J'$ of $J$ for which the conclusion of Theorem \ref{proc:InvertibilityOverEndHoelder} continues to hold. Suppose now that that $K$ contains the boundary of $\hat{X}_{r_0}$, and let $\chi:\hat{X}_{r_0}\rightarrow[0,1]$ be a smooth, compactly supported function equal to $1$ over $K$. Then, $(1-\chi)\sigma$ vanishes along the boundary and, outside $\opSupp(\chi)$,
\begin{equation*}
J'(1-\chi)\sigma=J'\sigma=L(\sigma)\sigma=-\hat{H}\ ,
\end{equation*}
so that, by \eqref{eqn:EstimateOfH},
\begin{equation*}
g:=J'(1-\chi)\sigma\in C^{0,\alpha}_{-1}(N\hat{X}_{r_0})\ .
\end{equation*}
We now claim that $(1-\chi)\sigma$ is an element of $C^{2,\alpha}_{-1,0}(N\hat{X}_{r_0})$. Indeed, by Theorem \ref{proc:InvertibilityOverEndHoelder} applied to $C^{2,\alpha}_{-1,0}(N\hat{X}_{r_0})$, there exists an element $\sigma'\in C^{2,\alpha}_{-1,0}(N\hat{X}_{r_0})$ such that $J'\sigma'=g$. Since $C^{2,\alpha}_{-1,0}(N\hat{X}_{r_0})$ is contained in $C^{2,\alpha}_{0,0}(N\hat{X}_{r_0})$,
\begin{equation*}
(\sigma'-(1-\chi)\sigma)\in C^{2,\alpha}_{0,0}(N\hat{X}_{r_0})\ .
\end{equation*}
Finally, since $J'(\sigma'-(1-\chi)\sigma')=(g-g)=0$, by Theorem \ref{proc:InvertibilityOverEndHoelder} again, applied now to $C^{2,\alpha}_{0,0}(N\hat{X}_{r_0})$,
\begin{equation*}
\sigma'-(1-\chi)\sigma=0\ ,
\end{equation*}
so that
\begin{equation*}
(1-\chi)\sigma=\sigma'\in C^{2,\alpha}_{-1,0}(N\hat{X}_{r_0})\ ,
\end{equation*}
as asserted. Since $\chi$ vanishes outside a compact set, it follows that $\sigma$ is also an element of $C^{2,\alpha}_{-1}(N\hat{X}_{r_0})$, and this completes the proof.
\end{proof}

\begin{proof}[Proof of Theorem \ref{thm:DecayOfSecondFundamentalForm}]
We continue to use the notation of Lemma \ref{lemma:SigmaDecaysExponentially}. Let $\opdVol$ and $\ophatdVol$ denote respectively the volume forms of $X$ and $\hat{X}$. By \eqref{ConeMetric},
\begin{equation*}
\ophatdVol=\opSinh^{p-1}(r)\opdVol\wedge dr\ .
\end{equation*}
Let $\ophatdVol^\sigma$ denote the volume form of the embedding $\hat{e}^\sigma$. Since $\hat{X}$ is asymptotically totally-geodesic, and since $\|\sigma|_{B_1(x)}\|_{C^{2,\alpha}}$ tends to zero as $x$ tends to infinity, there exists $C_1>0$ such that
\begin{equation*}
\ophatdVol^\sigma\leq C_1\ophatdVol = C_1\opSinh^{p-1}(r)\opdVol\wedge dr\ .
\end{equation*}
By \eqref{eqn:SigmaDecaysExponentially}, there exists $C_2>0$ such that
\begin{equation*}
\|\opII\|^2 \leq C_2 e^{-2r}\ .
\end{equation*}
It follows that
\begin{equation*}
\eqalign{
\int_{\hat{X}_{r_0}}\|\opII\|^s\ophatdVol^\sigma
&\leq C_1C_2^{s/2}\int_{\hat{X}_{r_0}}\opSinh^{p-1}(r)e^{-sr}\opdVol dr\cr
&=2^{1-p}C_1C_2^{s/2}\opVol(X)\int_{r_0}^\infty e^{(p-1)r-sr}dr\cr
&<\infty\ ,\cr}
\end{equation*}
as desired.
\end{proof}

\section{Anosov representations}\label{sec:anosov}

We now discuss the applications of Theorem \ref{thm:introhomeo} to the theory of positive $\mathsf P_1$-Anosov representations in $\mathsf{PO}(p,q+1)$.
We refer the reader to \cite{DGK} for definitions and recall here only the main properties of such representations.
Let $\Gamma$ be a word-hyperbolic group with Gromov boundary homeomorphic to a $(p-1)$-sphere. Every $\mathsf{P}_1$-Anosov representation $\rho:\Gamma\rightarrow\mathsf{PO}(p,q+1)$ preserves a unique $(p-1)$-sphere $\Lambda_\rho$ in $\partial_\infty\H$, called the \emph{proximal limit set}. The proximal limit set $\Lambda_\rho$ is either positive for $\q$ or positive for $-\q$.  In the former case $\rho$ is called \emph{positive}, and in the latter \emph{negative}. Positive $\mathsf P_1$-Anosov representations are precisely those for which $\Conv(\Lambda_\rho)\cap\H$ is non-empty; and furthermore, $\rho$ acts properly discontinuously and cocompactly on    $\Conv(\Lambda_\rho)\cap\H$.

\subsection{Invariant maximal submanifolds}\label{sec:cocompact}

Let $\rho:\Gamma\rightarrow\mathsf{PO}(p,q+1)$ be a positive $\mathsf{P}_1$-Anosov representation.

\begin{proof}[Proof of Corollary \ref{cor:intro2}, Item (1)]
Let $\Lambda:=\Lambda_\rho$ denote the proximal limit set of $\rho$. By definition, $\Lambda_\rho$ is a positive $(p-1)$-sphere in $\partial_\infty\H$. In particular, it is non-negative so that, by Theorem \ref{thm:introhomeo}, there exists a unique complete maximal graph $M$ such that $\partial_\infty M=\Lambda$. Since $\rho$ preserves $\Lambda$, it also preserves $M$, and existence follows.

To prove uniqueness, let $M'$ be another complete maximal graph preserved by $\rho$, and denote $\Lambda'=\partial_\infty M'$. Since $M'$ is $\rho$-invariant, so too is $\Lambda'$, and therefore so too is $\opConv(\Lambda')$. By \cite[Theorem 1.15 and Theorem 1.24]{DGK2} (see also \cite[Fact 2.15]{BK}), $\partial_\infty\opConv(\Lambda')=\Lambda$. It follows by Lemma \ref{lemma:DeltaConvIsDelta} that $\Lambda'=\partial_\infty\opConv(\Lambda')=\Lambda$ so that, by uniqueness of solutions to the asymptotic Plateau problem, $M'=M$, as desired.

Finally, by Lemma \ref{lemma:ContainedInConvexHull}, $M$ is contained in $\Conv(\Lambda)$. Furthermore by \cite[Theorem 1.7]{DGK}, the action of $\rho(\Gamma)$ on $\Conv(\Lambda_\rho)\cap\H$ is properly discontinuous and cocompact.  Since $M$ is a closed subset of $\Conv(\Lambda_\rho)\cap\H$, it follows that the action of $\rho(\Gamma)$ on $M$ is also properly discontinuous and cocompact, and this completes the proof of Item $(1)$.
\end{proof}

\noindent Item (2) of Corollary \ref{cor:intro2} will be proven in the next section. We now prove Corollary \ref{cor:intro3}.

\begin{proof}[Proof of Corollary \ref{cor:intro3}]
Let $M$ be as in Corollary \ref{cor:intro2}. Note that $M$ is diffeomorphic to $\Rp$, and the action of $\rho(\Gamma)$ on $M$ is properly discontinuous. In particular, the stabilizer of any point is finite.  Since $\Gamma$ is torsion-free, the action of $\rho(\Gamma)$ on $M$ is free. By Item (1) of Corollary \ref{cor:intro2},  $\rho(\Gamma)$ acts cocompactly on $M$, and the quotient $M/\rho(\Gamma)$ is thus a  closed smooth manifold of dimension $p$. Finally, $\mathsf P_1$-Anosov representations have finite kernel. Since $\Gamma$ is torsion-free,  $\rho$ is faithful, $\Gamma$ is isomorphic to the fundamental group of $M/\rho(\Gamma)$, and the result follows.
\end{proof}

\subsection{Analytic dependence}\label{sec:analytic}

In order to conclude the proof of Corollary \ref{cor:intro2}, we now show that $M$ depends real analytically on the representation. This will use the implicit function theorem in a manner similar to that of the proof of Theorem \ref{thm:perturb H+}. However, since we are concerned here with the cocompact case, the proof is in fact much simpler, differing little from similar cases already treated in the literature (such as, for example, \cite{white89}). For this reason, we will only sketch the main ideas.

Let $\rho_0:\Gamma\rightarrow\mathsf{PO}(p,q+1)$ be a representation acting properly-discontinuously and cocompactly on a complete maximal $p$-submanifold $M$ of $\H$. Note that $\rho_0$ lifts to a representation taking values in $\mathsf{O}(p,q+1)$ which preserves a complete maximal $p$-submanifold of the double cover $\H_+$ which we also denote by $M$.

\begin{lemma}\label{lemma:PartitionOfUnity}
There exists a positive function $\phi\in C_0^\infty(M)$ such that, for all $x\in M$,
\begin{equation}\label{equation:PartitionOfUnity}
\sum_{\gamma\in\Gamma}\phi(\rho_0(\gamma)\cdot x) = 1\ .
\end{equation}
\end{lemma}

\begin{proof}
Indeed, let $K$ be a fundamental domain for the action of $\rho_0$ on $M$. Let $\psi\in C^\infty_0(M)$ be a positive function which does not vanish over $K$. Choose $\epsilon>0$ and let $N$ denote the number of translates of $K$ required to cover the $\epsilon$-neighbourhood of the support of this function, which is finite by  proper discontinuity of the action of $\rho_0$. For all $x\in M$, there exist at most $N$ values of $\gamma$ for which $\psi\circ\rho_0(\gamma)$ does not vanish over $B_\epsilon(x)$. We thus define the smooth function $\tilde{\psi}:M\rightarrow\mathbf{R}$ by
\begin{equation*}
\tilde{\psi}(x) := \sum_{\gamma\in\Gamma}\psi(\rho_0(\gamma)\cdot x)\ .
\end{equation*}
Clearly $\tilde\psi(\rho_0(\gamma)\cdot x)=\tilde\psi(x)$ for all $x\in M$ and $\gamma\in\Gamma$, and we readily verify that
\begin{equation*}
\phi(x) := \frac{\psi(x)}{\tilde{\psi}(x)}
\end{equation*}
has the desired properties.
\end{proof}

Recall now the quadric $\Quad$ introduced in Section \ref{subsec:pseudohyperbolicspace}. Let $e_0:M\rightarrow\H_+=\Quad\subseteq\Rpqp$ denote the canonical embedding. For all $\rho\in\opHom(\Gamma,\mathsf{O}(p,q+1))$, we define $\tilde{e}^\rho:M\rightarrow E$ by
\begin{equation}
\tilde{e}^\rho(x) := \sum_{\gamma\in\Gamma}\phi(\rho_0(\gamma)^{-1}\cdot x)\rho(\gamma)\rho_0(\gamma)^{-1}e_0(x)\ .
\end{equation}

\begin{lemma}
For all $\rho\in\opHom(\Gamma,\mathsf{O}(p,q+1))$, $\tilde{e}^\rho$ is $\rho$-equivariant.
\end{lemma}

\begin{proof}
Indeed, for all $\mu\in\Gamma$, and for all $x\in M$,
\begin{equation*}
\eqalign{
\tilde{e}^\rho(\rho_0(\mu)\cdot x) &= \sum_{\gamma\in\Gamma}\phi(\rho_0(\gamma)^{-1}\cdot\rho_0(\mu)\cdot x)\rho(\gamma)\rho_0(\gamma)^{-1}e_0(\rho_0(\mu)\cdot x)\cr
&= \sum_{\gamma\in\Gamma}\phi(\rho_0(\mu^{-1}\gamma)^{-1}\cdot x)\rho(\gamma)\rho_0(\gamma)^{-1}\rho_0(\mu)e_0(x)\cr
&= \sum_{\gamma\in\Gamma}\phi(\rho_0(\mu^{-1}\gamma)^{-1}\cdot x)\rho(\gamma)\rho_0(\mu^{-1}\gamma)^{-1}e_0(x)\cr
&= \sum_{\gamma\in\Gamma}\phi(\rho_0(\gamma)^{-1}\cdot x)\rho(\mu\gamma)\rho_0(\gamma)^{-1}e_0(x)\cr
&= \sum_{\gamma\in\Gamma}\phi(\rho_0(\gamma)^{-1}\cdot x)\rho(\mu)\rho(\gamma)\rho_0(\gamma)^{-1}e_0(x)\cr
&= \vphantom{\frac{1}{2}} \rho(\mu)\tilde{e}^\rho(x)\ ,\cr}
\end{equation*}
as desired.
\end{proof}

\noindent For all $\rho$, let $e^\rho$ denote the composition of $\tilde{e}^\rho$ with the canonical projection onto the projective space $\P_+(E)$ of oriented lines in $E$. Recall now that $\H_+$ is an open subset of $\P_+(E)$.

\begin{lemma}\label{lemma:NbdOfEquivariance}
There exists a neighbourhood $U$ of $\rho_0$ in $\opHom(\Gamma,\mathsf{O}(p,q+1))$ such that, for all $\rho\in U$, $e^\rho$ defines a spacelike embedding from $M$ into $\H_+$.
\end{lemma}

\begin{proof} Indeed, denote
\begin{equation*}
\Omega := \{ x\in\Rpqp\ |\ \q(x,x) < 0\}\ .
\end{equation*}
Let $K\subseteq M$ be a fundamental domain of $\rho_0$. Define the neighbourhood $U_1$ of $\rho_0$ in $\opHom(\Gamma,\mathsf{O}(p,q+1))$ by
\begin{equation*}
U_1 := \{\rho\ |\ \tilde{e}^\rho(x)\in\Omega\ \forall x\in K\}\ .
\end{equation*}
By $\rho$-equivariance of $\tilde{e}^\rho$ and $\mathsf{O}(p,q+1)$-invariance of $\Omega$, for all $\rho\in U_1$,
\begin{equation*}
\opIm(\tilde{e}^\rho)\subseteq \Omega\ ,
\end{equation*}
so that
\begin{equation*}
\opIm(e^\rho)\subseteq\H_+\ .
\end{equation*}
By compactness and equivariance again, there exists a neighbourhood $\rho_0\in U_2\subseteq U_1$ such that, for all $\rho\in U_2$, $e^\rho$ is a spacelike embedding, and this completes the proof.
\end{proof}

\noindent We now sketch the proof of the analytic dependence.

\begin{proof}[Sketch of proof of Corollary \ref{cor:intro2}, Item (2).]
Let $NM$ denote the normal bundle of $M$ and, for all $(k,\alpha)$, let $C^{k,\alpha}_{\opequiv}(NM)$ denote the space of $\rho_0$-equivariant $C^{k,\alpha}$ sections of this bundle. Let $U$ be as in Lemma \ref{lemma:NbdOfEquivariance}. For all $\rho\in U$, and for all $x\in M$, let $q^\rho(x)$ denote the composition of parallel transport along the geodesic from $e_0(x)$ to $e^\rho(x)$ with orthogonal projection onto the normal bundle of $e^\rho(M)$, and note that $q^\rho$ is a bundle isomorphism. For every section $\sigma$ of the normal bundle of $M$, we define $e^{\rho,\sigma}$ by
\begin{equation*}
e^{\rho,\sigma}(x) := \opExp_{e^\rho(x)}(q^\rho(x)\sigma(x))\ ,
\end{equation*}
where $\opExp$ here denotes the exponential map of $\H_+$. By cocompactness and equivariance, upon reducing $U$ if necessary, there exists a neighbourhood $\mathcal{V}^{2,\alpha}$ of the zero section in $C^{2,\alpha}_\opequiv(NM)$ such that, for all $(\rho,\sigma)\in U\times\mathcal{V}^{2,\alpha}$, $e^{\rho,\sigma}$ is a spacelike embedding from $M$ into $\H_+$. For all such $(\rho,\sigma)$, let $q^{\rho,\sigma}(x)$ denote the composition of parallel transport from $e^\rho(x)$ to $e^{\rho,\sigma}(x)$ with orthogonal projection onto the normal bundle of $e^{\rho,\sigma}(M)$, and note that $q^{\rho,\sigma}$ is also a bundle homeomorphism. For all $(\rho,\sigma)\in U\times\mathcal{V}^{2,\alpha}$ let $H^{\rho,\sigma}$ denote the image under $(q^{\rho})^{-1}\circ(q^{\rho,\sigma})^{-1}$ of the mean curvature vector of $e^{\rho,\sigma}$.

We define the functional $\mathcal{H}:U\times\mathcal{V}^{2,\alpha}\rightarrow C^{0,\alpha}_\opequiv(NM)$ by
\begin{equation*}
\mathcal{H}(\rho,\sigma) := H^{\rho,\sigma}\ .
\end{equation*}
This functional is smooth, and its partial derivative with respect to the second component is the Jacobi operator $J$ of $M$. Since $J$ is a generalised laplacian over a cocompact manifold, it is Fredholm of index zero. Finally, as in the proof of Theorem \ref{proc:InvertibilityOverEndSobolev}, $J$ has trivial kernel, and is thus a linear isomorphism, and the result follows by the implicit function theorem.
\end{proof}

\subsection{The Guichard--Wienhard domain}

In this section, we prove Theorem \ref{theorem:GeometricStructures} following the ideas of \cite[Section 4]{CTT}. We first recall the construction described by Guichard--Wienhard in \cite{GW}. Let $\q$ be a signature $(p,q+1)$ quadratic form on some vector space $E$, and let $\Isot(E)$ denote the space of maximally isotropic subspaces of $E$. Observe that every element $V$ in $\Isot(E)$ has dimension $\min\{p,q+1\}$ and that $\Isot(E)$ is a $\PO$-homogenous space.

\begin{theorem}[\cite{GW}]
Let $\Gamma$ be a word hyperbolic group, let $\rho$ be a $\mathsf P_1$-Anosov representation of $\Gamma$ in $\SO$, let $\Lambda\subset\bH$ denote the proximal limit set of $\rho$, and denote
\begin{equation}
\Omega_\rho := \left\{ V\in \Isot(E)~,~\P(V)\cap \Lambda= \emptyset \right\}~.
\end{equation}
The action of $\rho(\Gamma)$ on $\Omega_\rho$ is properly discontinuous and cocompact.
\end{theorem}

\noindent We first provide an elementary description of the intersections of maximal totally isotropic subspaces with $\bH$.

\begin{lemma}\label{lemma structure totally isotropic}
Let $E=U\oplus W$ be an orthogonal decomposition such that $\q$ restricts to a positive-definite form on $U$ and a negative-definite form on $W$. Endow $U$ and $W$ with the scalar products $\langle\cdot,\cdot\rangle_U=\q|_U$ and $\langle\cdot,\cdot\rangle_W=-\q|_W$. Let $V\in \Isot(E)$.
\begin{enumerate}
\item If $p\leq q+1$, then $V$ is the graph of a linear map $\varphi:U\to W$ which is an isometry onto its image. In any Fermi chart $\P_+(V)\subset\bH_+$ is the graph of an isometric immersion $\psi:\S^{p-1}\to \S^q$.
\item If $p\geq q+1$, then $V$ is the graph of a linear map $\varphi:W\to U$ which is an isometry onto its image. In any Fermi chart $\P_+(V)\subset\bH_+$ is the graph of an isometry $\psi:\Sigma^q\to \S^q$ for some totally-geodesic sphere $\Sigma^q\subseteq \S^{p-1}$.
\end{enumerate}
\end{lemma}

\begin{proof}
For any two vectors, $v_1=(u_1,w_1)$ and $v_2=(u_2,w_2)$,
\begin{equation*}
\q (v_1,v_2)=\langle u_1,u_2\rangle_U-\langle w_1,w_2\rangle_W\ ,
\end{equation*}
so that $V$ is the graph of some isometric linear map $\varphi$ from $U$ to $W$ or from $W$ to $U$, depending on the sign of $(p-q-1)$. Taking the intersections with $\bH$, we see that, if $p\leq q+1$, then $\P_+(V)\subset\bH_+$ is the graph of $\psi=\varphi|_{\S^p}$, and if $p\geq q+1$, then $\P_+(V)\subset\bH_+$ is the graph of the inverse of $\varphi$ restricted to $\Sigma^q=\varphi(W)\cap \S^p$.
\end{proof}

\begin{lemma}\label{lemma GW p>q}
Let $\Lambda$ be a positive sphere in $\bH$ and let $V\in\Isot(E)$ be a maximal totally isotropic subspace. If $p\geq q+1$, then $\P(V)$ intersects $\Lambda$.
\end{lemma}

\begin{proof}
Let $\Lambda_+$ be a lift of $\Lambda$ in $\bH_+$. We work in a Fermi chart, so that $\bH_+\cong\S^{p-1}\times \S^q$. In particular, $\Lambda_+$ is the graph of a $1$-Lipschitz map $f:\S^{p-1}\rightarrow\S^q$. By Lemma \ref{lemma structure totally isotropic}, there exists a $q$-dimensional totally geodesic sphere $\Sigma^q$ in $\S^{p-1}$ together with an isometry $\psi: \Sigma^q \to \S^q$ such that $\P_+(V)\cap\bH_+$ is the graph of $\psi$.

Consider the map $g:=\psi^{-1}\circ f_{\vert \Sigma^q}$ from $\Sigma^q$ to itself. Since $\Lambda$ is positive, the image of $g$ does not contain antipodal points so that, by Lemma \ref{cor:deformation to constant}, it is homotopic to the constant map and thus has degree $0$. It then follows by the Lefschetz fixed-point theorem that $g$ has at least one fixed point. The graphs of $\psi$ and $f$ thus intersect so that $\P(V)$ and $\Lambda$ also intersect, as desired.
\end{proof}

This proves the second item in Corollary \ref{theorem:GeometricStructures}. Before completing the proof, we require two more technical results.

\begin{lemma}\label{lemma:dichotomy}
Suppose that $p\leq q$, and let $M$ be a smooth entire graph in $\H$ with asymptotic boundary $\Lambda$. For all $V\in \Isot(E)$, either
\begin{enumerate}
	\item there is a point $x$ in $M$, which is unique, such that $V$ is contained in $x^\bot$, or
	\item $\P(V)$ intersects $\Lambda$.
\end{enumerate}
\end{lemma}

\begin{proof}
We first show that $\P(V^\bot)$ meets $M\cup\Lambda$. Indeed, since $p\leq q$, the isotropic space $V$ has dimension $p$ and its orthogonal complement $V^\bot$ is a degenerate $(q+1)$-dimensional subspace of $E$ which does not contain any positive-definite line and on which the kernel of $\q$ is exactly $V$. In particular, $V^\bot$ is the limit of some sequence $\seqn{W_n}$ of negative-definite $(q+1)$-dimensional subspaces in $E$. For every $n$ we choose a Fermi chart $\Psi_n:\Upp^p\times \S^q\to\H_+$ such that $\Psi_n(\{N\}\times \S^q)=\P_+(W_n)$. Since any lift $M_+$ of $M$ is an entire graph in every such chart, $M_+$ intersects each $\P_+(W_n)$ in a unique point, and it follows that $M$ intersects each $\P(W_n)$ in a unique point. By compactness of $M\cup\Lambda$, $\P(V^\bot)$ intersects $M\cup\Lambda$, as asserted.

Observe now that $\P(V^\bot)$ intersects $M$ at a point $x$ if and only if $V$ is contained in $x^\bot$. Note, furthermore, that for any two points $\hat x,\hat y$ in $V^\bot$ satisfying $\q(\hat x,\hat x)=\q(\hat y,\hat y)=-1$,
\begin{equation*}
\q (\hat x,\hat y)\geq -1\ ,
\end{equation*}
and uniqueness of $x$ follows by Lemma \ref{lemma:Acausal0}. Finally, if $\P(V^\bot)$ intersects $\Lambda$, then, since the kernel of $\q_{\vert V^\bot}$ is exactly $V$, it follows that $\P(V)$ intersects $\Lambda$, and this completes the proof.
\end{proof}

Recall that the \emph{Stiefel manifold} $\V_{k,n}$ is the space of $k$-tuples of unit vectors in $\R^n$ that are pairwise orthogonal.

\begin{lemma}\label{lemma:topologyB}
Suppose that $p\leq q$, let $M$ be a smooth entire graph in $\H$, and let
\begin{equation}
 B(M):=\{(x,V)\,|\,x\in M,V\in \Isot(x^\bot)\}
\end{equation}
denote the bundle over $M$ whose fiber over $x$ is $\Isot(x^\bot)$. Then:
\begin{enumerate}
\item The fibers of $B(M)$ are diffeomorphic to $\V_{p,q}$.
\item If $\rho:\Gamma\to\PO$ is a representation acting freely and properly discontinuously on $M$, then the quotient $B(M)/\rho(\Gamma)$ is connected unless $p=q$ and $\N M/\rho(\Gamma)$ has vanishing first Stiefel-Whitney class.
\end{enumerate}
\end{lemma}

\begin{proof}
To prove $(1)$, consider an orthogonal splitting $x^\bot = U \oplus V$ such that $\q$ restricts to a positive-definite form on $U$ and a negative-definite form on $V$. By Lemma \ref{lemma structure totally isotropic}, elements in $\Isot(x^\bot)$ are graphs of linear maps $\varphi:U\to V$ satisfying $\langle \varphi(a),\varphi(b)\rangle_V=  \langle a,b\rangle_U$ for any $a,b\in U$. Fixing an orthonormal basis $(e_1,...,e_p)$ of $U$, such a map is given by an orthonormal $p$-tuple of vectors in $V$, so that $\Isot(x^\bot)$ is diffeomorphic to $\V_{p,q}$, as desired.

To prove $(2)$, observe that if $p<q$ then $\V_{p,q}$ is connected, and therefore so too is $B(M)/\Gamma$. When $p=q$, the fiber $B(M)_x$ has two connected components, and so the quotient $B(M)/\Gamma$ either has one or two connected components. Fix an orientation of $M$. Given $x\in M$, and identifying $\T_xM$ and $\N_xM$ respectively with $U$ and $V$, we see that the two connected components of the fiber correspond to the two orientations of $\N_x M$. It follows that the quotient $B(M)/\Gamma$ has two connected components if and only if $\N M/\rho(\Gamma)$ carries a global orientation, that is, if and only if the $\mathsf O(q)$-bundle $\N M/\rho(\Gamma)$ reduces to a $\mathsf{SO}(q)$-bundle. Since the obstruction to this is precisely the first Stiefel-Whitney class, this proves $(2)$.
\end{proof}

We now conclude the proof of Corollary \ref{theorem:GeometricStructures}.

\begin{proof}[Proof of Corollary \ref{theorem:GeometricStructures}]
Let $B(M)$ be the bundle constructed in Lemma \ref{lemma:topologyB}, where $M$ is the complete maximal $\rho(\Gamma)$-invariant $p$-submanifold given by Corollary \ref{cor:intro2}.

Assume $q\geq p$. Since elements in $\Isot(x^\bot)$ have dimension $p$, the inclusion of $x^\bot$ into $E$ induces an embedding of $\Isot(x^\bot)$ into $\Isot(E)$. Taking this embedding fiberwise yields a map $\delta: B(M) \to \Isot(E)$. We claim that $\delta$ is a homeomorphism onto its image, and $\delta(B(M))=\Omega_\rho$. Indeed, the homogeneous space $\Isot(E)$ is diffeomorphic to $\V_{p,q+1}$, so has dimension  $\dim(\V_{p,q})+q$ which is thus equal to the dimension of $B(M)$. By Lemma \ref{lemma:dichotomy}, the map $\delta$ is a bijection onto the open set $\Omega_\rho$. Since $\delta$ is continuous it is a homeomorphism by invariance of domain theorem.

Now, $\delta$ is clearly equivariant for the actions of $\rho(\Gamma)$ on $B(M)$ and $\Omega_\rho$. The conclusions of Theorem \ref{theorem:GeometricStructures} then follow by Lemma \ref{lemma GW p>q} and Lemma \ref{lemma:topologyB}, which can be applied since the action of $\rho(\Gamma)$ on $M$ is properly discontinuous by Corollary \ref{cor:intro2}, and is free by the assumption that $\Gamma$ is torsion-free.
\end{proof}

\section{Renormalized area and minimal lagrangian extensions}\label{sec:RenormalizedArea}

We conclude this paper with a discussion of the applications of Theorem \ref{thm:DecayOfSecondFundamentalForm}. Note first that Corollary \ref{cor:renarea} is an immediate consequence of Theorem \ref{thm:DecayOfSecondFundamentalForm} for $p=2$, together with the definition of renormalized area given in \eqref{eqn:RenormalizedArea}. In this section, we prove Corollary \ref{cor:minlag}.

\subsection{The $\mathrm{PSL}(2,\R)$ model of anti-de Sitter space}

We first recall the $\mathrm{PSL}(2,\R)$ model of $\Ads$ (see \cite{mess} or \cite[Section 3]{surveyandreafra} for more details). Following the notation of Section \ref{subsec:pseudohyperbolicspace}, we take $E=\mathcal M(2,\R)$ to be the vector space of $2\times2$ matrices, and consider the following bilinear form of signature $(2,2)$
\begin{equation}\label{eq:bilformads}
\q(A_1,A_2)=-\frac{1}{2}\tr(A_1\cdot\mathrm{adj(A_2)})~,
\end{equation}
where $\mathrm{adj}$ here denotes the adjugate matrix. Since $A\cdot\mathrm{adj(A)}=\det(A)\mathrm{Id}$, the quadratic form induced by $\q$ is $\q(A,A)=-\det(A)$. With this choice of $V$ and $\q$, $\Ads_+$ identifies with $\mathrm{SL}(2,\R)$, and $\Ads$ identifies with $\mathrm{PSL}(2,\R)$.

In this model, $\bAds$ thus identifies with the space of projective classes of $2\times2$ matrices of rank $1$. Although we already know from Section \ref{subsection:FermiCoordinates} that $\bAds_+$ is homeomorphic to $\S^1\times \S^1$, and thus $\bAds$, being the quotient of $\bAds_+$ by the antipodal map,  is also in this particular dimension homeomorphic to $\S^1\times \S^1$, we will need to use here a different parameterization for $\bAds$. This is given by the map
\begin{equation*}
\Xi:\bAds\to\P(\R^2)\times\P(\R^2),
\end{equation*}
defined by
\begin{equation}\label{defi Xi}
\Xi([A])=(\mathrm{Im}(A),\Ker(A)),
\end{equation}
where, for any rank $1$ matrix $A$, $[A]$ denotes its projective class.

Finally, since multiplication on the left and on the right by matrices of unit determinant preserves $\q$, there is a monomorphism of $\mathrm{PSL}(2,\R)\times \mathrm{PSL}(2,\R)$ into the group $\PO$ of isometries of $\Ads$. For dimensional reasons, it turns out that its image is the identity component of $\PO$. The parameterization $\Xi^{-1}$ is clearly equivariant with respect to the monomorphism $\mathrm{PSL}(2,\R)\times \mathrm{PSL}(2,\R)\to\PO$, where $\mathrm{PSL}(2,\R)\times \mathrm{PSL}(2,\R)$ acts on $\P(\R^2)\times\P(\R^2)$ by the obvious product action.

\begin{lemma}
Given any orientation-preserving circle homeomorphism $f:\P(\R^2)\to\P(\R^2)$, the set $\Xi^{-1}(\mathrm{graph}(f))$ is a positive 1-sphere in $\bAds$.
\end{lemma}
\begin{proof}
Let $x,y,z$ be a triple of pairwise distinct points in $\bAds$. By definition, we need to show that $x\oplus y\oplus z$ has signature $(2,1)$. Denote $\Xi(x)=(x_1,x_2)$, $\Xi(y)=(y_1,y_2)$ and $\Xi(z)=(z_1,z_2)$. Observe that, since $\Xi(x)$, $\Xi(y)$ and $\Xi(z)$ are on the graph of $f$ and $x,y,z$ are pairwise distinct, $x_1,y_1$ and $z_1$ are pairwise distinct. Likewise, since $f$ is injective, $x_2,y_2$ and $z_2$ are pairwise distinct.  Now, since $\mathrm{PSL}(2,\R)$ acts transitively on oriented triples in $\P(\R^2)$, using the action of $\mathrm{PSL}(2,\R)\times \mathrm{PSL}(2,\R)$ we can assume that
\begin{equation*}
x_1=x_2=\begin{bmatrix} 0 \\ 1 \end{bmatrix}\qquad y_1=y_2=\begin{bmatrix} 1 \\ 1 \end{bmatrix} \qquad y_1=y_2=\begin{bmatrix} 1 \\ 0 \end{bmatrix}~.
\end{equation*}
From the definition of $\Xi$ in \eqref{defi Xi}, it follows immediately that
\begin{equation*}
x=\begin{bmatrix} 0 & 0 \\ 1 & 0 \end{bmatrix}\qquad y=\begin{bmatrix} 1 & -1 \\ 1 & -1 \end{bmatrix}\qquad z=\begin{bmatrix} 0 & 1 \\ 0 & 0 \end{bmatrix}~.
\end{equation*}
The span of $x$, $y$ and $z$ is the three-dimensional subspace of traceless matrices in $\mathcal M(2,\R)$, which, by \eqref{eq:bilformads}, is the $\q$-orthogonal complement of the identity matrix. Since $\q(\mathrm{id},\mathrm{id})<0$, the space $x\oplus y\oplus z$ has signature $(2,1)$, as desired.
\end{proof}

\subsection{Minimal lagrangian extensions}

Let us now outline the construction, first developed in \cite{bonschl}, of minimal Lagrangian diffeomorphisms of the hyperbolic plane from maximal surfaces in $\Ads$.

\begin{definition}
A diffeomorphism $F:\Htwo\to\Htwo$ is minimal Lagrangian if it is area-preserving and its graph is a minimal surface in $\Htwo\times\Htwo$.
\end{definition}

Given a maximal surface in $\AdS$, or more generally any spacelike surface, one can define two \emph{Gauss maps} $\pi_1,\pi_2:M\to\Htwo$ as follows. Given any point $p=[A_p]\in \AdS$, which we think as the projective class of a matrix $A_p\in \mathcal M(2,\R)$ with $\det(A)=1$, let $\nu(p)$ be a unit normal vector to $M$ at $p$. Since $M$ is spacelike, $\nu(p)$ is represented by an element $N_p$ of $\mathcal M(2,\R)$, well-defined up to a sign, satisfying $\q(A_p,N_p)=0$ and $\q(N_p,N_p)=-1$. Now, since $\q$ is preserved by multiplication by matrices of unit determinant, $\q(\mathrm{id},A_p^{-1}N_p)=0$ and $\q(A_p^{-1}N_p,A_p^{-1}N_p)=-1$. Likewise $\q(\mathrm{id},N_pA_p^{-1})=0$ and $\q(N_pA_p^{-1},N_pA_p^{-1})=-1$. By \eqref{eq:bilformads}, this means that  $A_p^{-1}N_p$ and $N_pA_p^{-1}$ are traceless matrices with determinant 1, and, by the Cayley-Hamilton theorem, their projective classes in $\mathrm{PSL}(2,\R)$ have order two. We define
\begin{equation*}
\pi_1(p)=\mathrm{Fix}([N_pA_p^{-1}])\qquad \text{and}\qquad \pi_2(p)=\mathrm{Fix}([A_p^{-1}N_p]),
\end{equation*}
where $\mathrm{Fix}$ denotes the unique fixed point of the action of $\Htwo$, in the upper half-plane model, of an elliptic element of $\mathrm{PSL}(2,\R)$.

The key point for the construction of minimal Lagrangian maps is then the following identity for the pull-backs of the hyperbolic metric of $\Htwo$ via the two Gauss maps $\pi_1$ and $\pi_2$:
\begin{equation}\label{eq:pullback formula}
(\pi_1)^*\g_{\Htwo}=\I((\mathrm{id}-JB)\cdot,(\mathrm{id}-JB)\cdot)\qquad \text{and}\qquad (\pi_2)^*\g_{\Htwo}=\I((\mathrm{id}+JB)\cdot,(\mathrm{id}+JB)\cdot)~,
\end{equation}
where as usual $\I$ denotes the first fundamental form of $M$ and $B$ its shape operator, and $J$ denotes the almost-complex structure on $M$ associated to $\I$. Here $B$ is considered as a smooth section of the bundle of endomorphisms of $\T M$, and the same holds for $J$.
For a proof of \eqref{eq:pullback formula}, see \cite[Lemma 3.16]{krasschl}, \cite[Section 6.2]{barbotkleinian} or  \cite[Proposition 6.3.7]{surveyandreafra}.

Now, recall that an orientation-preserving circle homeomorphism $f:\P(\R^2)\to\P(\R^2)$ is \emph{quasisymmetric} if it admits a quasiconformal extension to $\Htwo$. The key properties required to construct minimal Lagrangian extensions of quasisymmetric circle homeomorphisms are summarized in the following lemma, proved in \cite{bonschl}.

\begin{lemma}\label{lemma extension}
Let $f:\P(\R^2)\to\P(\R^2)$ be any quasisymmetric circle homeomorphism, and let $M$ be the  complete maximal surface in $\AdS$ with $\partial_\infty M=\Xi^{-1}(\mathrm{graph}(f))$.
\begin{enumerate}
\item For $i=1,2$, $\pi_i:M\to\Htwo$ is a diffeomorphism;
\item $F:=\pi_2\circ \pi_1^{-1}$ is a quasiconformal minimal Lagrangian diffeomorphism of $\Htwo$ whose continuous extension to $\partial_\infty\Htwo\cong \P(\R^2)$ equals $f$.
\end{enumerate}
\end{lemma}

Since Lemma \ref{lemma extension} is well-known, we only outline here the proof and the fundamental references, together with some observations which will be important for the proof of Corollary \ref{cor:minlag}.

Let $B$ denote the shape operator of $M$, as above, and let $\pm\lambda$ be the eigenvalues of $B$. By Ishihara's bound on the norm of the second fundamental form (Theorem \ref{theorem:Ishihara}) of a complete maximal surface, $\lambda^2\leq 1$. Moreover, by \cite[Theorem 1.12, Proposition 5.2]{bonschl}, if $\partial_\infty M=\Xi^{-1}(\mathrm{graph}(f))$ for $f$ a quasisymmetric circle homeomorphism, then $\lambda^2\leq 1-\epsilon$ for some $\epsilon>0$.

Now, since $B$ is a symmetric and traceless endomorphism of $TM$, so too is $JB$. A direct computation shows that
\begin{equation}\label{eq:jacobian}
\det(\mathrm{id}\pm JB)=1+\det(B)=1-\lambda^2\in (\epsilon,1]~.
\end{equation}
The identity \eqref{eq:jacobian} has several consequences. First, it implies that $\pi_1$ and $\pi_2$ are local diffeomorphisms. By \cite[Proposition 3.17]{bonschl}, $\pi_l$ (resp. $\pi_r$) extends continuously to the maps $\partial_\infty M\to \partial_\infty\Htwo\cong \P(\R^2)$ sending $\Xi^{-1}(x_1,x_2)$ to $x_1$ (resp. $x_2$). As a consequence, $\pi_l$ and $\pi_r$ are global diffeomorphisms, and so too is $F:=\pi_2\circ \pi_1^{-1}$. Moreover $F$ extends continuously to $f$.

Second, by \eqref{eq:jacobian} and the easy observation that $\mathrm{tr}(\mathrm{id}\pm JB)=2$, the eigenvalues of $\mathrm{id}\pm JB$ are bounded above and below by positive constants. Hence the quasiconformal dilatation of $\pi_l$ and $\pi_r$, which equals the ratio between these eigenvalues, is bounded. This shows that $\pi_l$ and $\pi_r$ are quasiconformal, and so too is $F$.

Third, \eqref{eq:jacobian} also shows that the Jacobian determinants of $\pi_1$ and $\pi_2$ coincide. Hence $F$ is area-preserving.
To complete the proof of Lemma \ref{lemma extension}, it thus only remains to show that the graph of $F$ is minimal in $\Htwo\times\Htwo$. This follows easily from  \eqref{eq:pullback formula}, which can be rewritten as:
\begin{equation*}
(\pi_i)^*\g_{\Htwo}=(1+\det(B))\I+(-1)^i\I(2JB\cdot,\cdot)~.
\end{equation*}
Since $JB$ is symmetric and traceless, this shows that the Hopf differential $\mathrm{Hopf}(\pi_i)$ is the unique quadratic differential on $M$ whose real part equals $(-1)^i\I(2JB\cdot,\cdot)$. By the Codazzi equation, $\mathrm{Hopf}(\pi_i)$ is holomorphic. Hence $\pi_i$ is a harmonic diffeomorphism. Moreover $\mathrm{Hopf}(\pi_1)=-\mathrm{Hopf}(\pi_2)$, showing that $(\pi_1,\pi_2):M\to\Htwo\times\Htwo$ is a conformal harmonic embedding, hence its image is a minimal surface, and therefore   the graph of $F$ is minimal.

\subsection{Proof of Corollary \ref{cor:minlag}}

We are now ready to provide the proof of Corollary \ref{cor:minlag}.

\begin{proof}[Proof of Corollary \ref{cor:minlag}]
Let $f$ be a $C^{3,\alpha}$ circle diffeomorphism. In particular, $f$ is quasisymmetric (\cite[Chapter 16]{zbMATH01399773}). Recall that the modulus of the Beltrami differential of a quasiconformal diffeomorphism $F:\Htwo\to\Htwo$ satisfies the identity
\begin{equation}\label{eq:quasiconformal1}
|\mu(p)|=\frac{K(p)-1}{K(p)+1}~,
\end{equation}
where $K(p)$ is the quasiconformal dilatation of $F$ at $p$ (\cite[Chapter 1]{zbMATH01399773}). Moreover, when $F=\pi_2\circ \pi_1^{-1}$ is the quasiconformal minimal Lagrangian extension of $f$ constructed in \cite{bonschl} (c.f. Lemma \ref{eq:jacobian}), a direct computation based on \eqref{eq:pullback formula} (see \cite[Proposition 5.5]{andreaJEMS}) shows that
\begin{equation}\label{eq:quasiconformal2}
K(p)=\left(\frac{1+\lambda(\pi_1^{-1}(p))}{1-\lambda(\pi_1^{-1}(p))}\right)^2~,
\end{equation}
where as usual $\pm\lambda$ are the principal curvatures of the complete maximal surface $M$. From \eqref{eq:quasiconformal1} and \eqref{eq:quasiconformal2}, and recalling that, if $\opII$ denotes the second fundamental form of $M$, then $\|\opII\|^2=2\lambda^2$, we obtain
\begin{equation}\label{eq:quasiconformal3}
|\mu(p)|=\frac{2|\lambda\circ \pi_1^{-1}|}{1+(\lambda\circ \pi_1^{-1})^2}\leq C(\|\opII\|\circ\pi_1^{-1})~.
\end{equation}
Finally, using \eqref{eq:pullback formula} and \eqref{eq:jacobian}, we see that $\pi_1^*\mathrm{dArea}_{\Htwo}\leq \mathrm{dArea}_{M}$. Therefore, by Corollary \ref{cor:renarea}, we obtain
\begin{equation*}
\int_{\Htwo}|\mu|^2\mathrm{dArea}_{\Htwo}\leq C\int_M\|\opII\|^2\mathrm{dArea}_{M}<+\infty\ ,
\end{equation*}
and this completes the proof.
\end{proof}

\bibliographystyle{alpha}
\bibliographystyle{ieeetr}
\bibliography{biblioSST.bib}
\end{document}